\documentclass[11pt]{article}
\usepackage[utf8]{inputenc} 
\usepackage[T1]{fontenc}    
\usepackage{amsmath, amsfonts, amsthm, graphicx, microtype, url, xcolor, mathtools}
\usepackage[colorlinks=true,linkcolor=blue,urlcolor=black,citecolor=blue,breaklinks]{hyperref}
\usepackage{multirow}
\usepackage{varwidth}
\usepackage{subcaption}
\usepackage{overpic}

\usepackage[numbers, sort&compress]{natbib}
\bibpunct{(}{)}{;}{a}{,}{,}

\usepackage{fullpage}
\usepackage{authblk}

\usepackage{amsmath,amsthm,amssymb,colonequals,etoolbox}
\usepackage{thmtools}

\title{Implicit regularization and momentum algorithms in nonlinearly parameterized adaptive control and prediction}

\author[1]{Nicholas M.\ Boffi}
\author[2]{Jean-Jacques E.\ Slotine}

\affil[1]{John A.\ Paulson School of Engineering and Applied Sciences, Harvard University}
\affil[2]{Nonlinear Systems Laboratory, Massachusetts Institute of Technology}

\newcommand{\T}{\mathsf{T}}
\newcommand{\HH}{\mathcal{H}}
\newcommand{\hp}{\hat{\mathbf{p}}}
\newcommand{\hq}{\hat{\mathbf{q}}}
\newcommand{\be}{\mathbf{e}}
\newcommand{\sN}{\mathcal{N}}
\newcommand{\bs}{\mathbf{s}}
\newcommand{\bc}{\mathbf{c}}
\newcommand{\p}{\partial}
\newcommand{\bw}{\mathbf{w}}
\newcommand{\bphi}{\boldsymbol\phi}
\newcommand{\bGam}{\boldsymbol\Gamma}
\newcommand{\LL}{\mathcal{L}}

\newcommand{\brho}{\boldsymbol\rho}

\newcommand{\bg}{\mathbf{g}}

\newcommand{\bP}{\mathbf{P}}
\newcommand{\bp}{\mathbf{p}}
\newcommand{\bJ}{\mathbf{J}}

\newcommand{\bx}{\mathbf{x}}

\newcommand{\bsig}{\boldsymbol{\sigma}}

\newcommand{\bz}{\mathbf{z}}
\newcommand{\ba}{\mathbf{a}}
\newcommand{\bA}{\mathbf{A}}

\newcommand{\bu}{\mathbf{u}}

\newcommand{\bq}{\mathbf{q}}
\newcommand{\bM}{\mathbf{M}}
\newcommand{\bb}{\mathbf{b}}
\newcommand{\bv}{\mathbf{v}}

\newcommand{\bThet}{\boldsymbol\Theta}
\newcommand{\bthet}{\boldsymbol\theta}

\newcommand{\bH}{\mathbf{H}}
\newcommand{\bC}{\mathbf{C}}
\newcommand{\bV}{\mathbf{V}}

\newcommand{\bY}{\mathbf{Y}}
\newcommand{\bI}{\mathbf{I}}

\newcommand{\by}{\mathbf{y}}
\newcommand{\bQ}{\mathbf{Q}}
\newcommand{\bff}{\mathbf{f}}

\newcommand{\balf}{\boldsymbol{\alpha}}
\newcommand{\bxi}{\boldsymbol{\xi}}
\newcommand{\ha}{\hat{\ba}}
\newcommand{\hv}{\hat{\bv}}

\newcommand{\sign}{\text{sign}}
\newcommand{\qeddef}{\ \ \diamond}

\newcommand{\bK}{\mathbf{K}}
\newcommand{\norm}[1]{\Vert#1\Vert}

\newcommand{\bregd}[2]{d_\psi\left(#1\ \middle\|\ #2\right)}
\newcommand{\bregdg}[3]{d_#1\left(#2\ \middle\|\ #3\right)}
\newcommand{\bregdphi}[2]{d_\phi\left(#1\ \middle\|\ #2\right)}

\newtheorem{prop}{Proposition}[section]

\newtheorem{lem}{Lemma}[section]

\theoremstyle{definition}
\newtheorem{rmk}{Remark}[section]

\theoremstyle{definition}

\theoremstyle{definition}

\theoremstyle{definition}
\newtheorem{assumption}{Assumption}[section]

\begin{document}
\maketitle

\begin{abstract}
Stable concurrent learning and control of dynamical systems is the subject of adaptive control. Despite being an established field with many practical applications and a rich theory, much of the development in adaptive control for nonlinear systems revolves around a few key algorithms. By exploiting strong connections between classical adaptive nonlinear control techniques and recent progress in optimization and machine learning, we show that there exists considerable untapped potential in algorithm development for both adaptive nonlinear control and adaptive dynamics prediction. We begin by introducing first-order adaptation laws inspired by natural gradient descent and mirror descent. We prove that when there are multiple dynamics consistent with the data, these non-Euclidean adaptation laws implicitly regularize the learned model. Local geometry imposed during learning thus may be used to select parameter vectors -- out of the many that will achieve perfect tracking or prediction -- for desired properties such as sparsity. We apply this result to regularized dynamics predictor and observer design, and as concrete examples, we consider Hamiltonian systems, Lagrangian systems, and recurrent neural networks. We subsequently develop a variational formalism based on the Bregman Lagrangian. We show that its Euler Lagrange equations lead to natural gradient and mirror descent-like adaptation laws with momentum, and we recover their first-order analogues in the infinite friction limit. We illustrate our analyses with simulations demonstrating our theoretical results.
\end{abstract}

\section{Introduction}

Adaptation is an online learning problem concerned with control or prediction of the
dynamics of an unknown nonlinear system. This task is accomplished by constructing an approximation to the true dynamics through the online adjustment of a vector of parameter estimates under the assumption that there exists a fixed vector of parameters that globally fits the dynamics. The overarching goal is provably safe, stable, and concurrent learning and control of nonlinear dynamical systems.

Adaptive control theory is a mature field, and many results exist tailored to specific system structures~\citep{ioannou, annaswammy_book,slot_li_book}. An adaptive control algorithm typically consists of a parameter estimator coupled in feedback to the controlled system, and the parameter estimator is often strongly inspired by gradient-based optimization algorithms. A significant difference between standard optimization algorithms and adaptive control algorithms is that the parameter estimator must not only converge to a set of parameters that leads to perfect tracking of the desired trajectory, but the system must remain stable throughout adaptation. The additional requirement of stability prevents the immediate application of optimization algorithms as adaptive control algorithms, and stability must be proved by jointly analyzing the closed-loop system and estimator.

Significant progress has been made in adaptive control even for nonlinear systems in the \textit{linearly parameterized} setting, where the dynamics approximation is of the form $\bY(\bx, t)\hat{\ba}$ for some known regressor matrix $\bY(\bx, t)$ and vector of parameter estimates $\ha(t)$. Well-known examples include the adaptive robot trajectory controller of~\citet{slot_li_robot} and the neural network-based controller of~\citet{sanner}, which employs a mathematical expansion in physical nonlinear basis functions to uniformly approximate the unknown dynamics. 

Unlike its linear counterpart, solutions to the adaptive control problem in the general nonlinearly parameterized setting have remained elusive. Intuitively, this is unsurprising: guarantees for gradient-based optimization algorithms typically rely on convexity, with a few notable exceptions such as the Polyak-Lojasiewicz condition~\citep{pl_polyak}. In the linearly parameterized setting, the underlying optimization problem will be convex. When the parameters appear nonlinearly, the problem is in general nonconvex and thus difficult to provide guarantees for.

In this work, we provide new provably globally convergent algorithms for both the linearly and nonlinearly parameterized adaptive control problems, along with new insight into existing adaptive control algorithms for the linearly parameterized setting. Our results for nonlinearly parameterized systems are valid under the monotonicity assumptions of~\citet{tyukin_paper} or the convexity assumptions of~\citet{fradkov-1979}. The monotonicity assumptions are equivalent to those commonly satisfied by generalized linear models in statistics~\citep{glmtron}.

\subsection{Description of primary contributions}
\label{sec:contr}
Our contributions can be categorized into two main advances. 
\begin{enumerate}
    \item We further develop a class of natural gradient and mirror descent-like algorithms that have recently appeared in the literature in the context of physically consistent inertial parameter learning in robotics~\citep{natural_robot} and geodesically convex optimization~\citep{geodesic}. We prove that these algorithms implicitly regularize the learned system model in both the linearly parameterized and nonlinearly parameterized settings.
    \item We construct a general class of higher-order in-time adaptive control algorithms that incorporate momentum into existing adaptation laws. We prove that our new momentum algorithms are stable and globally convergent for both linearly parameterized and nonlinearly parameterized systems. We connect these higher-order methods with the first advance by designing a number of natural gradient and mirror descent-like adaptive control algorithms with momentum, which also implicitly regularize the learned model.
\end{enumerate}
Unlike standard problems in optimization and machine learning, explicit regularization terms cannot be naively added to adaptive control algorithms without impacting stability and performance. Our approach enables a provably stable and globally convergent implementation of regularization in adaptive control. We demonstrate the utility of these results through examples in the context of dynamics prediction, such as sparse estimation of a physical system's Hamiltonian or Lagrangian function, and estimating the weights of a continuous-time recurrent neural network model.
    
It is well-known in adaptive control that the true parameters are only recovered when the desired trajectory satisfies a strong condition known as \textit{persistent excitation}~\citep{annaswammy_book,slot_li_book}. In general, an adaptation law need only find parameters that enable perfect tracking, and very little is known about what parameters are found when the estimator converges without persistent excitation. Our categorization of implicit regularization provides an answer and shows that standard Euclidean adaptation laws lead to parameters of minimum $\ell_2$ norm.

For the second contribution, we utilize the Bregman Lagrangian~\citep{ashia_1, ashia_2, ashia_3} in tandem with the velocity gradient methodology~\citep{fradkov-1979, fradkov-book, fradkov_paper2, fradkov_paper1} to define a general formalism that generates higher-order in-time~\citep{morse} velocity gradient algorithms. Our key insight is that the velocity gradient formalism provides an optimization-like framework that encompasses many well-known adaptive control algorithms, and that the velocity gradient ``loss function'' can be placed directly in the Bregman Lagrangian.

\subsection{Summary of related work}
Our work continues in a long-standing tradition that utilizes a continuous-time view to analyze optimization algorithms, and here we consider a non-exhaustive list. \citet{gen_hamil} develop momentum algorithms from the perspective of Hamiltonian dynamics, while~\citet{maddison} use Hamiltonian dynamics to prove linear convergence of new optimization algorithms without strong convexity.~\citet{muehlenbach, muehlebach_2} study momentum algorithms from the viewpoint of dynamical systems and control.~\citet{boffi-slot} analyze distributed stochastic gradient descent algorithms via dynamical systems and nonlinear contraction theory.~\citet{su_boyd} provide an intuitive justification for Nesterov's accelerated gradient method~\citep{nest-mom}. Continuous-time differential equations were used as early as 1964 by Polyak to derive the classical momentum or ``heavy ball'' optimization method~\citep{polyak1964}. In all cases, continuous-time often affords simpler proofs, and it enables the application of physical intuition when reasoning about optimization algorithms. Given the gradient-based nature of many adaptive control algorithms, the continuous-time view of optimization provides a natural bridge from modern optimization to modern adaptive control.

Despite the simplicity of proofs in continuous-time, finding a discretization that provably retains the convergence rates of a given differential equation is challenging. In a significant advance, \citet{ashia_1} showed that many accelerated methods in optimization can be derived via a variational point of view from a single mathematical object known as the Bregman Lagrangian. The Bregman Lagrangian leads to second-order mass-spring-damper-like dynamics, and careful discretization provides discrete-time algorithms such as Nesterov's celebrated accelerated gradient method~\citep{nest-mom}. We similarly use the Bregman Lagrangian to generate our new adaptive control algorithms, which generalize and extend a recently developed algorithm due to~\citet{gaudio_1}.

Progress has been made in nonlinearly parameterized adaptive control in a number specific cases.~\citet{anna_adapt_1},~\citet{anna_adapt_2}, and~\citet{anna_adapt_3} develop stable adaptive control laws for convex and concave parameterizations, though they may be overly conservative and require solving optimization problems at each timestep.~\citet{II1} and~\citet{II2} develop the Immersion and Invariance (I\&I) approach, and prove global convergence if a certain monotone function can be constructed.~\citet{ortega_sysid} use a similar approach for system identification.~\citet{tyukin_paper} consider dynamical systems satisfying a monotonicity assumption that is essentially identical to conditions required for learning generalized linear models in machine learning and statistics~\citep{glmtron, alphatron, convotron}, and develop provably stable adaptive control algorithms for nonlinearly parameterized systems in this setting.~\citet{fradkov-1979},~\citet{fradkov_paper1},~\citet{fradkov_paper2}, and~\citet{fradkov-book} develop the velocity gradient methodology, an optimization-like framework for adaptive control that allows for provably global convergence under a convexity assumption. As mentioned in Section~\ref{sec:contr}, this framework, in tandem with the Bregman Lagrangian, is central to our development of momentum algorithms.

Our work is strongly related to and inspired by a line of recent work that analyzes the implicit bias of optimization algorithms in machine learning.~\citet{gunasekar_1},~\citet{gunasekar_2}, and~\citet{gunasekar_3} characterize implicit regularization of common gradient-based optimization algorithms such as gradient descent with and without momentum, as well as natural and mirror descent in the settings of regression and classification.~\citet{azizan_1}, and~\citet{azizan_2} arrive at similar results via a different derivation based on results from $\HH_{\infty}$ control. Similarly,~\citet{belkin} consider the importance of implicit regularization in the context of the successes of deep learning. Our results are the adaptive control analogues of those presented in these papers.

\subsection{Paper outline}
The paper is organized as follows. In Section~\ref{sec:adaptive}, we present some required mathematical background on direct adaptive control in the linearly and nonlinearly parameterized settings, along with a review of velocity gradient methods. 
In Section~\ref{sec:nat} we analyze the implicit bias of adaptive control algorithms.
In Section~\ref{sec:obs_predict} we consider observer and dynamics predictor design, Hamiltonian dynamics prediction, control of Lagrangian systems, and estimation of recurrent neural networks. 
In Section~\ref{ssec:breg}, we provide a review of the Bregman Lagrangian, and use it in Section~\ref{sec:main} to devise new higher-order adaptation algorithms with momentum.
We extend these to the non-Euclidean section in Section~\ref{sec:ho_nat}.
In Section~\ref{ssec:elastic}, we discuss extensions of our adaptation laws to distributed settings.
We illustrate our results via simulation in Section~\ref{sec:sim}, and we conclude with some closing remarks and future directions in Section~\ref{sec:conc}.

\section{Direct adaptive control}
In this section, we provide an introduction to direct adaptive control for both linearly parameterized and nonlinearly parameterized systems, along with a description of some natural gradient-like adaptive laws that have appeared in the recent literature.
\label{sec:adaptive}
\subsection{Linearly parameterized dynamics}
\label{ssec:adaptive_basic_lin}
For simplicity, we restrict ourselves to the class of $\text{n}^{\text{th}}$-order nonlinear systems
\begin{equation}
    x^{(n)} + f(\bx, \ba, t) = u
    \label{eqn:gen_sys}
\end{equation}
where $x^{(i)} \in \mathbb{R}$ denotes the $i^{\text{th}}$ derivative of $x$, $\bx = \left(x, x^{(1)}, \hdots, x^{(n-1)}\right)^\T \in \mathbb{R}^n$ is the system state, $\ba \in \mathbb{R}^p$ is a vector of unknown parameters, $f: \mathbb{R}^n \times \mathbb{R}^p \times \mathbb{R}_+ \rightarrow \mathbb{R}$ is of known functional form but is unknown due to its dependence on $\ba$, and $u \in \mathbb{R}$ is the control input. We seek to design a feedback control law $u = u(\bx, \hat{\ba})$ that depends on a set of adjustable parameters $\hat{\ba} \in \mathbb{R}^p$ and ensures that $\bx(t) \rightarrow \bx_d(t)$ where $\bx_d(t) \in \mathbb{R}^n$ is a known desired trajectory. Along the way, we require that all system signals remain bounded. The estimated parameters $\hat{\ba}$ are updated according to a learning rule or adaptation law
\begin{equation}
    \dot{\hat{\ba}} = \mathbf{g}(\ba, \hat{\ba}, \bx)
    \label{eqn:gen_adapt}
\end{equation}
where $\bg : \mathbb{R}^p\times\mathbb{R}^p\times\mathbb{R}^n\rightarrow\mathbb{R}^p$ must be implementable solely in terms of known system signals despite its potential dependence on $\ba$. For $\text{n}^{\text{th}}$ order systems as considered in (\ref{eqn:gen_sys}), a common approach is to define the \textit{sliding variable}~\citep{slot_li_book}
\begin{equation}
    s(\bx, \bx_d(t)) = \left(\frac{d}{dt} + \lambda\right)^{n-1} \tilde{x} = \tilde{x}^{(n-1)} - \tilde{x}^{(n-1)}_r
    \label{eqn:s_def}
\end{equation}
where $\lambda > 0$ is a constant, and $\tilde{x}(t) = x(t) - x_d(t)$. We have defined $\tilde{x}^{(i)}(t) = x^{(i)}(t) - x^{(i)}_d(t)$ and $\tilde{x}^{(n-1)}_r$ as the remainder based on the definition of $s(\bx, \bx_d(t))$. According to the definition (\ref{eqn:s_def}), $s(\bx, \bx_d(t))$ obeys the differential equation
\begin{equation}
    \dot{s} = u - f(\bx, \ba, t) - \tilde{x}^{(n)}_r.
    \label{eqn:s_dyn_def}
\end{equation}
Hence, from (\ref{eqn:s_dyn_def}), we may choose
\begin{equation}
    \label{eqn:control}
    u = f(\bx, \hat{\ba}, t) + \tilde{x}^{(n)}_r - \eta s
\end{equation} 
to obtain the stable first-order linear filter
\begin{equation}
    \dot{s} = -\eta s + f(\bx, \hat{\ba},t) - f(\bx, \ba,t).
    \label{eqn:s_dyn}
\end{equation}
For future convenience, we define $\tilde{f}(\bx, \hat{\ba}, \ba, t) = f(\bx, \hat{\ba}, t) - f(\bx, \ba, t)$. From the definition of $s(\bx, \bx_d(t))$ in (\ref{eqn:s_def}), $s(\bx, \bx_d(t)) = 0$ defines the \textit{dynamics} 
\begin{equation}
    \left(\frac{d}{dt} + \lambda\right)^{n-1}\tilde{x} = 0.
    \label{eqn:s=0}
\end{equation}
Equation (\ref{eqn:s=0}) is a stable $(n-1)^{\text{th}}$-order filter which ensures that $\tilde{x}(t) \rightarrow 0$ exponentially. For systems of the form (\ref{eqn:gen_sys}), it is thus sufficient to consider the two first-order dynamics (\ref{eqn:gen_adapt}) and (\ref{eqn:s_dyn}). The adaptive control problem has thus been reduced to finding a learning algorithm that ensures $s(\bx(t), \bx_d(t)) \rightarrow 0$.
\begin{rmk}
Systems in the matched uncertainty form 
\begin{equation*}
\dot{\bx} = \bA\bx + \bb\left(u - f(\bx, \ba, t)\right),
\end{equation*}
where the constant pair $(\bA,\bb)$ is controllable and the constant parameter vector $\ba$ in the nonlinear function $ f(\bx, \ba, t)$ is unknown, can always be put in the form \eqref{eqn:gen_sys} by using a state transformation to the second controllability canonical form $-$ see~\citet{luenberger}, Chapter 8.8. After such a transformation, the new state variables $\bz$ satisfy $\dot{z}_i = z_{i+1}$ for $i < n$ and $\dot{z}_n = -\sum_{i=1}^{n-1}a_i z_i + u - f(\bx, \ba, t)$. Defining $s$ as in (\ref{eqn:s_def}) and choosing $u$ accordingly leads to (\ref{eqn:s_dyn}). Hence, all results in this paper extend immediately to such systems.$\qeddef$
\end{rmk}
\begin{rmk}
\label{rmk:error}
Our results may be simply extended to other error models~\citep{anna_adapt_2} of the form (\ref{eqn:s_dyn}), or error models with similar input-output guarantees, as summarized by Lemma \ref{lem:conv}.$\qeddef$
\end{rmk}
\begin{rmk}
We will use $\bff$ to denote the equivalent first-order system to (\ref{eqn:gen_sys}), $\dot{\bx} = \bff(\bx, \ba,t) + \bu$, where $\bff = \left(x_2, x_3, \hdots, f(\bx, \ba,t)\right)$ and $\bu = \left(0, 0, \hdots, u\right)$.$\qeddef$
\end{rmk}
The classic setting for adaptive control assumes that the unknown nonlinear dynamics depends linearly on the set of unknown parameters, that is
\begin{equation*}
    f(\bx, \ba, t) = \bY(\bx, t)\ba,
\end{equation*}
with $\bY:\mathbb{R}^n \times \mathbb{R}_+ \rightarrow \mathbb{R}^{1\times p}$ a known function. In this setting, a well-known algorithm is the adaptive controller of~\citet{slot_slide_adaptive}, given by
\begin{equation}
    \dot{\hat{\ba}} = -\bP\bY(\bx, t)^\T s(\bx, \bx_d(t)),
    \label{eqn:slotine_li}
\end{equation}
and its extension to multi-input adaptive robot control~\citep{slot_li_robot}, where $\bP = \bP^\T > 0 \in \mathbb{R}^{p\times p}$ is a constant positive definite matrix of learning rates. Consideration of the Lyapunov-like function $V(\bx, \ha, t) = \frac{1}{2}s(\bx, \bx_d(t))^2 + \frac{1}{2}\tilde{\ba}^\T\bP^{-1}\tilde{\ba}$ shows stability of the feedback interconnection of (\ref{eqn:s_dyn}) and (\ref{eqn:slotine_li}) and convergence to the desired trajectory via an application of Lemma~\ref{lem:barbalat}. We will refer to (\ref{eqn:slotine_li}) as the Slotine and Li controller.

In this work, we make a mild local boundedness assumption that simplifies some of the proofs.
\begin{assumption}
\label{assmp:bound}
If $\Vert\bx\Vert \leq \infty$ and $\Vert\ha\Vert < \infty$, then $\forall t \ge 0,\ |\hat{f}(\bx, \ha, t)| < \infty .\qeddef$
\end{assumption}
\subsection{Nonlinearly parameterized dynamics}
\label{ssec:adaptive_basic_nlin}
While a difficult problem in general, significant progress has been made for the nonlinearly parameterized adaptive control problem under the assumption of \textit{monotonicity}, and several notions of monotonicity have appeared in the literature~\citep{tyukin_paper, tyukin_book, II1, II2, ortega_sysid}. We consider one such notion as presented by~\citet{tyukin_paper}, which is captured in the following assumption.

\begin{assumption}
\label{assmp:tyukin}
There exists a known time- and state-dependent function $\balf : \mathbb{R}^n \times \mathbb{R}_{\geq 0} \rightarrow \mathbb{R}^p$ such that
\begin{align}
    \tilde{\ba}^\T\balf(\bx, t)\left(f\left(\bx, \hat{\ba}, t\right) - f\left(\bx, \ba, t\right)\right) &\geq 0,
    \label{assmp:tyukin_1}\\
     |\balf(\bx, t)^\T\tilde{\ba}| &\geq \frac{1}{D_1}|f\left(\bx, \hat{\ba}, t\right) - f\left(\bx, \ba, t\right)|.
    \label{assmp:tyukin_2}
\end{align}
where $D_1 > 0$ is a positive scalar.$\qeddef$
\end{assumption}
This assumption is satisfied, for example, by all functions of the form 
\begin{equation}
    f(\bx, \ba, t) = \lambda(\bx, t)f_m(\bx, \bphi(\bx, t)^\T\ba, t),
    \label{eqn:single_layer}
\end{equation} 
where $\lambda : \mathbb{R}^n \times \mathbb{R}_{\geq 0} \rightarrow \mathbb{R}$, $\bphi : \mathbb{R}^n\times\mathbb{R}_{\geq 0} \rightarrow \mathbb{R}^p$, $f_m : \mathbb{R}^n \times \mathbb{R} \times \mathbb{R}_{\geq 0} \rightarrow \mathbb{R}$, and where $f_m$ is monotonic and Lipschitz in $\bphi(\bx, t)^\T\ba$. In this setting, $\balf(\bx, t)$ may be taken as $\balf(\bx, t) = (-1)^pD_1\lambda(\bx, t)\bphi(\bx, t)$ where $p=0$ if $f_m$ is non-decreasing in $\bphi^\T\ba$ and $p=1$ if $f_m$ is non-increasing in $\bphi^\T\ba$~\citep{tyukin_paper, tyukin_book}.

Under Assumption~\ref{assmp:tyukin},~\citet{tyukin_paper} showed that the adaptation law
\begin{equation}
    \dot{\hat{\ba}} = - \tilde{f}(\bx, \hat{\ba}, \ba, t)\bP\balf(\bx, t)
    \label{eqn:tyukin_alg}
\end{equation}
with $\bP = \bP^\T > 0$ a positive definite matrix of learning rates of appropriate dimensions ensures that $\tilde{f}(\bx(\cdot), \ha(\cdot), \ba, \cdot) \in \LL_2$ over the maximal interval of existence of $\bx(t)$. Under suitable conditions on the error model, this then ensures that $\tilde{f}(\bx(\cdot), \ha(\cdot), \ba, \cdot) \in \LL_2\cap\LL_\infty$, $\bx(t)$ and $\ha(t)$ both remain bounded for all $t$, and that $\bx(t)\rightarrow\bx_d(t)$. The proof follows by consideration of the Lyapunov-like function $V(\ha) = \frac{1}{2}\tilde{\ba}^\T\bP^{-1}\tilde{\ba}$. 

While $\tilde{f}(\bx, \ha, \ba, t)$ itself is unknown, and hence (\ref{eqn:tyukin_alg}) is not directly implementable, it is contained in $\dot{s}(\bx, \bx_d(t), t)$. Intuitively, unknown quantities contained in $\dot{s}(\bx, \bx_d(t), t)$ can be obtained in the adaptation dynamics through a proportional term in $\ha$ that contains $s(\bx, \bx_d(t))$. This idea of gaining a ``free'' derivative is the basis of the reduced-order Luenberger observer for linear systems~\citep{luenberger}. Proportional-integral adaptive laws of this type have been known as algorithms in finite form~\citep{fradkov-book, tyukin_finite} and appear in the well-known I\&I framework~\citep{II1, II2}. Following this prescription, (\ref{eqn:tyukin_alg}) may be implemented in a proportional-integral form,
\begin{align}
    \label{eqn:tyukin_pi_1}
    \bxi(\bx, t) &= -\bP s(\bx, \bx_d(t))\balf(\bx, t),\\
    \label{eqn:tyukin_pi_2}
    \brho(\bx, t) &= \bP\int_{x_n(t_0)}^{x_n(t)}s(\bx, \bx_d(t))\frac{\p \balf(\bx, t)}{\p x_n}dx_n,\\
    \label{eqn:tyukin_pi_3}
    \hat{\ba} &= \overline{\ba} + \bxi(\bx, t) + \brho(\bx, t),\\
    \dot{\overline{\ba}} &= -\eta s(\bx, \bx_d(t))\bP\balf(\bx, t) + \bP s(\bx, \bx_d(t)) \sum_{i=1}^{n-1}\frac{\p \balf(\bx, t)}{\p x_i}x_{i+1} - \sum_{i=1}^{n-1}\frac{\p \brho(\bx, t)}{\p x_i}x_{i+1} \nonumber\\
    &\phantom{=} - \left(\frac{\p \brho(\bx, t)}{\p \bx_d}\right)^\T\dot{\bx}_d - \frac{\p \bxi}{\p t}(\bx, t) - \frac{\p \brho}{\p t}(\bx, t).
    \label{eqn:tyukin_pi_4}
\end{align}
Algorithm (\ref{eqn:tyukin_alg}) is similar to a gradient flow algorithm. If $f(\bx, \ba, t)$ has the form (\ref{eqn:single_layer}) and is non-decreasing, gradient flow on the loss function $L(\bx, \hat{\ba}, \ba, t) = \frac{1}{2}\tilde{f}^2(\bx, \hat{\ba}, \ba, t)$ with a gain matrix $D_1\bP$ leads to
\begin{equation*}
    \dot{\hat{\ba}} = -\tilde{f}(\bx, \hat{\ba}, \ba, t)f_m'(\bx, \bphi^\T\hat{\ba}, t)\bP\balf(\bx, t)
\end{equation*}
where $'$ denotes differentiation with respect to the second argument. $f_m'(\bx, \bphi^\T\hat{\ba}, t)$ is of known sign due to the monotonicity assumption, but of unknown magnitude. It is sufficient to remove this quantity from the adaptation law and instead to follow the \textit{pseudogradient} $\tilde{f}(\bx, \ha, \ba, t)\balf(\bx, t)$ despite non-convexity of the square loss in this setting. Similarly, if $f$ is non-increasing, we find 
\begin{equation*}
    \dot{\hat{\ba}} = \tilde{f}(\bx, \hat{\ba}, \ba, t)f_m'(\bx, \bphi^\T\hat{\ba}, t)\bP\balf(\bx, t)
\end{equation*}
and it is sufficient to set $f_m'(\bx, \bphi^\T\ba, t)$ to negative one.

\subsection{Velocity gradient algorithms}
\label{ssec:sg}
We now provide a brief introduction to a class of adaptive control methods known as velocity gradient algorithms~\citep{fradkov-book, fradkov-1979, fradkov_paper1, fradkov_paper2}. In their most basic form, they are specified by a ``local'' goal functional $Q(\bx, t):\mathbb{R}^n\times\mathbb{R}_+\rightarrow\mathbb{R}$; driving $Q(\bx, t)$ to zero then ensures that $\bx(t)\rightarrow\bx_d(t)$. The adaptation law is defined as
\begin{equation}
    \dot{\hat{\ba}} = - \bP\nabla_{\hat{\ba}}\dot{Q}(\bx, \hat{\ba}, t),
    \label{eqn:sg_alg}
\end{equation}
where $\bP = \bP^\T > 0$ is a positive definite matrix of learning rates of appropriate dimension, and $\dot{Q}(\bx, \ha, t) = \left(\nabla_{\bx}Q(\bx, t)\right)^\T\dot{\bx} + \frac{\p Q(\bx, t)}{\p t}$. Intuitively, while the goal functional $Q(\bx, t)$ may only depend on the control parameters $\hat{\ba}$ indirectly through $\bx$, its time derivative will depend explicitly on $\hat{\ba}$ through $\dot{\bx}$\footnote{It will also depend on $\ba$, but we suppress this dependence for notational simplicity.}. The adaptation law (\ref{eqn:sg_alg}) ensures that $\hat{\ba}$ moves in a direction that instantaneously decreases $\dot{Q}(\bx, \ha, t)$. Under the conditions specified by Assumptions \ref{assmp:sg1}-\ref{assmp:sg3}, this causes $\dot{Q}(\bx, \ha, t)$ to be negative for long enough to drive $Q(\bx, t)$ to zero~\citep{fradkov-book}. $Q(\bx, t)$ is required to satisfy three main assumptions to ensure that this is the case.

\begin{assumption}
$Q(\bx, t)$ is non-negative and radially unbounded, so that $Q(\bx, t) \geq 0$ for all $\bx$, $t$ and $Q(\bx, t) \rightarrow \infty$ when $\Vert\bx\Vert \rightarrow \infty$. $Q(\bx, t)$ is uniformly continuous in $t$ whenever $\bx$ is bounded. $\qeddef$
\label{assmp:sg1}
\end{assumption}

\begin{assumption}
There exists an ideal set of control parameters $\ba$ such that the origin of the system (\ref{eqn:gen_sys}) is globally asymptotically stable when the control is evaluated at $\ba$. Furthermore, $Q(\bx, t)$ is a Lyapunov function for the system when the control is evaluated at $\ba$. That is, there exists a strictly increasing function $\rho:\mathbb{R}\rightarrow\mathbb{R}_+$ such that $\rho(0) = 0$ with $\dot{Q}(\bx, \ba, t) \leq -\rho(Q(\bx, \ba, t))$.$\qeddef$
\label{assmp:sg2}
\end{assumption}

\begin{assumption}
The time derivative of $Q$ is convex in the control parameters $\hat{\ba}$, i.e.,
\begin{equation}
    \dot{Q}(\bx, \ba_1, t) \geq \dot{Q}(\bx, \ba_2, t) + \left(\ba_1 - \ba_2\right)^\T\nabla_{\ba_2}\dot{Q}(\bx, \ba_2, t),
    \label{eqn:sg_conv}
\end{equation}
is satisfied for all $\ba_1$ and $\ba_2$.$\qeddef$
\label{assmp:sg3}
\end{assumption}
The properties of \eqref{eqn:sg_alg} are summarized in the following proposition~\citep{fradkov-book}. 

\begin{prop}[cf.~\citet{fradkov-1979}]
\label{prop:sg_local}
Consider the local velocity gradient algorithm (\ref{eqn:sg_alg}) under Assumptions \ref{assmp:sg1}-\ref{assmp:sg3}. Then all solutions $(\bx(t), \hat{\ba}(t))$ of (\ref{eqn:gen_sys}) and (\ref{eqn:sg_alg}) remain bounded, and for all $\bx(0) \in \mathbb{R}^n$
\begin{equation*}
    \lim_{t\rightarrow\infty} Q(\bx(t), t) = 0.
\end{equation*}
\end{prop}
The proof follows by consideration of the Lyapunov-like function $V(\bx, t) = Q(\bx, t) + \frac{1}{2}\tilde{\ba}^\T\bP^{-1}\tilde{\ba}$.

\begin{rmk}
\label{rmk:slotine_li}
If $Q(\bx, t)$ is chosen so that $\dot{Q}(\bx, \ha, t)$ depends on $\hat{\ba}$ only through $\hat{f}(\bx, \ha, t)$ and $f(\bx, \ba, t)$ is linearly parameterized, then Assumption~\ref{assmp:sg3} will immediately be satisfied by convexity of affine functions. Indeed, consider defining the goal functional $Q(\bx, t) = \frac{1}{2}s(\bx, \bx_d(t))^2$ for system (\ref{eqn:gen_sys}). It is clear that this proposed goal functional satisfies Assumptions~\ref{assmp:sg1} and \ref{assmp:sg2} for bounded $x_d(t)$. Then $\dot{Q}(\bx, \ha, t) = -\eta s(\bx, \bx_d(t))^2 + s(\bx, \bx_d(t))\tilde{f}(\bx, \ha, \ba, t)$, and (\ref{eqn:sg_alg}) exactly recovers the Slotine and Li controller (\ref{eqn:slotine_li}).$\qeddef$
\end{rmk}

\begin{rmk}
\label{rmk:dvdx}
An alternative perspective on velocity gradient algorithms can be found by using the expression $\dot{Q}(\bx, \ha, t) = \left(\nabla_{\bx}Q(\bx, t)\right)^\T\dot{\bx} + \frac{\p Q(\bx, t)}{\p t}$. Assume that $\dot{\bx} = \bu(\bx, \ha, t) - \bY(\bx, t)\ba$, and set $\bu(\bx, \ha, t) = \bY(\bx, t)\ha + \bu_d(\bx, \bx_d(t))$ where $\bu_d(\bx, \bx_d(t))$ ensures that $\bx(t) \rightarrow \bx_d(t)$ for $\ha = \ba$. Then $\nabla_{\ha}\dot{Q}(\bx, \ha, t) = \bY(\bx, t)^\T \nabla_{\bx}Q(\bx, t)$. This shows that the adaptation law $\dot{\ha} = -\bP\nabla_{\ha}\dot{Q}(\bx, \ha, t) = -\bP\bY(\bx, t)^\T\nabla_{\bx}Q(\bx, t)$ transforms the gradient of $Q(\bx, t)$ with respect to $\bx$ by pre-multiplication by the regressor $\bY(\bx, t)^\T$. This perspective also applies to the adaptation law for contracting systems developed in~\citep{brett}. Conversely, this perspective shows that if a Lyapunov function $V(\bx, t)$ is known for a nominal system $\dot{\bx} = \bff(\bx, t)$, then the adaptive control input $\bu(\bx, \ha, t) = \bY(\bx, t)\ha$ with adaptation law $\dot{\ha} = -\bP\bY(\bx, t)^\T\nabla_{\bx}V(\bx, t)$ will return the perturbed system $\dot{\bx} = \bff(\bx, t) + \bu(\bx, \ha, t) - \bY(\bx, t)\ba$ back to its nominal behavior. This was recently exploited to learn adaptive control laws directly from data in~\citet{boffi2020learning}.
\end{rmk}

Rather than a local functional, one may instead specify an integral goal functional of the form $Q(\bx, \hat{\ba}, t) = \int_0^t R(\bx(t'), \hat{\ba}(t'), t')dt'$. In this case, (\ref{eqn:sg_alg}) takes the form 
\begin{equation}
    \dot{\hat{\ba}} = -\bP\nabla_{\hat{\ba}}R(\bx, \hat{\ba}, t).
    \label{eqn:sg_int}
\end{equation} 
Equation (\ref{eqn:sg_int}) is a gradient flow algorithm on the loss function $R(\bx, \hat{\ba}, t)$. We now replace Assumptions \ref{assmp:sg1} and \ref{assmp:sg2} by a slightly modified setting.

\begin{assumption}
$R$ is a non-negative function and $R(\bx(t), \ha(t), t)$ is uniformly continuous in $t$ for bounded $\bx$ and $\ha$. Furthermore, $\nabla_{\ha}R(\bx, \ha, t)$ is locally bounded in $\bx$ and $\ha$ uniformly in $t$. $\qeddef$
\label{assmp:sg4}
\end{assumption}

\begin{assumption}
There exists an ideal set of controller parameters $\ba$ and a scalar function $\mu$ such that $\int_0^\infty \mu(t')dt' < \infty$, $\lim_{t\rightarrow \infty} \mu(t) = 0$, and $R(\bx(t), \ba, t) \leq \mu(t)$ for all $t$.$\qeddef$
\label{assmp:sg5}
\end{assumption}

The properties of algorithm (\ref{eqn:sg_int}) are summarized in the following proposition~\citep{fradkov-book}.

\begin{prop}
\label{prop:sg_nonloc}
Consider the integral velocity gradient algorithm (\ref{eqn:sg_int}) where the goal functional $Q(\bx, t)$ satisfies Assumptions \ref{assmp:sg3}-\ref{assmp:sg5}. Then $Q(\bx(t), t) \leq \alpha$ where
\begin{equation*}
    \alpha = \frac{1}{2}\tilde{\ba}(0)^\T\bP^{-1}\tilde{\ba}(0) + \int_0^\infty\mu(t')dt',
\end{equation*}
and $\int_0^{T_x} R(\bx(t'), \ha(t'), t')dt' < \infty$ where $T_x$ denotes the maximal interval of existence of $\bx(t)$. Furthermore, $R(\bx(t), \ha(t), t) \rightarrow 0$ for any bounded solution $\bx(t)$.
\end{prop}
The proof follows by consideration of the Lyapunov-like function $V = \int_0^tR(\bx(t'), \ha(t'), t')dt' + \frac{1}{2}\tilde{\ba}^\T\bP^{-1}\tilde{\ba} + \int_{t}^{\infty}\mu(t')dt'$.

Integral functionals allow the specification of a control goal that depends on all past data. $R(\bx, \ha, t)$ is chosen so that it does not necessarily depend on the structure of the dynamics, but depends explicitly on $\ha$. Local functionals, on the other hand, result in adaptation laws that \textit{do} have an explicit dependence on the dynamics through the appearance of the term $\frac{\p Q}{\p \bx}(\bx, t)^\T\dot{\bx}$ in $\dot{Q}(\bx, \ha, t)$. 

Integral functionals can be particularly useful if $R(\bx(t), \ha(t), t) \rightarrow 0$ implies the desired control goal. In this work, we will focus on the choice $R(\bx, \ha, t) = \frac{1}{2}\tilde{f}(\bx, \ha, \ba, t)^2$, which will require a PI form as described in Section~\ref{sec:adaptive} in the context of Tyukin's algorithm. For this choice of $R(\bx, \ha, t)$ the result of Proposition~\ref{prop:sg_nonloc} implies that $\tilde{f}(\bx(\cdot), \ha(\cdot), \ba, \cdot)\in\LL_2$ over the maximal interval of existence of $\bx(t)$. For some error models, this is enough to ensure that $\bx(\cdot)\in\LL_\infty$, and hence that $\tilde{f}(\bx(t), \ha(t), \ba, t)\rightarrow 0$ and $\bx(t)\rightarrow\bx_d(t)$. This is verified for \eqref{eqn:s_dyn_def} in Lemma~\ref{lem:conv}.

Goal functionals can also be written as a sum of local and integral functionals with similar guarantees, and these approaches will lead to composite algorithms in the subsequent sections. The interested reader is referred to~\citet{fradkov-book}, Chapter 3 for more details.

\subsection{The Bregman divergence and natural adaptation laws}
\label{ssec:fo_rie_adapt}
\citet{natural_robot} introduced an elegant modification of the Slotine and Li adaptive robot controller, later generalized by~\citet{geodesic}. It consists of replacing the usual parameter estimation error term $\frac{1}{2}\tilde{\ba}^\T\bP^{-1}\tilde{\ba}$ in the Lyapunov-like function $V(\bx, \ha, t) = \frac{1}{2}s(\bx, \bx_d(t))^2 + \frac{1}{2}\tilde{\ba}^\T\bP^{-1}\tilde{\ba}$ with the Bregman divergence~\citep{bregman},
\begin{equation*}
    \bregd{\by}{\bx} = \psi(\by) - \psi(\bx) - \left(\by - \bx\right)^\T\nabla \psi(\bx)
    \label{eqn:breg_div}
\end{equation*}
to obtain the new ``non-Euclidean'' Lyapunov-like function
\begin{equation}
V(\bx, \ha, t) = \frac{1}{2}s(\bx, \bx_d(t))^2 + \bregd{\ba}{\ha}
\end{equation}
for an arbitrary strongly convex function $\psi:\mathbb{R}^p \rightarrow \mathbb{R}$. 

The Bregman divergence may be understand as the error made when approximating $\psi(\mathbf{y})$ by a first-order Taylor expansion around $\bx$. It is guaranteed to be non-negative for strongly convex functions by the first-order characterization of convexity. While it is not a norm in general, it defines a distance-like function for $\psi$ strongly convex related to the \textit{Hessian metric} $\frac{1}{2}\Vert\bx\Vert_{\nabla^2\psi}^2 = \frac{1}{2}\bx^\T\nabla^2\psi(\bx)\bx$. As two simple examples, for $\psi(\bx) = \frac{1}{2}\Vert\bx\Vert^2$, $\bregd{\bx}{\by} = \frac{1}{2}\Vert\bx - \by\Vert^2$. For $\psi(\bx) = \frac{1}{2}\bx^\T\mathbf{Q}\bx$ with $\bQ > 0$ a positive definite matrix, $\bregd{\bx}{\by} = \frac{1}{2}\left(\bx-\by\right)^\T\mathbf{Q}\left(\bx-\by\right)$. For general convex functions, $\bregd{\cdot}{\cdot}$ can always be written via Hadamard's Lemma as
\begin{equation*}
    \bregd{\by}{\bx} = \left(\by-\bx\right)^\T\left(\int_0^1\nabla^2\psi\left(\bx + s\left(\by-\bx\right)\right)(1-s)ds\right)\left(\by-\bx\right).
\end{equation*}

The derivative of the Bregman divergence is simply
\begin{equation}
    \label{eqn:breg_deriv}
    \frac{d}{dt}\bregd{\ba}{\ha} = \tilde{\ba}^\T\nabla^2\psi(\hat{\ba})\dot{\hat{\ba}},
\end{equation}
which can be directly used to show stability of the adaptation law
\begin{equation*}
    \dot{\ha} = -\left[\nabla^2\psi(\ha)\right]^{-1}\bY(\bx, t)^\T s(\bx, \bx_d(t)).
\end{equation*}
This procedure replaces the gain matrix $\bP$ in the adaptation law by the $\ha$-dependent \emph{inverse Hessian} $\left[\nabla^2\psi(\ha)\right]^{-1}$ of the strongly convex function $\psi$. In essence, this amounts to the adaptive control equivalent of the \emph{natural gradient} algorithm in the sense of~\citet{amari}, so that the resulting adaptation law respects the underlying Riemannian geometry captured by the Hessian metric $\nabla^2\psi(\hat{\ba})$. The standard adaptation law $\dot{\ha} = -\bP\bY(\bx, t)^\T s(\bx, \bx_d(t))$ uses the constant metric $\bP^{-1}$, which in turn explains the appearance of $\bP$ in the natural gradient-like system. 

The choice of $\psi$ enables the design of adaptation algorithms that respect physical Riemannian constraints~\citep{wensing_ty} obeyed by the true parameters, as in the estimation of mass properties in robotics~\citep{wensing_lmi}. Similarly, it allows one to introduce a priori bounds on parameter estimates without resorting to parameter projection techniques by choosing $\psi$ to be a $\log$-barrier function~\citep{geodesic}. In Section~\ref{ssec:imp_reg}, we further prove that the choice of $\psi$ implicitly regularizes the learned system model.

\begin{rmk}
It is straightforward to generalize velocity gradient algorithms to the non-Euclidean setting. The adaptation law
\begin{equation*}
    \dot{\ha} = -\gamma\left(\nabla^2\psi(\ha)\right)^{-1}\nabla_{\ha}\dot{Q}(\bx, \ha, t)
\end{equation*}
with $\gamma > 0$ a positive learning rate and $\psi:\mathbb{R}^p\rightarrow\mathbb{R}$ a strongly convex function obtains matching guarantees as \eqref{eqn:sg_alg}. The proof follows by consideration of the Lyapunov-like function $V(\bx, \ha, t) = Q(\bx, t) + \bregd{\ba}{\ha}$.
Similarly, the Lyapunov-like function $V(\bx, \ha, t) = \int_0^tR(\bx(t'), \ha(t'), t')dt' + \bregd{\ba}{\ha} + \int_t^{\infty}\mu(t')dt'$ provides a guarantee for the integral velocity gradient algorithm
\begin{equation*}
    \dot{\ha} = -\gamma\left(\nabla^2\psi(\ha)\right)^{-1}\nabla_{\ha}R(\bx, \ha, t).
\end{equation*}
In the same manner, a non-Euclidean metric can be incorporated into Tyukin's algorithm \eqref{eqn:tyukin_alg}.
\end{rmk}

\section{Natural gradient adaptation and implicit regularization}
\label{sec:nat}
In this section, we show that the ``natural'' adaptation algorithms of the previous section implicitly regularize the learned system model.

\subsection{Implicit regularization and adaptive control}
\label{ssec:imp_reg}
As exemplified by deep networks, modern machine learning utilizes highly over-parameterized models capable of interpolating the training data while still generalizing well to unseen examples. The classical principles of statistical learning theory emphasize a trade-off between generalization performance and model capacity, and predict that in the highly over-parameterized regime, generalization performance should be poor due to a tendency of the model to fit noise in the training data. Nevertheless, empirical evidence indicates that deep networks and other modern machine learning models do not obey classical statistical learning wisdom~\citep{belkin}, and can even generalize with significant label noise~\citep{und_dl}. 

More surprisingly, the ability to simultaneously fit label noise in the training data yet generalize to new examples has been observed in over-parameterized linear models~\citep{barlett, sahai}. A possible explanation for the ability of highly over-parameterized models to generalize when optimized using simple first-order algorithms is their \textit{implicit bias} -- that is, the tendency of an algorithm to converge to a particular (e.g. minimum norm) solution when there are many that interpolate the training data~\citep{gunasekar_1, gunasekar_2, gunasekar_3, azizan_1, azizan_2}.

In adaptive control, the existence of many possible parameter vectors $\ha$ that lead to zero tracking error is not unique to the over-parameterized case. Unless the trajectory is \textit{persistently exciting}\footnote{A typical characterization of persistent excitation in the linearly parameterized setting is that there exists some $\delta > 0$ and some $T > 0$ such that for all $t, \int_t^{t+T} \bY(\bx(\tau), \tau)^\T\bY(\bx(\tau), \tau)d\tau \geq \delta \bI$.}~\citep{annaswammy_book, slot_li_book}, it is well-known that $\ha$ will not converge to the true parameters $\ba$ in general. Depending on the complexity of the trajectory, there may even be many solutions in the \textit{under-parameterized} case where $\dim(\ha) < \dim(\ba)$. To achieve perfect tracking, the adaptation algorithm need only fit the unknown dynamics $f(\bx(t), \ba, t)$ along the trajectory rather than the whole state space, so that the \textit{effective} number of parameters may be less than $\dim(\ba)$. 

The space of possible solutions in the linearly parameterized case is described by the time-dependent null space of $\bY(\bx(t), t)$: when $\bx(t)\rightarrow\bx_d(t)$, we can conclude that $\bY(\bx_d(t), t)\tilde{\ba}(t) = 0$, and hence that $\ha(t) = \ba + \hat{\mathbf{n}}(t)$ where $\bY(\bx_d(t), t)\hat{\mathbf{n}}(t) = 0$ for all $t$. In the over-parameterized case when $\dim(\ha) > \dim(\ba)$, the set of parameters that achieve zero tracking error is not unique regardless of the complexity of the desired trajectory. 

We now categorize the implicit bias of the natural adaptive laws of the previous section. The proof technique exploits a continuous-time generalization of a recent technique due to~\citet{azizan_1} and~\citet{azizan_2}. This proof of implicit regularization exactly categorizes the parameter vectors found by adaptive control algorithms. By taking $\psi(\cdot) = \frac{1}{2}\norm{\cdot}_2^2$, it recovers a result for the standard Euclidean setting. Towards this end, define the interpolating set
\begin{equation}
    \label{set:A}
    \mathcal{A} = \{\bthet | \bff(\bx(t), \bthet, t) = \bff(\bx(t), \ba, t) \ \ \forall t\}.
\end{equation}
(\ref{set:A}) contains only parameters that interpolate the unknown dynamics $\bff(\bx(t), \ba, t)$ along the entire trajectory. We are now in a position to state the following proposition.
\begin{prop}
\label{prop:implicit_reg}
Consider the natural velocity gradient algorithm
\begin{equation}
    \label{eqn:natural_sl}
    \dot{\ha} = -\left[\nabla^2\psi(\ha)\right]^{-1}\nabla_{\ha}\dot{Q}(\bx, \ha, t),
\end{equation}
where $\psi:\mathbb{R}^p\rightarrow\mathbb{R}$ is a strongly convex function. Suppose that the unknown dynamics $f(\bx, \ba, t) = \bY(\bx, t)\ba$ is linearly parameterized. Assume that $\ha(t) \rightarrow \ha_\infty \in \mathcal{A}$. Then $\ha_\infty = \arg\min_{\bthet \in \mathcal{A}}\bregd{\bthet}{\ha(0)}$.
In particular, if $\ha(0) = \arg\min_{\bthet\in\mathbb{R}^p}\psi(\bthet)$, then 
\begin{equation}
    \label{eqn:imp_reg}
    \ha_\infty = \arg\min_{\bthet \in \mathcal{A}} \psi(\bthet).
\end{equation}
\end{prop}
The proof is given in Appendix~\ref{app2:prop:implicit_reg}. (\ref{eqn:imp_reg}) captures the implicit regularization imposed by the adaptation algorithm (\ref{eqn:natural_sl}): out of all possible interpolating parameters, it chooses the $\ha$ that achieves the minimum value of $\psi$.

\begin{rmk}
The assumptions of Proposition~\ref{prop:implicit_reg} provide a setting where theoretical insight may be gained into the implicit regularization of adaptive control algorithms, but they are stronger than needed. In general, the parameters $\ha(t)$ found by an adaptive controller need not converge to a constant despite the fact that $\dot{\ha} \rightarrow 0$\footnote{Lyapunov function arguments based on a parameter estimation error term generally lead to the conclusion that the parameters remain bounded. It is also often the case that $\dot{\ha} \rightarrow 0$ as it is driven by an error term. Nevertheless, $\ha$ may stay time-varying for all $t$. For instance, the function $f(t) = \sin(\sqrt{t})$ remains bounded and time-varying for all $t$, but has $f'(t) = \frac{1}{2\sqrt{t}}\cos(\sqrt{t})\rightarrow 0$. A sufficient condition (by Barbalat's Lemma [Lemma~\ref{lem:barbalat}]) for convergence to a constant $\ha_\infty$ is that $\left(\ha-\ha_\infty\right)\in\LL_p$ for some $p$.}. Similarly, even in the case that the parameters converge, it is not strictly required that $\bY(\bx(t), t)\ha_\infty = \bff(\bx(t), \ba, t)$ along the entire trajectory, as this condition is satisfied asymptotically. Numerical simulations in Section~\ref{sec:sim} will demonstrate the implicit regularization of parameters $\ha(t)$ found by adaptive control along the entire trajectory.$\qeddef$
\end{rmk}

We may make a similar claim in the nonlinearly parameterized setting captured by Assumption~\ref{assmp:tyukin}. To do so, we require an additional assumption.
\begin{assumption}
\label{assmp:invert}
For any vector of parameters $\bthet$ and the true parameters $\ba$, $f(\bx(t), \bthet, t) = f(\bx(t), \ba, t)$ implies that $\balf(\bx(t), t)^\T\bthet = \balf(\bx(t), t)^\T\ba$.$\qeddef$
\end{assumption}
For the class of systems (\ref{eqn:single_layer}), a sufficient condition for Assumption~\ref{assmp:invert} is that $\lambda(\bx(t), t) \neq 0$ and that the map $\bphi(\bx, t)^\T\ba \rightarrow f_m(\bx(t), \bphi(\bx, t)^\T\ba)$ is invertible at every $t$. We may now state our implicit regularization result for nonlinearly parameterized systems.
\begin{prop}
\label{prop:tyukin_reg}
Consider the adaptation algorithm
\begin{equation}
    \label{eqn:tyukin_natural}
    \dot{\ha} = -\left[\nabla^2\psi(\ha)\right]^{-1}\tilde{f}(\bx, \ha, \ba, t)\balf(\bx, t)
\end{equation}
under Assumptions~\ref{assmp:tyukin} \& \ref{assmp:invert}. Assume $\ha(t)\rightarrow\ha_\infty\in\mathcal{A}$. Then $\ha_\infty = \arg\min_{\bthet \in \mathcal{A}}\bregd{\bthet}{\ha(0)}$.
\end{prop}
The proof is similar to the linearly parameterized case and is given in Appendix~\ref{app2:tyukin_reg}.
\begin{rmk}
\label{rmk:tyukin_imp}
Algorithm (\ref{eqn:tyukin_natural}) must be implemented in PI form due to the appearance of $\tilde{f}(\bx, \ha, \ba, t)$, but the use of the PI form~(\ref{eqn:tyukin_pi_1})-(\ref{eqn:tyukin_pi_4}) in $\ha$ is complicated by the presence of the inverse Hessian of $\psi$. To implement (\ref{eqn:tyukin_alg}), the Euclidean variant may be implemented through the usual PI form for an auxiliary variable $\dot{\hv} = -\tilde{f}(\bx, \ha, \ba, t)\balf(\bx, t)$, and then the controller parameters may be computed by inverting the gradient of $\psi$, $\ha(t) = \left(\nabla\psi^{-1}\right)(\hv(t))$. Concretely, the identity $\frac{d}{dt}\nabla\psi(\ha) = \nabla^2\psi(\ha)\dot{\ha}$ shows that $\dot{\ha} = -\left(\nabla^2\psi(\ha)\right)^{-1}\tilde{f}(\bx, \ha, \ba, t)\balf(\bx, t)$ is equivalent to $\frac{d}{dt}\nabla\psi(\ha) = -\tilde{f}(\bx, \ha, \ba, t)\balf(\bx, t)$. The auxiliary variable $\hv$ can then be identified with $\nabla\psi(\ha)$.$\qeddef$
\end{rmk}
Propositions~\ref{prop:implicit_reg} \& \ref{prop:tyukin_reg} demonstrate the implicit bias of adaptive control algorithms. In doing so, they identify an additional design choice that may be exploited for the application of interest. Proposition~\ref{prop:implicit_reg} implies that the Slotine and Li controller, when initialized with the parameters at $\ha(0) = \mathbf{0}$, finds the interpolating parameter vector of minimum $l_2$ norm. Other norms, such as the $l_p$ norm with $p > 1$ will find alternative parameter vectors that may have desirable properties such as sparsity\footnote{Because the $l_1$ norm is not strongly convex, it may be replaced with a suitable approximation such as the $l_{1+\epsilon}$ norm for $\epsilon > 0$ and small~\citep{azizan_1, azizan_2}.}. The usual Euclidean geometry-based adaptive laws can be seen as a form of ridge regression, while imposing $l_1$, $l_2$ and $l_1$ simultaneously, or $l_p$ regularization through the choice of $\psi$ can be seen as the adaptive control equivalents of LASSO~\citep{LASSO} or compressed sensing, elastic net, and bridge regression respectively. In the context of adaptive control, this notion of implicit regularization is particularly interesting, as typical regularization terms such as $l_1$ and $l_2$ penalties cannot in general be added to the adaptation law directly without affecting stability and performance of the algorithm.

\subsection{Non-Euclidean measure of the tracking error}
The usual Lyapunov function incorporates a Euclidean tracking error term given by $\frac{1}{2}s(\bx, \bx_d(t))^2$. In a similar vein to the derivation of the ``natural'' adaptive laws, for any strongly convex function $\phi:\mathbb{R}\rightarrow\mathbb{R}$, we may instead replace this tracking error term by the Bregman divergence $\bregdphi{0}{s(\bx, \bx_d(t))}$. This quantity has time derivative
\begin{equation*}
    \frac{d}{dt}\bregdphi{0}{s(\bx, \bx_d(t))} = -\eta s(\bx, \bx_d(t))^2 \phi''(s(\bx, \bx_d(t))) + \phi''(s(\bx, \bx_d(t))) \bY(\bx, t)\tilde{\ba}
\end{equation*} 
in the linearly parameterized case. Because $\phi''(s) \geq 0$ for all $s$ for $\phi$ strongly convex, it is simple to see that this modification to the usual Lyapunov function in combination with a non-Euclidean measure of the parameter estimation error leads to a family of stable adaptation laws parameterized by $\phi$ and $\psi$ of the form $\dot{\ha} = -\left[\nabla^2\psi(\ha)\right]^{-1}\bY(\bx, t)^\T\phi''(s(\bx, \bx_d(t)))s(\bx, \bx_d(t))$. This shows, for example, that any odd power of $s$ may be stably employed in the adaptation law by taking $\phi(s) = s^p$ for some even power $p$. Surprisingly, more exotic adaptation laws such as $\dot{\ha} = -\left[\nabla^2\psi(\ha)\right]^{-1}\bY(\bx, t)^\T e^{\lambda |s(\bx, \bx_d(t))|}s(\bx, \bx_d(t))$ for $\lambda > 0$ may also be used.

In the single-input case, these laws could be more simply obtained by replacing the $\frac{1}{2}s(\bx, \bx_d(t))^2$ term in the Lyapunov-like function with a term of the form $g(s(\bx, \bx_d(t))$ where $g'(s)s \geq 0$ and $g'(s)$ is known. In the multi-input case, these two approaches differ. Taking $g$ to be a strongly convex function with minimum attained at $s=0$ and a known gradient, the Lyapunov-like function 
\begin{equation*}
    V(\bx, \hat{\ba}, t) = g\left(\bs\left(\bx, \bx_d\left(t\right)\right)\right) - \inf_\bs g\left(\bs\left(\bx, \bx_d\left(t\right)\right)\right) + \bregd{\ba}{\ha}
\end{equation*} 
shows that the adaptation law
\begin{equation*}
    \dot{\ha} = -\left[\nabla^2\psi(\ha)\right]^{-1}\bY(\bx, t)^\T\nabla g(\bs(\bx, \bx_d(t)))
\end{equation*}
is globally convergent. On the other hand, the Lyapunov-like function 
\begin{equation*}
    V(\bx, \ha, t) = \bregdphi{0}{\bs(\bx, \bx_d(t))} + \bregd{\ba}{\ha}
\end{equation*}
shows that the distinct adaptation law
\begin{equation*}
    \dot{\ha} = -\left[\nabla^2\psi(\ha)\right]^{-1}\bY(\bx, t)^\T\left[\nabla^2\phi(\mathbf{s}(\bx, \bx_d(t)))\right]\mathbf{s}(\bx, \bx_d(t))
\end{equation*} 
is also globally convergent.

\section{Adaptive dynamics prediction, control, and observer design}
\label{sec:obs_predict}
In this section, we demonstrate how the new non-Euclidean adaptation laws of Section~\ref{ssec:imp_reg} may be used for regularized dynamics prediction, regularized adaptive control, and regularized observer design.
\subsection{Regularized adaptive dynamics prediction}
\label{ssec:dyn_predict}
Similar to direct adaptive control, online parameter estimation may also be used within an observer-like framework for dynamics prediction. This enables, for instance, the design of provably stable online learning rules for the weights of a recurrent neural network in the dynamics approximation context~\citep{force, deneve, gerst}. Consider a nonlinear dynamical system
\begin{equation}
    \label{eqn:dyn_predict_gen}
    \dot{\bx} = \bff(\bx) + \bc(t),
\end{equation}
where $\bx\in\mathbb{R}^n$ is the system state, $\bff : \mathbb{R}^n\rightarrow\mathbb{R}^n$ is the system dynamics and $\bc : \mathbb{R}_+ \rightarrow \mathbb{R}^n$ is a system input. Define the observer-like system
\begin{equation*}
    \dot{\hat{\bx}} =  \bY(\hat{\bx})\hat{\ba} + \bg\left(\hat{\bx}, \bx(t)\right) + \bc(t),
\end{equation*}
where $\bY:\mathbb{R}^n \rightarrow \mathbb{R}^{n\times p}$, $\hat{\ba} \in \mathbb{R}^p$, and $\bg\left(\hat{\bx}, \bx(t)\right)$ is a feedback term incorporating the measurement $\bx(t)$ which satisfies $\bg(\bx, \bx) = 0$ for all $\bx\in\mathbb{R}^n$. Assume that there exists a fixed parameter vector $\ba$ such that for all $\bx \in \mathbb{R}^n$, $\bY(\bx)\ba = \bff(\bx)$. Then we may state the following proposition.
\begin{prop}
\label{prop:dyn_predict_gen}
Consider the observer-like system \eqref{eqn:dyn_predict_gen} along with adaptation law
\begin{equation}
    \label{eqn:adapt_dyn_predict}
    \dot{\ha} = -\gamma\left[\nabla^2\psi(\ha)\right]^{-1}\bY(\hat{\bx})^\T\bGam\left(\hat{\bx} - \bx(t)\right)
\end{equation}
for $\gamma > 0$ a fixed learning rate, $\Gamma \in \mathbb{R}^{n\times n}$ a symmetric positive definite matrix, and $\psi$ strongly convex. Let $\frac{\p \bg}{\p \bx}(\bx_1, \bx_2)$ denote the derivative of $\bg$ with respect to its first argument evaluated at $(\bx_1, \bx_2)$. Then $\hat{\bx}(t)\rightarrow\bx(t)$ if $\bff(\bx) + \bg\left(\bx, \bu(t)\right)$ is contracting in the metric $\bGam$ for any external input $\bu(t)$~\citep{contraction, modular}, i.e., if
\begin{equation*}
    \left(\frac{\p \bff}{\p \bx}(\bx) + \frac{\p \bg}{\p \bx}\left(\bx, \bu(t)\right)\right)^\T\bGam + \bGam\left(\frac{\p \bff}{\p \bx}(\bx) + \frac{\p \bg}{\p \bx}\left(\bx, \bu(t)\right)\right) \leq -2\lambda\bGam
\end{equation*}
uniformly in $\bx$ for some contraction rate $\lambda > 0$.
\end{prop}
The proof is given in Appendix~\ref{app2:dyn_predict_gen}. The implicit regularization results of Section~\ref{ssec:imp_reg} show that this framework provides a technique for provably regularizing learned predictive dynamics models without negatively impacting stability or convergence of the combined error and parameter estimation systems.

Proposition~\ref{prop:dyn_predict_gen} represents a separation theorem for adaptive dynamics prediction. If a dynamics predictor can be designed under the assumption that the true system dynamics is known -- e.g., if bounds on $\frac{\p \bff(\bx)}{\p \bx}$ are available -- then the same dynamics predictor can be made adaptive by incorporating the skew-symmetric law (\ref{eqn:adapt_dyn_predict}). Convergence properties then only depend on the nominal system with control feedback and are independent of the parameter estimator, as shown by the conditions for contraction. 

\begin{rmk}
In principle, these simple results could be made more general using the techniques developed in~\citet{brett}, or could be performed in a latent space computed via a nonlinear dimensionality reduction technique such as an autoencoder~\citep{brunton}. This could also extend to adaptive control, for example in robot control applications where an adaptive controller could be designed in a latent space computed from raw pixels via a neural network.
\end{rmk}

\subsection{Regularized dynamics prediction for Hamiltonian systems}
\label{ssec:ham}
If the underlying system is known to have a specific structure, this structure may be leveraged in a principled way to adaptively compute models for dynamics prediction~\citep{sanner_neco}. For example, large classes of physical systems can be well-described via Hamiltonian dynamics,
\begin{align*}
    \dot{\bp} &= -\nabla_{\bq}\HH(\bp, \bq),\\
    \dot{\bq} &= \nabla_{\bp}\HH(\bp, \bq),
\end{align*}
where $\HH(\bp, \bq)$ is the system Hamiltonian, $\bp$ is the generalized momentum, and $\bq$ is the generalized coordinate conjugate to $\bp$. This structure was exploited in recent work
by~\citet{bottou} via direct estimation of the system Hamiltonian with a deep feedforward network in combination with symplectic integration of the resulting dynamics. In a similar spirit, rather than parameterizing the system dynamics as in Section~\ref{ssec:dyn_predict}, we now estimate the scalar Hamiltonian itself as a linear expansion in a set of known nonlinear basis functions $\{Y_k\}$.
\begin{prop}
\label{prop:ham_estimate}
Consider direct estimation of a system's scalar Hamiltonian
\begin{equation*}
    \widehat{\HH}(\hat{\ba}, \bp, \bq) = \sum_k \hat{a}_k Y_k(\bp, \bq) = \bY(\bp, \bq)\ha,
\end{equation*}
where $\bY(\bp, \bq) \in \mathbb{R}^{1\times p}$. Assume that there exists some true parameter vector $\ba \in\mathbb{R}^p$ such that $\HH = \bY(\bp, \bq)\ba$. Consider the induced dynamics prediction model 
\begin{align}
    \label{eqn:ham_pred_p}
    \dot{\hp} &= -\left(\nabla_{\hq}\bY(\hp, \hq)\right)\ha + k_p\left(\bp(t) - \hp\right),\\
    \label{eqn:ham_pred_q}
    \dot{\hq} &= \left(\nabla_{\hp}\bY(\hp, \hq)\right)\ha + k_q\left(\bq(t) - \hq\right).
\end{align}
for $k_p > 0$, $k_q > 0$ feedback gains, along with the adaptation law
\begin{equation*}
    \dot{\ha} = \gamma\left[\nabla^2\psi(\ha)\right]^{-1}\left(\left(\nabla_{\hq}\bY(\hp, \hq)\right)^\T\tilde{\bp}(t) - \left(\nabla_{\hp}\bY(\hp, \hq)\right)^\T\tilde{\bq}(t)\right),
\end{equation*}
with $\gamma > 0$ a positive learning rate and $\psi$ strongly convex. Then $\hat{\bp}(t)\rightarrow \bp(t)$ and $\hat{\bq}(t) \rightarrow \bq(t)$ if 
\begin{align*}
    k_p &> -\frac{1}{2}\lambda_{\min}\left(\nabla_{\bp}\nabla_{\bq}\HH(\bp, \bq) + \nabla_{\bq}\nabla_{\bp}\HH(\bp, \bq)\right),\\
    k_q &> \frac{1}{2}\lambda_{\max}\left(\nabla_{\bp}\nabla_{\bq}\HH(\bp, \bq) + \nabla_{\bq}\nabla_{\bp}\HH(\bp, \bq)\right),\\
    \lambda_p\lambda_q &> \frac{1}{4}\lambda_{\max}^2\left[\nabla^2_{\bp}\HH(\bp, \bq) - \nabla^2_{\bq}\HH(\bp, \bq)\right],
\end{align*}
where $\lambda_p > 0$ and $\lambda_q > 0$ are given by
\begin{align*}
    \lambda_p &= k_p + \frac{1}{2}\lambda_{\min}\left(\nabla_{\bp}\nabla_{\bq}\HH(\bp, \bq) + \nabla_{\bq}\nabla_{\bp}\HH(\bp, \bq)\right)\\
    \lambda_q &= k_q - \frac{1}{2}\lambda_{\max}\left(\nabla_{\bp}\nabla_{\bq}\HH(\bp, \bq) + \nabla_{\bq}\nabla_{\bp}\HH(\bp, \bq)\right).
\end{align*}
\end{prop}
The proof is given in Appendix~\ref{app2:ham_estimate}. The above predictor employs parameter sharing between both dynamics due to the direct estimation of the system Hamiltonian. The basis functions for the individual dynamics reflect the symplectic structure of the true dynamics, as they are given by partial derivatives of the basis functions for the Hamiltonian.

Rather than a general Hamiltonian $\HH = \HH(\bp, \bq)$, it is common to have a separable Hamiltonian structure with $T(\bp)$ denoting the kinetic energy and $U(\bq)$ the potential energy,
\begin{equation*}
    \HH(\bp, \bq) = T(\bp) + U(\bq).
\end{equation*}
The conditions for successful estimation then simplify to
\begin{equation}
    \label{eqn:contr_sep}
    k_qk_p > \frac{1}{4}\lambda_{\max}^2\left(\nabla_{\hp}^2T(\hp) - \nabla_{\hq}^2 U(\hq)\right).
\end{equation}
The results of Section~\ref{ssec:imp_reg} show that the choice of $\psi$ may be used to regularize the estimate of the Hamiltonian, and in turn, the dynamics. This may be used, for instance, for parsimonious Hamiltonian estimation through the combination of a rich set of physically motivated scalar basis functions and a sparse representation obtained via $l_1$ regularization, similar to~\citet{brunton}. Further results that exploit the structure of separable Hamiltonians through independent estimation of the kinetic and potential energies are presented in Appendix~\ref{app2:ham}.

\subsection{Regularized adaptive control for Lagrangian systems}
\label{ssec:lag}
A similar methodology can be applied to parameterize a scalar Lagrangian rather than Hamiltonian, leading to a second order differential equation with inertia matrix, centripetal and Coriolis forces, and potential energy parameterized by a shared set of weights. As we now show, generalizing the derivation of the Slotine and Li robot controller~\citep{slot_li_robot} to this setting allows for stable adaptive control of Lagrangian systems by direct estimation of the Lagrangian itself. Consider the Lagrangian
\begin{equation*}
    \LL = \frac{1}{2}\dot{\bq}^\T\bH(\bq)\dot{\bq} - U(\bq),
\end{equation*}
with $\bH(\bq)$ an unknown inertia matrix and $U(\bq)$ an unknown potential. Assume that the inertia matrix and scalar potential are given exactly by an expansion in physically- or mathematically-motivated basis functions. That is, for a set of positive definite matrices $\bM^l > 0$ and scalar functions $\phi^l$,
\begin{align*}
    \bH(\bq) &= \sum_{l=1}^{p^{(K)}} a_l^{(K)}\bM^{l}(\bq),\\
    U(\bq) &= \sum_{l=1}^{p^{(P)}} a_l^{(P)}\phi^l(\bq),
\end{align*}
where superscript $(K)$ and $(P)$ denote kinetic and potential, respectively, and the vectors $\ba^{(K)}$ and $\ba^{(P)}$ are unknown. We may then state the following proposition.
\begin{prop}
\label{prop:lagr_est}
Consider direct estimation of a system Lagrangian through expansion of the inertia matrix and potential energy in known basis functions
\begin{align*}
    \widehat{\bH}(\bq) &= \sum_l^{p^{(K)}} \hat{a}_l^{(K)}\bM^{l}(\bq),\\
    \widehat{U}(\bq) &= \sum_l^{p^{(P)}} \hat{a}_l^{(P)}\phi^l(\bq),\\
    \widehat{\LL}(\bq, \dot{\bq}) &= \frac{1}{2}\dot{\bq}^\T \widehat{\bH}(\bq)\dot{\bq} - \widehat{U}(\bq).
\end{align*}
Consider the induced dynamics predictor formed by the Euler-Lagrange equations 
\begin{equation*}
    \frac{d}{dt}\frac{\partial \widehat{\LL}(\bq, \dot{\bq})}{\partial \dot{\bq}} - \frac{\partial \widehat{\LL}(\bq, \dot{\bq})}{\partial \bq} = \bu(\bq, \dot{\bq}, \bq_d(t), \dot{\bq}_d(t))
\end{equation*}
with $\bu(\bq, \dot{\bq}, \bq_d(t), \dot{\bq}_d(t)) = -\bK\bs(\bq, \dot{\bq}, \bq_d(t), \dot{\bq}_d(t)) + \bY^{(p)}(\bq)\ha^{(P)} + \bY^{(K)}(\bq, \dot{\bq}, \bq_d(t), \dot{\bq}_d(t))\ha^{(K)}$ for $\bq_d(t)$ a desired trajectory, $\bK > 0$ a positive definite feedback gain, $\bs(\bq, \dot{\bq}, \bq_d(t), \dot{\bq}_d(t)) = \left(\frac{d}{dt} + \lambda\right)\left(\bq - \bq_d(t)\right)$, and $\bY^{(P)}(\bq)$ and $\bY^{(K)}(\bq, \dot{\bq}, \bq_d(t), \dot{\bq}_d(t))$ given by
\begin{align*}
    Y^{(P)}_{il}(\bq) &= \frac{\p\phi^l(\bq)}{\p q_i},\\
    Y^{(K)}_{il}(\bq, \dot{\bq}, \bq_d(t), \dot{\bq}_d(t)) &= \sum_{kj} \frac{1}{2}\left[\frac{\p M^l_{ij}(\bq)}{\p q_k} - \left(\frac{\p M^l_{kj}(\bq)}{\p q_i} - \frac{\p M^l_{ki}(\bq)}{\p q_j}\right)\right]\dot{q}_k\dot{q}_{r, j} + \sum_{j}M_{ij}^l(\bq)\ddot{q}_{r, j},
\end{align*}
with $\ddot{\bq}_r = \bq_d - \lambda \left(\bq - \bq_d(t)\right)$. Then, for $\gamma_K > 0, \gamma_P > 0$ positive learning rates and $\psi^{(K)}, \psi^{(P)}$ strongly convex functions, the adaptation laws
\begin{align*}
    \dot{\ha}^{(K)} &= -\gamma_K\nabla^2\psi^{(K)}\left(\ha^{(K)}\right)^{-1}\bY^{(K)}(\bq, \dot{\bq}, \bq_d(t), \dot{\bq}_d(t))^\T\bs,\\
    \dot{\ha}^{(P)} &= -\gamma_P\nabla^2\psi^{(P)}\left(\ha^{(P)}\right)^{-1}\bY^{(P)}(\bq)^\T\bs,
\end{align*}
ensure that $\bq(t) \rightarrow \bq_d(t)$. Above, $\bs = \bs(\bq, \dot{\bq}, \bq_d(t), \dot{\bq}_d(t))$.
\end{prop}
The proof is given in Appendix~\ref{app2:lagr_est}. As in Section~\ref{ssec:ham}, by using an $l_1$ approximation for $\psi$, this approach may find sparse, interpretable models of the kinetic and potential energies. Estimating the potential energy directly may in some cases lead to simpler parameterizations than estimating the resulting forces.

If more structure in the inertia matrix is known, for example that it depends only on a few unknown parameters, it may still be approximated using the usual Slotine and Li controller. The external forces can then be estimated by directly estimating the corresponding potential that generates them.

\subsection{Regularized adaptive observer design}
In many physical and engineering systems, only a low-dimensional output of the system $\by(\bx) \in \mathbb{R}^m$ is available for measurement. Assuming that $\by(\bx) = \bC\bx$ is a rank-$m$ linear readout for some known matrix $\bC \in \mathbb{R}^{m\times n}$ with $m < n$, we now show that the tools of the previous two subsections can be used to design regularized adaptive observers for the full system state. Assume that the true system dynamics satisfies
\begin{equation*}
    \dot{\bx} = \bff(\bx) + \bc(t) = \bY(\by(\bx))\ba + \bc(t),
\end{equation*}
with $\ba\in\mathbb{R}^p$ a vector of unknown parameters, and where the known regressor matrix $\bY\in\mathbb{R}^{n\times p}$ only depends on the system output $\by(\bx)$. We may now state the following proposition.
\begin{prop}
\label{prop:adaptive_observer}
Consider the adaptive observer
\begin{align*}
    \dot{\hat{\bx}} &= \bY(\hat{\by})\ha + \bc(t) + \bg(\hat{\by}, \by(t)),\\
    \dot{\ha} &= -\gamma\nabla^2\psi(\ha)^{-1}\bY(\hat{\by})^\T\bC^\T\bGam\tilde{\by}(t),
\end{align*}
with $\gamma > 0$ a positive learning rate, $\hat{\by} = \by(\hat{\bx})$, $\psi:\mathbb{R}^p\rightarrow\mathbb{R}$ a strongly convex potential function, $\bGam$ a positive definite matrix, and $\bg:\mathbb{R}^m\times\mathbb{R}^m\rightarrow\mathbb{R}^n$ an arbitrary function satisfying $\bg(\by, \by) = \mathbf{0}$ for all $\by\in\mathbb{R}^m$. Let $\frac{\p \bY\ba}{\p \by}(\by)$ denote the derivative of $\bY(\by)\ba$ with respect to its argument evaluated at $\by$ and let $\frac{\partial \bg}{\p \by}(\by_1, \by_2)$ denote the derivative of $\bg$ with respect to its first argument evaluated at $(\by_1, \by_2)$. Then $\hat{\bx}(t)\rightarrow \bx(t)$ if
\begin{equation*}
    \bGam\bC\left(\frac{\p \bY\ba}{\p \by}(\by) + \frac{\p \bg}{\p \by}\left(\by, \bu(t)\right)\right) + \left(\frac{\p \bY\ba}{\p \by}(\by) + \frac{\p \bg}{\p \by}\left(\by, \bu(t)\right)\right)^\T\bC^\T\bGam \leq -\lambda\bGam
\end{equation*}
for some $\lambda > 0$ and any external input $\bu(t)$.
\end{prop}
The proof is given in Appendix~\ref{app2:adaptive_observer}. The requirement of the theorem is equivalent to contraction of the unknown output dynamics
\begin{equation*}
    \dot{\by} = \bC\bY(\by)\ba + \bC\bg(\by, \bu(t)) + \bC\bc(t)
\end{equation*}
in the metric $\bGam$ for any external input $\bu(t)$. A natural choice of $\bg(\hat{\by}, \by(t))$ to satisfy this condition with $\bGam = \bI$ is $\bg(\hat{\by}, \by(t)) = -k\bC^\T\left(\hat{\by} - \by(t)\right)$ for some $k > 0$.

As in Section~\ref{ssec:dyn_predict}, Proposition~\ref{prop:adaptive_observer} demonstrates a separation theorem for adaptive observer design. If an observer can be designed for the true system with unknown parameters, then the same observer can be made adaptive by incorporating the adaptation law presented in this section. Convergence properties then depend only on the true system with feedback, and are independent of the parameter estimator. The results of Section~\ref{ssec:imp_reg} show that the choice of $\psi$ can be used to regularize the observer model while maintaining provable reconstruction of the full system state.

\subsection{Regularized dynamics prediction for recurrent neural networks}
\label{ssec:rec_net}
Consider the continuous-time recurrent neural network model
\begin{equation}
    \label{eqn:rec_net}
    \tau\dot{\bx} = -\bx + \bsig\left(\bThet\bx\right) 
\end{equation}
with $\bx\in\mathbb{R}^n$ a vector of neuron firing rates, $\bThet\in\mathbb{R}^{n\times n}$ the synaptic weights, $\bsig\left(\bThet\bx\right)$ the post-synaptic potentials, and $\tau > 0$ a relaxation timescale. Let $\bsig(\cdot)$ be an elementwise Lipschitz and monotonic activation function, i.e., 
\begin{align*}
    \bsig(\bx)_i &= \sigma_i(x_i),\\
    |\sigma_i(x) - \sigma_i(y)| &\leq L_i|x - y|,\\ 
    (x-y)\left(\sigma_i(x) - \sigma_i(y)\right) &\geq 0.
\end{align*}
These requirements are satisfied by common activation functions such as the ReLU, softplus, $\tanh$, and sigmoid. We may now state the following proposition.
\begin{prop}
\label{prop:rec_net}
Consider the regularized adaptive dynamics predictor for (\ref{eqn:rec_net}),
\begin{align}
    \label{eqn:rec_net_predict}
    \tau\dot{\hat{\bx}} &= -\hat{\bx} + \bsig\left(\widehat{\bThet}\bx(t)\right) + k\left(\bx(t) - \hat{\bx}\right),\\
    \dot{\widehat{\bThet}} &= -\gamma\left(\nabla^2\psi\left(\widehat{\bThet}\right)\right)^{-1}\left(\bsig\left(\widehat{\bThet}\bx(t)\right) - \bsig\left(\bThet\bx(t)\right)\right)\bx(t)^\T,
    \label{eqn:rec_net_adaptive}
\end{align}
with $\psi$ a strongly convex function on $n\times n$ matrices or vectors in $\mathbb{R}^{n^2}$ and $\gamma > 0$ a positive gain.
\end{prop}
The proof is given in Appendix~\ref{app2:rec_net}. In (\ref{eqn:rec_net_predict}), the true vector of firing rates $\bx$ is used underneath the application of $\bsig(\cdot)$ in the $\dot{\hat{\bx}}$ dynamics. This approach could be used, for example, for identifying regularized low-dimensional models in computational neuroscience. Our results are similar to those of~\citet{foster2020learning}, but handle a mirror descent or natural gradient extension valid in the continuous-time deterministic setting.

The adaptation law (\ref{eqn:rec_net_adaptive}) cannot be implemented directly through a PI form. However, it can be well-approximated, for example by the PI construction
\begin{align*}
    \dot{\bar{\bx}} &= \lambda \left(\bx(t) - \bar{\bx}\right),\\
    \nabla\psi\left(\widehat{\bThet}\right) &= \gamma\left(\overline{\bThet} - \tilde{\bx}(t)\bar{\bx}^\T\right),\\
    \dot{\overline{\bThet}} &= -\left(k+1\right)\tilde{\bx}(t)\bar{\bx}^\T + \lambda\tilde{\bx}(t)\left(\bx(t) - \bar{\bx}\right)^\T,
\end{align*}
for $\lambda > 0$ a positive gain ensuring $\bar{\bx} \approx \bx$.

\section{The Bregman Lagrangian and accelerated optimization algorithms}
\label{ssec:breg}
In~\citet{ashia_1}, the \textit{Bregman Lagrangian} was shown to generate a suite of accelerated optimization algorithms in continuous-time by appealing to the Euler Lagrange equations through the principle of least action. In its original form, the Bregman Lagrangian is given by
\begin{equation}
    \mathcal{L}(\bx, \dot{\bx}, t) = e^{\overline{\alpha}(t) + \overline{\gamma}(t)}\left(\bregd{\bx + e^{-\overline{\alpha}(t)}\dot{\bx}}{\bx} - e^{\overline{\beta}(t)}f(\bx)\right).
    \label{eqn:breg_lag}
\end{equation}
In (\ref{eqn:breg_lag}), $f(\bx)$ is the loss function to be optimized and $\psi(\bx)$ is a strongly convex function. We will take $\psi(\cdot) = \frac{1}{2}\Vert\cdot\Vert_2^2$ in Section~\ref{sec:main}, and will consider extensions to arbitrary $\psi$ in Section~\ref{sec:ho_nat}.

The quantities $\overline{\alpha}:\mathbb{R}_+ \rightarrow \mathbb{R}, \overline{\beta}:\mathbb{R}_+ \rightarrow \mathbb{R}$, and $\overline{\gamma}:\mathbb{R}_+ \rightarrow \mathbb{R}$ in (\ref{eqn:breg_lag}) are arbitrary time-dependent functions that will ultimately set the damping and learning rates in the second-order Euler Lagrange dynamics. To generate accelerated optimization algorithms, Wibisono, Wilson, and Jordan required two \textit{ideal scaling conditions}: $\dot{\overline{\beta}}(t) \leq e^{\overline{\alpha}(t)}$ and  $\dot{\overline{\gamma}}(t) = e^{\overline{\alpha}(t)}$. These conditions originate from the Euler Lagrange equations, where the second is used to eliminate an unwanted term, and a Lyapunov argument, where the first is used to ensure decrease of a chosen Lyapunov function.

\citet{gaudio_1} recently utilized the Bregman Lagrangian to derive a momentum-like adaptive control algorithm. To do so, they defined $\overline{\alpha}(t) = \log(\beta \sN(t))$, $\overline{\beta}(t) = \log\left(\frac{\gamma}{\beta \sN(t)}\right)$, and $\overline{\gamma}(t) = \int_0^t e^{\overline{\alpha}(t')}dt'$\footnote{Note that these conditions validate the second ideal scaling condition but not the first. As mentioned above, the first ideal scaling condition is required only by the choice of Lyapunov function in the original work, which was used to derive convergence rates for optimization algorithms~\citep{ashia_1}. In this sense, it is not strictly required for adaptive control.}. Here, $\gamma\geq 0$ and $\beta\geq 0$ are non-negative scalar hyperparameters and $\sN = \sN(t)$ is a time-dependent signal chosen based on the system. With these definitions, choosing the Euclidean norm $\psi(\cdot) = \frac{1}{2}\Vert\cdot\Vert^2$, and modifying the Bregman Lagrangian presented in~\citet{gaudio_1} to the adaptive control framework defined in Section~\ref{sec:adaptive}, (\ref{eqn:breg_lag}) becomes
\begin{equation}
    \mathcal{L}\left(\hat{\ba}, \dot{\hat{\ba}}, t\right) = e^{\int_0^t \beta \sN(t') dt'}\frac{1}{\beta
    \sN(t)}\left(\frac{1}{2}\dot{\hat{\ba}}^\T\dot{\hat{\ba}} - \gamma \beta \sN(t)\frac{d}{dt}\left[\frac{1}{2}s(\bx(t), \bx_d(t))^2\right]\right).
    \label{eqn:breg_adaptive}
\end{equation}
Comparing (\ref{eqn:breg_lag}) and (\ref{eqn:breg_adaptive}), it is clear that the loss function $f(\bx)$ in (\ref{eqn:breg_lag}) has been replaced by $\frac{d}{dt}\frac{1}{2}s(\bx(t), \bx_d(t))^2$ in (\ref{eqn:breg_adaptive}). Following Remark~\ref{rmk:slotine_li}, this is precisely the $\dot{Q}(\bx, \hat{\ba}, t)$ velocity gradient functional that gives rise to the Slotine and Li controller. For (\ref{eqn:breg_adaptive}), the Euler-Lagrange equations lead to the adaptation law
\begin{equation}
    \ddot{\hat{\ba}} + \dot{\hat{\ba}}\left(\beta \sN(t) - \frac{\dot{\sN}(t)}{\sN(t)}\right) = -\gamma\beta \sN(t) \bY(\bx, t)^\T s(\bx, \bx_d(t)).
    \label{eqn:anna}
\end{equation}
(\ref{eqn:anna}) may be understood as a modification of the Slotine and Li adaptive controller to incorporate momentum and time-dependent damping. (\ref{eqn:anna}) may also be re-written as two first-order systems which are useful for proving stability.
\begin{align}
    \label{eqn:vdot_basic}
    \dot{\hat{\bv}} &= - \gamma \bY(\bx, t)^\T s(\bx, \bx_d(t))\\
    \label{eqn:adot_basic}
    \dot{\hat{\ba}} &= \beta \sN(t) \left(\hat{\bv} - \hat{\ba}\right)
\end{align}
The properties of (\ref{eqn:anna}) are summarized in the following proposition.
\begin{prop}
\label{prop:anna}
Consider the higher-order adaptation algorithm (\ref{eqn:anna}) with $\sN(t) = 1 + \mu\Vert\bY(\bx(t), t)\Vert^2$ and $\mu > \frac{\gamma}{\eta\beta}$. Then, all trajectories $(\bx(t), \hat{\bv}(t), \hat{\ba}(t))$ remain bounded, $s(\bx(\cdot), \bx_d(\cdot)) \in \LL_\infty \cap \LL_2$, $\left(\hat{\ba}(\cdot) - \hat{\bv}(\cdot)\right)\in\LL_2$, $s(\bx(t), \bx_d(t)) \rightarrow 0$ and $\bx(t) \rightarrow \bx_d(t)$.
\end{prop}
The proof follows by consideration of the Lyapunov-like function 
\begin{equation*}
V(\bx, \hat{\ba}, \hat{\bv}, t) = \frac{1}{2}\left(s(\bx, \bx_d(t))^2 + \frac{1}{\gamma}\Vert\tilde{\bv}\Vert^2 + \frac{1}{\gamma}\Vert\hat{\bv} - \hat{\ba}\Vert^2\right).
\end{equation*}
\begin{rmk}
\label{rmk:ham}
The transformation to a system of two first-order equations may seem somewhat \textit{ad-hoc}, but it follows immediately by consideration of the non-Euclidean Bregman Lagrangian (\ref{eqn:breg_lag}). Indeed, it is easy to check that $\hv = \ha + \frac{\dot{\ha}}{\beta\sN(t)}$, which is precisely the adaptive control equivalent of $\bx + e^{-\overline{\alpha}(t)}\dot{\bx}$ in the first argument of $\bregd{\cdot}{\cdot}$ in (\ref{eqn:breg_lag}). The transformation is also readily apparent by use of the \textit{Bregman Hamiltonian}
\begin{equation}
    \mathcal{H}(\hat{\ba}, \bp) = \frac{1}{2}\beta \sN(t) e^{-\overline{\gamma}(t)}\Vert\bp\Vert^2 + \gamma e^{\overline{\gamma}(t)}\left[\frac{d}{dt}\frac{1}{2}s(\bx(t), \bx_d(t))^2\right],
    \label{eqn:breg_ham}
\end{equation}
which, via Hamilton's equations, leads to
\begin{align*}
    \dot{\bp} &= -\frac{\p \mathcal{H}}{\p \hat{\ba}} = - \gamma e^{\overline{\gamma}(t)}\bY(\bx, t)^\T s(\bx, \bx_d(t)),\\
    \dot{\hat{\ba}} &= \frac{\p \mathcal{H}}{\p \bp} = \beta \sN(t) e^{-\overline{\gamma}(t)}\bp.
\end{align*}
Defining $\hat{\bv} = e^{-\overline{\gamma}(t)}\bp + \hat{\ba}$ immediately leads to (\ref{eqn:vdot_basic}) \& (\ref{eqn:adot_basic}). This line of reasoning was recently investigated further by~\citet{gaudio_ham}. As is typical in classical mechanics, the Bregman Hamiltonian may be obtained from a Legendre transform of the Bregman Lagrangian. The Hamiltonian dynamics may be useful for discrete-time algorithm development through application of symplectic discretization techniques~\citep{ashia_2, relativistic, symp_disc}.$\qeddef$
\end{rmk}
\begin{rmk}
It is well known, for example from a passivity interpretation of the Lyapunov-like analysis~\citep{slot_li_book}, that the pure integrator in the standard Slotine and Li adaptation law (\ref{eqn:slotine_li}) can be replaced by any linear positive real transfer function containing a pure integrator. The higher-order algorithms presented in this work are distinct from this approach, as most clearly seen by the state-dependent damping term in (\ref{eqn:anna}).$\qeddef$
\end{rmk}
\begin{rmk}
In~\citet{ashia_1}, the suggested Lyapunov function in the Euclidean setting is $V(\bx, \dot{\bx}, t) = \Vert \bx + e^{-\overline{\alpha}(t)}\dot{\bx} - \bx_*\Vert^2 + e^{\overline{\beta}(t)}f(\bx)$ where $\bx_*$ is the global optimum and $f(\bx)$ is the loss function. Noting that $\hv$ is the equivalent of $\bx + e^{-\overline{\alpha}(t)}\dot{\bx}$ in the adaptive control context (see Remark \ref{rmk:ham}), we see that the Lyapunov-like function used to prove stability of the adaptive law (\ref{eqn:anna}) is similar to that used to prove convergence in the optimization context. The loss function term $f(\bx)$ is replaced by $\frac{1}{2}s(\bx, \bx_d(t))^2$, and it is necessary to add the additional term $\frac{1}{\gamma}\Vert\hv-\ha\Vert^2$.
\end{rmk}

\section{Adaptation laws with momentum}
\label{sec:main}
In this section, we develop several new adaptation laws for both linearly and nonlinearly parameterized systems. We begin by noting that the Bregman Lagrangian generates velocity gradient algorithms with momentum, which provides a path towards adaptive control algorithms similar to Nesterov's seminal method. We prove some general conditions under which these momentum algorithms will achieve stable convergence. By analogy with integral velocity gradient functionals, we then derive a proportional-integral scheme to implement a first-order composite adaptation law driven directly by the function approximation error. We subsequently fuse the generating functional for the composite law with the Bregman Lagrangian to construct a composite algorithm with momentum. We then derive momentum algorithms for nonlinearly parameterized systems satisfying the monotonicity requirements in Assumption~\ref{assmp:tyukin}.

\subsection{Velocity gradient algorithms with momentum}
\label{ssec:ho_sg}
As noted in Section~\ref{ssec:breg}, the Bregman Lagrangian (\ref{eqn:breg_adaptive}) that generates the higher-order algorithm (\ref{eqn:anna}) contains the local velocity gradient functional $Q(\bx, t) = \frac{1}{2}s(\bx, \bx_d(t))^2$ that gives rise to the Slotine and Li controller (\ref{eqn:slotine_li}). Based on this observation, we define local and integral higher-order velocity gradient algorithms via the Euclidean Bregman Lagrangian. We begin with the local functional
\begin{equation}
    \label{eqn:breg_lag_sg}
    \mathcal{L}\left(\hat{\ba}, \dot{\hat{\ba}}, t\right) = e^{\int_0^t \beta \sN(t') dt'}\frac{1}{\beta \sN(t)}\left(\frac{1}{2}\dot{\hat{\ba}}^\T\dot{\hat{\ba}} - \gamma \beta \sN(t)\frac{d}{dt}Q(\bx, t)\right),
\end{equation}
which generates the higher-order law
\begin{equation}
    \ddot{\hat{\ba}} + \dot{\hat{\ba}}\left(\beta \sN(t) - \frac{\dot{\sN}(t)}{\sN(t)}\right) = -\gamma\beta \sN(t) \nabla_{\ha}\dot{Q}(\bx, \hat{\ba}, t).
    \label{eqn:ho_local}
\end{equation}
Algorithm (\ref{eqn:ho_local}) can be re-written as two first-order systems
\begin{align}
    \label{eqn:vdot_local}
    \dot{\hat{\bv}} &= - \gamma \nabla_{\ha}\dot{Q}(\bx, \hat{\ba}, t),\\
    \label{eqn:adot_local}
    \dot{\hat{\ba}} &= \beta \sN(t) \left(\hat{\bv} - \hat{\ba}\right).
\end{align}
To drive $Q(\bx(t), t)$ to zero, we require the following technical assumption in addition to Assumptions \ref{assmp:sg1} \& \ref{assmp:sg3}. This assumption replaces Assumption \ref{assmp:sg2} for first-order velocity gradient algorithms.
\begin{assumption}
\label{assmp:local_ho_cond}
There exists a time-dependent signal $N(t)$ and non-negative scalar values $\beta \geq 0, \mu \geq 0$ such that the time-derivative of the goal functional evaluated at the true parameters, $\dot{Q}(\bx, \ba, t)$, satisfies the following inequality for all $\bx, \ba, \ha, \hv$, and $t$
\begin{equation}
    \dot{Q}(\bx, \ba, t) - \frac{\beta\mu}{\gamma}N(t)\Vert\hat{\ba} - \hat{\bv}\Vert^2 + 2\left(\hat{\ba} - \hat{\bv}\right)^\T\nabla_{\ha}\dot{Q}(\bx, \hat{\ba}, t) \leq -\rho(Q(\bx, t)).
    \label{eqn:local_ho_cond}
\end{equation}
In (\ref{eqn:local_ho_cond}), $\rho(\cdot)$ is positive definite, continuous in its argument, and satisfies $\rho(0) = 0$.$\qeddef$
\end{assumption}
Assumption~\ref{assmp:local_ho_cond} is a formal statement that we may ``complete the square'' on the left-hand side of \eqref{eqn:local_ho_cond}. For example, for $\nabla_{\ha}\dot{Q}(\bx, \ha, t) = \bY^\T(\bx, t) s(\bx, \bx_d(t))$ and for $\dot{Q}(\bx, \ba, t) = -\eta s(\bx, \bx_d(t))^2$, we may choose $N(t) = \Vert \bY(\bx(t), t)^\T\Vert^2$. With Assumption \ref{assmp:local_ho_cond} in hand, we can state the following proposition.
\begin{prop}
\label{prop:local_ho_sg}
Consider the algorithm (\ref{eqn:ho_local}) or its equivalent form (\ref{eqn:vdot_local}) \& (\ref{eqn:adot_local}), and assume $Q(\bx, t)$ satisfies Assumptions \ref{assmp:sg1}, \ref{assmp:sg3}, and \ref{assmp:local_ho_cond}. Then, all trajectories $(\bx(t), \hat{\bv}(t), \hat{\ba}(t))$ remain bounded, $\left(\hat{\ba}(\cdot) - \hat{\bv}(\cdot)\right) \in \LL_2$, and $\lim_{t\rightarrow \infty} Q(\bx(t), t) = 0$.
\end{prop}
The proof is given in Appendix~\ref{app2:prop:local_ho_sg}. By taking $Q(\bx, t) = \frac{1}{2}s(\bx, \bx_d(t))^2$ in Proposition~\ref{prop:local_ho_sg}, we immediately recover Proposition~\ref{prop:anna}. 

We now consider the integral functional
\begin{equation*}
    \mathcal{L}\left(\hat{\ba}, \dot{\hat{\ba}}, t\right) = e^{\int_0^t \beta \sN(t') dt'}\frac{1}{\beta \sN(t)}\left(\frac{1}{2}\dot{\hat{\ba}}^\T\dot{\hat{\ba}} - \gamma \beta \sN(t)\frac{d}{dt}\int_0^tR(\bx(t'), \hat{\ba}(t'), t')dt'\right),
\end{equation*}
which generates the higher-order law
\begin{equation}
    \ddot{\hat{\ba}} + \dot{\hat{\ba}}\left(\beta \sN(t) - \frac{\dot{\sN}(t)}{\sN(t)}\right) = -\gamma\beta \sN(t) \nabla_{\ha}R(\bx, \hat{\ba}, t).
    \label{eqn:ho_int}
\end{equation}
We again re-write (\ref{eqn:ho_int}) as two first-order systems
\begin{align}
    \label{eqn:vdot_int}
    \dot{\hat{\bv}} &= - \gamma \nabla_{\ha}R(\bx, \hat{\ba}, t), \\
    \label{eqn:adot_int}
    \dot{\hat{\ba}} &= \beta \sN(t) \left(\hat{\bv} - \hat{\ba}\right),
\end{align}
and now require a modified version of Assumption~\ref{assmp:local_ho_cond}.
\begin{assumption}
\label{assmp:int_ho_cond}
$R(\bx, \hat{\ba}, t) \geq 0$ for all $\bx, \hat{\ba}$, and $t$, and is uniformly continuous in $t$ for bounded $\bx$ and $\ha$. $\nabla_{\ha}R(\bx, \ha, t)$ is locally bounded in $\bx$ and $\ha$ uniformly in $t$. Furthermore, there exists a time-dependent signal $N(t)$ and non-negative scalar values $\beta \geq 0, \mu \geq 0$ such that, for all $\bx, \ba, \ha, \hv$, and $t$,
\begin{equation*}
    R(\bx, \ba, t) - R(\bx, \hat{\ba}, t) - \frac{\beta\mu}{\gamma}N(t)\Vert\hat{\ba}-\hat{\bv}\Vert^2 + 2\left(\hat{\ba}-\hat{\bv}\right)^\T\nabla_{\hat{\ba}}R(\bx, \hat{\ba}, t) \leq -k R(\bx, \hat{\ba}, t)
\end{equation*}
for some constant $k > 0$.$\qeddef$
\end{assumption}

Similar to Assumption~\ref{assmp:local_ho_cond}, Assumption~\ref{assmp:int_ho_cond} is a formal requirement that we may ``complete the square''. With slight abuse of notation, consider the case when $R(\bx, \ha, t) = \frac{1}{2}\tilde{f}(\bx, \ha, \ba, t)^2$. Then $R(\bx, \ba, t) = 0$, $\nabla_{\ha}R(\bx, \ha, t) = \tilde{f}(\bx, \ha, \ba, t)\nabla_{\ha}\hat{f}(\bx, \ha, t)$, and we may choose $N(t) = \Vert \nabla_{\ha}\hat{f}(\bx(t), \ha(t), t)\Vert^2$. With Assumption~\ref{assmp:int_ho_cond}, we can state the following proposition. The proof is given in Appendix~\ref{app2:prop:int_ho}
\begin{prop}
\label{prop:int_ho}
Consider algorithm (\ref{eqn:ho_int}) along with Assumptions \ref{assmp:sg3} \& \ref{assmp:int_ho_cond}. Let $T_x$ denote the maximal interval of existence of $\bx(t)$. Then $\hat{\bv}(t)$ and $\hat{\ba}(t)$ remain bounded for $t\in[0, T_x]$, $\left(\hat{\ba} -\hat{\bv}\right) \in \LL_2$ over this interval, and $\int_0^{T_x} R(\bx(t'), \hat{\ba}(t'), t')dt' < \infty$. Furthermore, for any bounded solution $\bx(t)$, these conclusions hold for all $t$ and $R(\bx(t), \ha(t), t)\rightarrow 0$.
\end{prop}
As mentioned in Section~\ref{ssec:sg}, we will be particularly interested in Proposition~\ref{prop:int_ho} when $R(\bx, \ha, t) = \frac{1}{2}\tilde{f}(\bx, \ha, \ba, t)^2$, which will generate composite adaptation algorithms and algorithms applicable to nonlinearly parameterized systems. Proposition~\ref{prop:int_ho} then shows that $\tilde{f}(\bx(\cdot), \ha(\cdot), \cdot) \in \LL_2$ over the interval of existence of $\bx(t)$. As shown by Lemma~\ref{lem:conv}, with our error model \eqref{eqn:s_dyn}, this is enough to show that $\bx(t)$ always remains bounded and hence $\tilde{f}(\bx(t), \ha(t), t) \rightarrow 0$. 

\begin{rmk}
Classically, Lyapunov-like functions used in adaptive control consist of a sum of tracking and parameter estimation error terms, with $\dot{\hat{\ba}}$ chosen to cancel a term of unknown sign. Several Lyapunov functions in this work consist only of parameter estimation error terms, such as the Lyapunov function used in the proof of Proposition~\ref{prop:int_ho}. From a mathematical point of view, we require that $\dot{V}$ is negative semi-definite and contains signals related to the tracking error. Integrating $\dot{V}$ allows the application of tools from functional analysis to ensure that the control goal is accomplished. The lack of tracking error term in $V$ is the origin of the additional complication that $\bx(t)$ must be shown to be bounded even after it is known that $\dot{V} \leq 0$.$\qeddef$
\end{rmk}

While derived via the Bregman Lagrangian, the adaptation laws presented in this section exhibit a key difference from Nesterov's method in optimization. The gradient updates in $\dot{\hv}$ are evaluated at $\ha$, while Nesterov's method evaluates the gradient at $\hv$. In adaptive control, this forward prediction does not appear to be possible: in the linearly parameterized case, for example, $\nabla_{\ha}\dot{Q}(\bx, \ha, t) = \bY(\bx, t)^\T s(\bx, \bx_d(t))$ is independent of $\ha$.

\subsection{First- and second-order composite adaptation laws}
\label{ssec:comp}
Here we consider the linearly parameterized setting $f(\bx, \ba, t) = \bY(\bx, t)\ba$, and derive new first- and second-order composite adaptation laws based on Propositions~\ref{prop:sg_nonloc}~\&~\ref{prop:int_ho}. Composite adaptation laws are driven by two sources of error: the tracking error itself, as summarized by $s(\bx, \bx_d(t))$ in the Slotine and Li controller, and a prediction error. The prediction error term is generally obtained from an algebraic relation constructed by filtering the dynamics~\citep{slot_li_book}. We now present a composite algorithm that does not require any explicit filtering of the dynamics, but is instead driven simultaneously by $s(\bx, \bx_d(t))$ and $\tilde{f}(\bx, \ha, \ba, t)$.

A starting point is to consider a hybrid local and integral velocity gradient functional
\begin{equation}
    Q(\bx, t) = \frac{\gamma}{2} s(\bx, \bx_d(t))^2 + \frac{\kappa}{2}\int_0^t\tilde{f}^2(\bx(t'), \hat{\ba}(t'), \ba, t')dt',
    \label{eqn:hybrid_sg}
\end{equation}
where $\kappa > 0$ and $\gamma > 0$ are positive learning rates weighting the contributions of each term. As discussed in Section~\ref{ssec:sg}, the first term leads to the Slotine and Li controller. The second can be clearly seen to satisfy Assumptions \ref{assmp:sg4} and \ref{assmp:sg5} with $\mu(t) = 0$. It also satisfies Assumption \ref{assmp:sg3}, as $\tilde{f}(\bx, \ha, \ba, t)^2$ is a quadratic function of $\tilde{\ba}$ for linear $\tilde{f}(\bx, \ha, \ba, t)$. Following the velocity gradient formalism, the resulting adaptation law is given by
\begin{equation}
    \dot{\hat{\ba}} = -\bP\bY(\bx, t)^\T\left(\gamma s(\bx, \bx_d(t)) + \kappa \bY(\bx, t)\tilde{\ba}\right).
    \label{eqn:comp_adapt}
\end{equation}
Equation (\ref{eqn:comp_adapt}) depends on the function approximation error $\tilde{f}(\bx, \ha, \ba, t) = Y(\bx, t)\tilde{\ba}$, which is not measured and hence cannot be used directly in an adaptation law. Nevertheless, it can be obtained through a proportional-integral form for $\ha$ in an identical manner to Section~\ref{ssec:adaptive_basic_nlin}. To do so, we define
\begin{align}
    \label{eqn:pi_comp_1}
    \bxi(\bx, t) &= -\kappa \bP s(\bx, \bx_d(t))\bY(\bx, t)^\T,\\
    \label{eqn:pi_comp_2}
    \brho(\bx, t) &= \kappa \bP \int_{x_n(t_0)}^{x_n(t)}s(\bx, \bx_d(t))\frac{\p \bY(\bx, t)^\T}{\p x_n}dx_n,\\
    \label{eqn:pi_comp_3}
    \hat{\ba} &= \overline{\ba} + \bxi(\bx, t) + \brho(\bx, t),\\
    \label{eqn:pi_comp_4}
    \dot{\overline{\ba}} &= -\left(\kappa\eta + \gamma\right) s(\bx, \bx_d(t)) \bP \bY(\bx, t)^\T + \kappa s(\bx, \bx_d(t)) \sum_{i=1}^{n-1}\bP\frac{\p \bY(\bx, t)}{\p x_i}\dot{x}_i\nonumber\\
    &\phantom{=} - \sum_{i=1}^{n-1}\frac{\p \brho(\bx, t)}{\p x_i}\dot{x}_i - \frac{\p \bxi(\bx, t)}{\p t} - \frac{\p \brho(\bx, t)}{\p t}.
\end{align}
Computing $\dot{\hat{\ba}}$ demonstrates that (\ref{eqn:comp_adapt}) is obtained through only the known signals contained in (\ref{eqn:pi_comp_1})-(\ref{eqn:pi_comp_4}) despite its dependence on $\bY(\bx, t)\tilde{\ba}$.

\begin{rmk}
The $\bY(\bx, t)\tilde{\ba}$ term may also be obtained by following the Immersion and Invariance formalism~\citep{II1, II2}. To our knowledge, this discussion is the first that demonstrates the possibility of using a PI law in combination with a standard Lyapunov-stability motivated adaptation law to obtain a composite law.$\qeddef$
\end{rmk}

\begin{rmk}
Rearranging (\ref{eqn:comp_adapt}) shows that $\dot{\tilde{\ba}} + \bP\bY(\bx, t)^\T\bY(\bx, t)\tilde{\ba} = -\bP \bY(\bx, t)^\T s(\bx, \bx_d(t))$, so that the additional term can be seen to add a damping term that smooths adaptation~\citep{slot_li_book}.$\qeddef$
\end{rmk}
\begin{rmk}
\label{rmk:pde}
As mentioned in Section~\ref{ssec:adaptive_basic_lin}, for clarity of presentation we have restricted our discussion to the $n^{\text{th}}$-order system (\ref{eqn:gen_sys}). The PI form (\ref{eqn:pi_comp_3}) leads to \textit{undesired} unknown terms in addition to the \textit{desired} unknown term. Here, the desired unknown term is $-\kappa\bP\bY(\bx, t)^\T\bY(\bx, t)\tilde{\ba}$ while the undesired unknown term is $-\kappa \bP\frac{\p \bY(\bx, t)}{\p x_n}\dot{x}_ns(\bx, \bx_d(t))$. The purpose of introducing the additional proportional term $\brho(\bx, t)$ in (\ref{eqn:pi_comp_1}) is to cancel this undesired unknown term. In general, cancellation of the undesired terms can be obtained by choosing $\brho$ to solve a PDE, and solutions to this PDE will only exist if the undesired term is the gradient of an auxiliary function. $\brho$ is then chosen to be exactly this auxiliary function. In some cases, the PDE can be avoided, such as through dynamic scaling techniques~\citep{dyn_scale} or the related embedding technique of~\citet{tyukin_book}.$\qeddef$
\end{rmk}

The properties of the adaptive law (\ref{eqn:comp_adapt}) may be summarized with the following proposition.

\begin{prop}
\label{prop:comp_adapt}
Consider the adaptation algorithm (\ref{eqn:comp_adapt}) with a linearly parameterized unknown, $f(\bx, \ba, t) = \bY(\bx, t)\ba$. Then all trajectories $(\bx(t), \hat{\ba}(t))$ remain bounded, $s(\bx(\cdot), \bx_d(\cdot)) \in \LL_2\cap \LL_{\infty}$, $\tilde{f}(\bx(\cdot), \ha(\cdot), \cdot) \in \LL_2$, $s(\bx(t), \bx_d(t)) \rightarrow 0$, and $\bx(t) \rightarrow \bx_d(t)$.
\end{prop}
The proof is given in Appendix~\ref{app2:comp_adapt}.

Following the velocity gradient with momentum approach of Section~\ref{ssec:ho_sg}, we now obtain a higher-order composite algorithm, and give a PI implementation. We again consider a hybrid local and integral velocity gradient functional, so that (\ref{eqn:breg_lag}) takes the form
\begin{align}
    \mathcal{L}&\left(\hat{\ba}, \dot{\hat{\ba}}, t\right) =\nonumber\\
    &\phantom{=} \frac{e^{\int_0^t \beta \sN(t') dt'}}{\beta \sN(t)}\left(\frac{1}{2}\dot{\hat{\ba}}^\T\dot{\hat{\ba}} - \beta \sN(t)\frac{d}{dt}\left[\frac{\gamma}{2}s(\bx, \bx_d(t))^2 + \frac{\kappa}{2}\int_0^t\tilde{f}^2(\bx(t'), \hat{\ba}(t'), \ba, t')dt'\right]\right)
    \label{eqn:breg_adaptive_comp}
\end{align}

where $\gamma > 0$ and $\kappa > 0$ are positive constants weighting the two error terms. The Euler-Lagrange equations then lead to the higher-order composite system
\begin{equation}
    \ddot{\hat{\ba}} + \left(\beta \sN(t) - \frac{\dot{\sN}(t)}{\sN(t)}\right)\dot{\hat{\ba}} = -\beta \sN(t) \bY(\bx, t)^\T\left(\gamma s(\bx, \bx_d(t)) + \kappa \bY(\bx, t)\tilde{\ba}\right).
    \label{eqn:ho_comp}
\end{equation}
As in Section~\ref{ssec:breg}, (\ref{eqn:ho_comp}) may be implemented as two first-order systems
\begin{align}
    \label{eqn:ho_comp_v}
    \dot{\hat{\bv}} &= -\bY(\bx, t)^\T\left(\gamma s(\bx, \bx_d(t)) + \kappa \bY(\bx, t)\tilde{\ba} \right),\\
    \label{eqn:ho_comp_a}
    \dot{\hat{\ba}} &= \beta \sN(t) \left(\hat{\bv} - \hat{\ba}\right).
\end{align}
In an implementation, (\ref{eqn:ho_comp_v}) is obtained through the PI form $\hat{\bv} = \overline{\bv} + \bxi(\bx, t) + \brho(\bx, t)$ with $\bxi$, $\brho$, and $\dot{\overline{\bv}}$ given by (\ref{eqn:pi_comp_1}), (\ref{eqn:pi_comp_2}), and (\ref{eqn:pi_comp_4}) respectively with $\bP = \bI$. The properties of the higher-order composite adaptation law (\ref{eqn:ho_comp}) are stated in the following proposition.

\begin{prop}
\label{prop:ho_comp}
Consider the higher-order composite adaptation algorithm (\ref{eqn:ho_comp}) for a linearly parameterized unknown, $f(\bx, \ba, t) = \bY(\bx, t)\ba$. Set $\sN(t) = 1 + \mu \Vert\bY(\bx(t), t)\Vert^2$ and $\mu > \frac{\gamma}{\beta}\left(\frac{1}{\eta} + \frac{\kappa}{\gamma}\right)$. Then all trajectories $(\bx(t), \hat{\bv}(t), \hat{\ba}(t))$ remain bounded, $\hat{\bv}(\cdot) - \hat{\ba}(\cdot) \in \LL_2$, $s(\bx(\cdot), \bx_d(\cdot)) \in \LL_\infty \cap \LL_2$, $\tilde{f}(\bx(\cdot), \ha(\cdot), \cdot) \in \LL_\infty \cap \LL_2$, $s(\bx(t), \bx_d(t)) \rightarrow 0$, and $\bx(t) \rightarrow \bx_d(t)$.
\end{prop}
The proof is given in Appendix~\ref{app2:ho_comp}.

\begin{rmk}
\label{rmk:matrix}
The signal $\sN(t)$ may be chosen alternatively to be matrix-valued as $\mathbf{N} = \bI + \mu \bY(\bx(t), t)^\T\bY(\bx(t), t)$.$\qeddef$
\end{rmk}
\begin{rmk}
The $\bY(\bx, t)\tilde{\ba}$ term may be used in isolation, with proof by consideration of the Lyapunov-like function $V(\ha, \hv) = \frac{1}{2}\Vert\tilde{\bv}\Vert^2 + \frac{1}{2}\Vert\hat{\ba} - \hat{\bv}\Vert^2$.$\qeddef$
\end{rmk}
\begin{rmk}
\label{rmk:gain_mat}
A gain matrix $\bP = \bP^\T > 0$ of appropriate dimension may be placed in front of $\bY(\bx, t)^\T$ in $\dot{\hv}$. The quadratic parameter estimation error terms in the Lyapunov function should then be replaced by the weighted terms $\frac{1}{2}\tilde{\bv}^\T\bP^{-1}\tilde{\bv} + \frac{1}{2}\left(\hat{\ba} - \hat{\bv}\right)^\T\bP^{-1}\left(\hat{\ba} - \hat{\bv}\right)$, and bounds on $\mu$ will be given in terms of $\Vert\bP\Vert$.$\qeddef$
\end{rmk}

\subsection{A momentum algorithm for nonlinearly parameterized adaptive control}
\label{ssec:ho}
We now use the development in Proposition~\ref{prop:int_ho} to present a momentum algorithm applicable when the unknown parameters appear nonlinearly in the dynamics. An analogy to methods for learning generalized linear models in statistics and machine learning is explored in Appendix~\ref{app2:glms}.

We study the momentum-like variant of (\ref{eqn:tyukin_alg})
\begin{equation}
    \ddot{\hat{\ba}} + \left(\beta \sN(t) - \frac{\dot{\sN}(t)}{\sN(t)}\right)\dot{\hat{\ba}} = - \gamma \beta \sN(t)\tilde{f}(\bx, \ha, \ba, t)\balf(\bx, t),
    \label{eqn:ho_tyukin}
\end{equation}
which, as before, admits an equivalent representation in terms of two first-order systems,
\begin{align}
    \label{eqn:ho_tyukin_v}
    \dot{\hat{\bv}} &= -\gamma\tilde{f}(\bx, \ha, \ba, t)\balf(\bx, t),\\
    \label{eqn:ho_tyukin_a}
    \dot{\hat{\ba}} &= \beta \sN(t) \left(\hat{\bv} - \hat{\ba}\right).
\end{align}
(\ref{eqn:ho_tyukin}) may be implemented through (\ref{eqn:ho_tyukin_v}) \& (\ref{eqn:ho_tyukin_a}) via the PI form (\ref{eqn:tyukin_pi_1})-(\ref{eqn:tyukin_pi_4}) applied to the $\hat{\bv}$ variable. (\ref{eqn:ho_tyukin}) may be obtained via the Bregman Lagrangian (\ref{eqn:breg_adaptive_comp}) for velocity gradient laws with momentum by choosing only the integral term. It is then necessary to modify the resulting Euler-Lagrange equations by setting $f_m'$ to $\pm 1$ based on monotonicity of $\tilde{f}(\bx, \ha, \ba, t)$ as described in Section~\ref{ssec:adaptive_basic_nlin}. The following proposition summarizes the properties of (\ref{eqn:ho_tyukin_v}) and (\ref{eqn:ho_tyukin_a}).

\begin{prop}
\label{prop:ho_tyukin}
Consider the algorithm (\ref{eqn:ho_tyukin}) or its equivalent form (\ref{eqn:ho_tyukin_v}) \& (\ref{eqn:ho_tyukin_a}) under Assumption~\ref{assmp:tyukin} with $\sN(t) = 1 + \mu \Vert\balf(\bx(t), t)\Vert^2$ and $\mu > \frac{\gamma D_1}{\beta}$. Then, all trajectories $(\bx(t), \hat{\ba}(t), \hat{\bv}(t))$ remain bounded, $\tilde{f}(\bx(\cdot), \ha(\cdot), \cdot) \in \LL_2$, $\left(\hat{\ba}(\cdot) - \hat{\bv}(\cdot)\right) \in \LL_2$, $s(\bx(\cdot), \bx_d(\cdot)) \in \LL_2\cap\LL_\infty$, $s(\bx(t), \bx_d(t)) \rightarrow 0$ and $\bx(t) \rightarrow \bx_d(t)$.
\end{prop}
The proof is given in Appendix~\ref{app2:ho_tyukin}.

\begin{rmk}
Remarks~\ref{rmk:matrix}~\&~\ref{rmk:gain_mat} also apply in this setting. By following the proof of Proposition~\ref{prop:ho_tyukin}, one may also take $\sN(t)$ to be matrix-valued as $\mathbf{N}(t) = 1 + \mu \balf(\bx(t), t)\balf(\bx(t), t)^\T$. One may also use a gain matrix $\bP = \bP^\T > 0$ of appropriate dimension.$\qeddef$
\end{rmk}

\section{The elastic modification}
\label{ssec:elastic}
As noted in Section~\ref{ssec:ho_sg}, one prominent difference between Nesterov's accelerated method and the momentum algorithms presented in this work is the point at which the gradient is evaluated. This distinction suggests an alternative interpretation of the momentum algorithms in Section~\ref{ssec:ho_sg} as adding an additional averaging step that consists of passing the standard first-order parameter estimates through a low-pass filter in \textit{series}. It is well-known that iterate averaging schemes for stochastic gradient descent -- such as Polyak-Ruppert averaging~\citep{polyak_average, ruppert_average} and Polyak-Juditsky averaging~\citep{pol-jud} -- can improve convergence rates via variance reduction. The momentum algorithms of Section~\ref{ssec:ho_sg} provide a stable way to perform iterate averaging in nonlinear adaptive control with an exponentially weighted average similar to~\citet{lorenzo_average}, and show that the averaging rate must depend on the norm of a system signal such as $\bY(\bx, t)$ for stability.

An alternative option to perform iterate averaging would be to couple the parameter estimates to the output of a filter of the parameter estimates in \textit{feedback}. Incorporating this type of feedback coupling is similar in spirit to the elastic averaging SGD algorithm of~\citet{easgd}. Here we show that adding such feedback coupling can always be done for velocity gradient algorithms and velocity gradient algorithms with momentum. We focus on local velocity gradient algorithms: extensions to integral functionals and nonlinearly parameterized dynamics are straightforward combining the proof techniques here with those of Proposition~\ref{prop:int_ho}. Similar to iterate averaging methods in optimization, this approach may improve convergence in settings with sparse, noisy observations, such as real-world applications of the prediction methods developed in Section~\ref{sec:obs_predict}.

\begin{prop}
\label{prop:sg_local_elastic}
Consider the adaptive control algorithm
\begin{align*}
    \dot{\ha} &= -\bP\nabla_{\ha}\dot{Q}(\bx, \ha, t) + k\left(\overline{\ba} - \ha\right),\\
    \dot{\overline{\ba}} &= k\left(\ha - \overline{\ba}\right).
\end{align*}
under Assumptions \ref{assmp:sg1}-\ref{assmp:sg3} for $k \geq 0$ and $\bP > 0$. Then all solutions $(\bx(t), \hat{\ba}(t), \overline{\ba}(t))$ remain bounded, $\left(\ha(\cdot) - \overline{\ba}(\cdot)\right) \in \LL_2$, and for all $\bx(0) \in \mathbb{R}^n$,
\begin{equation*}
    \lim_{t\rightarrow\infty} Q(\bx(t), t) = 0.
\end{equation*}
\end{prop}
The proof is given in Appendix~\ref{app2:prop:sg_local_elastic}. We may also state a similar result with momentum, allowing for coupling to a filtered version of both $\hv$ and $\ha$.

\begin{prop}
\label{prop:sg_local_mom_elastic}
Consider the adaptive control algorithm
\begin{align*}
    \dot{\hv} &= -\bP\nabla_{\ha}\dot{Q}(\bx, \ha, t) + k_v\left(\overline{\bv} - \hv\right),\\
    \dot{\overline{\bv}} &= k_v\left(\ha - \overline{\ba}\right),\\
    \dot{\ha} &= \beta\sN(t)\left(\hv - \ha\right) + k_a\beta\sN(t)\left(\overline{\ba} - \ha\right)\\
    \dot{\overline{\ba}} &= k_a\beta\sN(t)\left(\ha - \overline{\ba}\right)
\end{align*}
under Assumptions~\ref{assmp:sg1}-\ref{assmp:sg3}. Fix $k_v \geq 0, k_a \geq 0, \gamma > 0,$ and $\beta \geq 2 k_v$. Then all solutions $(\bx(t), \hat{\ba}(t), \overline{\ba}(t), \hv(t), \overline{\bv}(t))$ remain bounded, $\ha(\cdot) - \overline{\ba}(\cdot) \in \LL_2$, $\ha(\cdot) - \hv(\cdot) \in \LL_2$, $\hv(\cdot) - \overline{\bv}(\cdot) \in \LL_2$, and for all $\bx(0) \in \mathbb{R}^n$,
\begin{equation*}
    \lim_{t\rightarrow\infty} Q(\bx(t), t) = 0.
\end{equation*}
\end{prop}
The proof is given in Appendix~\ref{app2:prop:sg_local_mom_elastic}.

\section{Natural momentum algorithms}
\label{sec:ho_nat}
The Bregman Lagrangian allows for the introduction of non-Euclidean metrics. In Section~\ref{ssec:breg}, we took the potential function $\psi$ to be the Euclidean norm, $\psi(\bx) = \frac{1}{2}\Vert\bx\Vert^2$. We now show that taking $\psi$ to be an arbitrary strongly convex function leads to a more general class of algorithms that can be seen as the higher-order variants of those discussed in Sections~\ref{ssec:fo_rie_adapt} and~\ref{sec:nat}. With the same definitions of $\bar{\alpha}, \bar{\gamma}$, and $\bar{\beta}$ as in Section~\ref{ssec:breg}, but now taking $\psi$ to be arbitrary in $\bregd{\cdot}{\cdot}$, the Bregman Lagrangian (\ref{eqn:breg_lag_sg}) takes the form
\begin{equation}
    \LL(\ha, \dot{\ha}, t) = e^{\int_0^t\beta \sN(t')dt'}\left(\beta \sN(t) \bregd{\ha + \frac{\dot{\ha}}{\beta \sN(t)}}{\ha} - \gamma \frac{d}{dt}Q(\bx, t)\right),
    \label{eqn:breg_rie}
\end{equation}
with a similar expression for the integral variant. In \eqref{eqn:breg_rie}, $\bregd{\cdot}{\cdot}$ has its first argument evaluated at $\ha + \frac{\dot{\ha}}{\beta\sN(t)}$. This quantity is precisely $\hv$, which leads to a particularly simple form of first-order equations, as highlighted in the following propositions.

We start with the case of a local functional, which requires a modification of Assumption \ref{assmp:local_ho_cond}.

\begin{assumption}
\label{assmp:local_ho_cond_rie}
There exists a time-dependent signal $N(t)$ and non-negative scalar values $\beta \geq 0, \mu \geq 0$ such that the time-derivative of the goal functional evaluated at the true parameters, $\dot{Q}(\bx, \ba, t)$, satisfies the following inequality
\begin{equation}
    \dot{Q}(\bx, \ba, t) - \frac{\beta\mu}{\gamma}N(t)\Vert\hat{\ba} - \hat{\bv}\Vert^2 + \left(\hat{\ba} - \hat{\bv}\right)^\T\left(\bI + \left[\nabla^2\psi(\hv)\right]^{-1}\right)\nabla_{\ha}\dot{Q}(\bx, \ha, t) \leq -\rho(Q(\bx, t)).
    \label{eqn:local_ho_cond_rie}
\end{equation}
In (\ref{eqn:local_ho_cond_rie}), $\rho(\cdot)$ is positive definite, continuous in its argument, and satisfies $\rho(0) = 0$.$\qeddef$
\end{assumption}

Assumption \ref{assmp:local_ho_cond_rie} leads to the following non-Euclidean equivalent of Proposition~\ref{prop:local_ho_sg}.
\begin{prop}
\label{prop:local_ho_sg_rie}
Consider the adaptive control algorithm
\begin{equation}
    \ddot{\ha} + \left(\beta \sN(t) - \frac{\dot{\sN}(t)}{\sN(t)}\right)\dot{\ha} + \gamma\beta\sN(t)\nabla^2\psi\left(\ha + \frac{\dot{\ha}}{\beta\sN(t)}\right)^{-1}\nabla_{\ha}\dot{Q}(\bx, \ha, t) = 0
    \label{eqn:local_ho_sg_rie}
\end{equation}
or its equivalent first order form
\begin{align}
    \label{eqn:vdot_rie}
    \dot{\hv} &= -\gamma\left[\nabla^2\psi(\hv)\right]^{-1}\nabla_{\ha}\dot{Q}(\bx, \ha, t),\\
    \label{eqn:adot_rie}
    \dot{\ha} &= \beta\sN(t)\left(\hv - \ha\right),
\end{align}
and assume $Q(\bx, t)$ satisfies Assumptions \ref{assmp:sg1}, \ref{assmp:sg3}, and \ref{assmp:local_ho_cond_rie}. Then all solutions $(\bx(t), \hat{\bv}(t), \hat{\ba}(t))$ remain bounded, $\left(\hat{\ba}(\cdot) - \hat{\bv}(\cdot)\right) \in \LL_2$, and $\lim_{t\rightarrow \infty} Q(\bx(t), t) = 0$.
\end{prop}
The proof is given in Appendix~\ref{app2:prop:local_ho_sg_rie}. A particularly satisfying property of the above algorithm is that taking the $\beta\rightarrow\infty$ limit recovers the first-order non-Euclidean velocity gradient algorithm. 

For the integral variant, we require a modified version of Assumption ~\ref{assmp:int_ho_cond}.
\begin{assumption}
\label{assmp:int_ho_cond_rie}
$R(\bx, \hat{\ba}, t) \geq 0$ for all $\bx, \hat{\ba}$, and $t$, and is uniformly continuous in $t$ for bounded $\bx$ and $\ha$. $\nabla_{\ha}R(\bx, \ha, t)$ is locally bounded in $\bx$ and $\ha$ uniformly in $t$. Furthermore, there exists a time-dependent signal $N(t)$ and non-negative scalar values $\beta \geq 0, \mu \geq 0$ such that, for all $\ha$ and $\hv$ and for some constant $k > 0$,
\begin{equation*}
    R(\bx, \ba, t) - R(\bx, \hat{\ba}, t) - \frac{\beta\mu}{\gamma}N(t)\Vert\hat{\ba}-\hat{\bv}\Vert^2 + \left(\hat{\ba}-\hat{\bv}\right)^\T\left(\bI + \left[\nabla^2\psi\left(\hv\right)\right]^{-1}\right)\nabla_{\hat{\ba}}R(\bx, \hat{\ba}, t) \leq -k R(\bx, \hat{\ba}, t).\qeddef
\end{equation*}
\end{assumption}
With Assumption \ref{assmp:int_ho_cond_rie}, we can state the following proposition.
\begin{prop}
\label{prop:int_ho_rie}
Consider the algorithm
\begin{equation*}
    \ddot{\ha} + \left(\beta\sN(t) - \frac{\dot{\sN}(t)}{\sN(t)}\right)\dot{\ha} + \gamma\beta\sN(t)\left[\nabla^2\psi\left(\ha + \frac{\dot{\ha}}{\beta\sN(t)}\right)\right]^{-1}\nabla_{\ha}R(\bx, \ha, t) = 0
\end{equation*}
or its equivalent first-order form 
\begin{align*}
    \dot{\hv} &= -\gamma\left[\nabla^2\psi(\hv)\right]^{-1}\nabla_{\ha}R(\bx, \ha, t),\\
    \dot{\ha} &= \beta\sN(t)\left(\hv - \ha\right),
\end{align*}
along with Assumptions ~\ref{assmp:sg3} \& \ref{assmp:int_ho_cond_rie}. Let $T_x$ denote the maximal interval of existence of $\bx(t)$. Then, $\hat{\bv}(t)$ and $\hat{\ba}(t)$ remain bounded for $t\in[0, T_x)$ and  $\left(\hat{\ba}(\cdot) -\hat{\bv}(\cdot)\right) \in \LL_2$ over this interval. Moreover, $\int_0^{T_x} R(\bx(t'), \hat{\ba}(t'), t')dt' < \infty$. For any bounded solution $\bx(t)$, these conclusions hold for all $t$ and $\lim_{t\rightarrow \infty}R(\bx(t), \ha(t), t) = 0$.
\end{prop}
 The proof is given in Appendix~\ref{app2:prop:int_ho_rie}. Propositions~\ref{prop:local_ho_sg_rie}~\&~\ref{prop:int_ho_rie} may be applied to derive non-Euclidean variants of the non-filtered composite algorithm and the momentum algorithm for nonlinearly parameterized adaptive control presented in this work\footnote{The nonlinearly parameterized algorithm requires slight modification, accounting for monotonicity.}. As a concrete example, these general theorems immediately lead to stable adaptation laws for the $n^{\text{th}}$ order system \eqref{eqn:gen_sys} in the linearly parameterized setting given by
 \begin{equation*}
    \ddot{\hat{\ba}} + \left(\beta \sN(t) - \frac{\dot{\sN}(t)}{\sN(t)}\right)\dot{\hat{\ba}} + \gamma \beta \sN(t) \left(\nabla^2\psi\left(\hat{\ba} + \frac{\dot{\hat{\ba}}}{\beta \sN(t)}\right)\right)^{-1}\bY(\bx, t)^\T s(\bx, \bx_d(t)) = 0.
\end{equation*}
or the equivalent system of equations
\begin{align*}
    \dot{\hat{\bv}} &= -\gamma\left(\nabla^2\psi\left(\hat{\bv}\right)\right)^{-1}\bY(\bx, t)^\T s(\bx, \bx_d(t)),\\
    \dot{\hat{\ba}} &= \beta \sN(t)\left(\hat{\bv} - \hat{\ba}\right).
\end{align*}
 
\begin{rmk}
There are several alternative forms of natural adaptive algorithms with momentum that can also be proven to converge. For example, with similar techniques, it can be shown that the algorithm
\begin{align*}
     \dot{\hat{\bv}} &= -\gamma\left(\nabla^2\psi(\hat{\bv})\right)^{-1}\nabla_{\ha}\dot{Q}(\bx, \ha, t),\\
    \dot{\hat{\ba}} &= \beta \sN(t) \left(\nabla^2\psi(\ha)\right)^{-1}\left(\nabla\psi(\hv) - \nabla\psi(\ha)\right),
\end{align*}
is also globally convergent. This algorithm is equivalent to performing the Euclidean variant entirely in the mirrored domain; it can be be re-written
\begin{align*}
    \frac{d}{dt}\nabla\psi(\hv) &= -\gamma \nabla_{\ha}\dot{Q}(\bx, \ha, t),\\
    \frac{d}{dt}\nabla\psi(\ha) &= \beta \sN(t) \left(\nabla\psi(\hv) - \nabla\psi(\ha)\right),
\end{align*}
which shows that $\nabla\psi(\ha)$ obtains the same values over time as $\ha$ computed via algorithm (\ref{eqn:ho_local}). A second modification with $\dot{\ha}$ driven by $\hv$ rather than $\nabla\psi(\hv)$ is given by
\begin{align*}
    \dot{\hat{\bv}} &= -\gamma\left(\nabla^2\psi(\hat{\bv})\right)^{-1}\nabla_{\ha}\dot{Q}(\bx, \ha, t),\\
    \dot{\hat{\ba}} &= \beta \sN(t) \left(\nabla^2\psi(\hat{\ba})\right)^{-1}\left(\hv - \ha\right).\qeddef
\end{align*}
\end{rmk}
\begin{rmk}
\label{rmk:rie_PI}
Discretization of the $\dot{\ha}$ and $\dot{\hv}$ dynamics directly results in a natural gradient-like update~\citep{amari}, while discretization of the $\frac{d}{dt}\nabla\psi(\ha)$ and $\frac{d}{dt}\nabla\psi(\hv)$ dynamics leads to a mirror descent-like update~\citep{beck_teb, nem_yud}; these discrete-time algorithms have the same continuous-time limit~\citep{krich}.$\qeddef$
\end{rmk}

\begin{rmk}
The implicit regularization properties of the first-order algorithms considered in Section~\ref{sec:nat} extend to the higher-order setting captured by algorithm (\ref{eqn:local_ho_sg_rie}). The assumption that $\ha(t)\rightarrow\ha_\infty$ can be used to conclude $\dot{\ha}(t)\rightarrow 0$. As noted in Section~\ref{ssec:breg}, $\hv(t) = \ha(t) + \frac{\dot{\ha}(t)}{\beta\sN(t)}$, and this shows that $\hv(t) \rightarrow \ha_\infty$. Because $\dot{\hv}$ in (\ref{eqn:vdot_rie}) is identical to algorithm (\ref{eqn:vdot_local}), the result follows. A formal statement of this fact is provided in Appendix~\ref{app2:ho_imp_reg}$\qeddef$
\end{rmk}

\section{Simulations}
In this section, we perform several numerical experiments demonstrating the validity of our theory, and consider a number of applications of our non-Euclidean adaptive laws.
\label{sec:sim}
\subsection{Convergence and implicit regularization of a momentum algorithm for nonlinearly parameterized systems}
\label{ssec:sim_nlin}
We first empirically verify the global convergence and implicit regularization of the momentum algorithm for nonlinearly parameterized systems (\ref{eqn:ho_tyukin}). In particular, we consider a second order system 
\begin{align*}
    \dot{x}_1 &= x_2,\\
    \dot{x}_2 &= u(\bx, \ha, t) - f(\bx, \ba),
\end{align*}
with an unknown system dynamics of the form
\begin{equation}
    f(\bx, \ba) = \sigma\left(\tanh\left(\mathbf{V}\bx\right)^\T\ba\right).
    \label{eqn:nn}
\end{equation}
Equation (\ref{eqn:nn}) represents a three-layer neural network with input layer $\bx$, hidden layer weights $\mathbf{V}$, hidden layer nonlinearity $\tanh(\cdot)$, hidden layer weights $\ba$, and output nonlinearity $\sigma(x) = e^{.1x}$. The system model (\ref{eqn:nn}) can clearly be seen to satisfy Assumption~\ref{assmp:tyukin} with $\balf(\bx) = \tanh(\mathbf{V}\bx)$\footnote{While the exponential is not globally Lipschitz continuous, it is locally.}. We apply the non-Euclidean variant of algorithm~\eqref{eqn:ho_tyukin} with $\psi(\cdot) = \frac{1}{2}\norm{\cdot}_p^2$ with $p \in \{1.1, 2, 4, 6, 10\}$ where $p=1.1$ represents an approximation to the $\ell_1$ norm and $p=2$ represents the Euclidean variant. Further simulation details, including an implementation of the algorithm in PI form, are provided in Appendix~\ref{app2:sim_nlin}.

The tracking error for each choice of $\psi$ along with a baseline comparison to fixed $\ha(t) = \ha(0)$ is shown in Figure~\ref{fig:ptraj}A. Figures~\ref{fig:ptraj}B-F show trajectories for $100$ out of the $500$ parameters. The timescale on each axis is set to show the trajectories approximately until the parameters converge for the given algorithm. Each case results in remarkably different dynamics and resulting converged parameter vectors $\ha_\infty^\psi$. The tracking performance is good for each algorithm.

Further insight can be gained into the structure of the parameter vector $\ha_\infty^\psi$ found by each adaptation algorithm by consideration of the histograms (rug plots shown on $x$ axis) for $\ha$ at the end of the simulation in Figures~\ref{fig:hist}A-F. Figure~\ref{fig:hist}A shows the true parameter vector. The choice of $\psi(\cdot) = \frac{1}{2}\Vert\cdot\Vert_{1.1}^2$ in Figure~\ref{fig:hist}B leads to a sparse solution with most of the weight placed on a few parameters. This is consistent with $\ell_1$ regularized solutions found by the LASSO algorithm~\citep{LASSO}. The inset displays a closer view around zero. The choice of $\psi(\cdot) = \frac{1}{2}\Vert\cdot\Vert_2^2$ in Figure~\ref{fig:hist}C (Euclidean adaptation law) leads to a parameter vector $\ha_\infty^{\frac{1}{2}\Vert\cdot\Vert_2^2}\neq\ba$ that is roughly Gaussian distributed. This distribution highlights the implicit $\ell_2$ regularization of standard adaptation laws. The progression from $\psi(\cdot) = \frac{1}{2}\Vert\cdot\Vert_4^2$ to $\psi(\cdot) = \frac{1}{2}\Vert\cdot\Vert_{10}^2$ displays a trend towards approximate $\ell_\infty$-norm regularization: the distribution of parameters is pushed to be bimodal and peaked around $\pm 1$, and the $\ell_\infty$ norm of $\ha_\infty$ decreases as $p$ is increased.

Figure~\ref{fig:control}A shows the function approximation error $\tilde{f}^2(\bx(t), \ha(t), \ba)$ for each algorithm along with a reference value for fixed $\ha(t) = \ha(0)$. Each algorithm, as expected by our theory and seen by the low tracking error in Figure~\ref{fig:ptraj}A, pushes $\tilde{f}(\bx(t), \ha(t), \ba)^2$ to zero despite the different forms of regularization imposed on the parameter vectors. Figures~\ref{fig:control}B-F show the control input as a function of time along with the unique ``ideal'' control law $u(t) = \ddot{x}_d(t) + f(\bx_d(t), \ba)$ valid when $\bx(0) = \bx_d(0)$. All control inputs can be seen to converge to the ideal law, though the rate of convergence depends on the choice of algorithm. The control input is of reasonable magnitude throughout adaptation for each algorithm.

\begin{figure}
    \begin{tabular}{cc}
        \begin{overpic}[width=.47\textwidth]{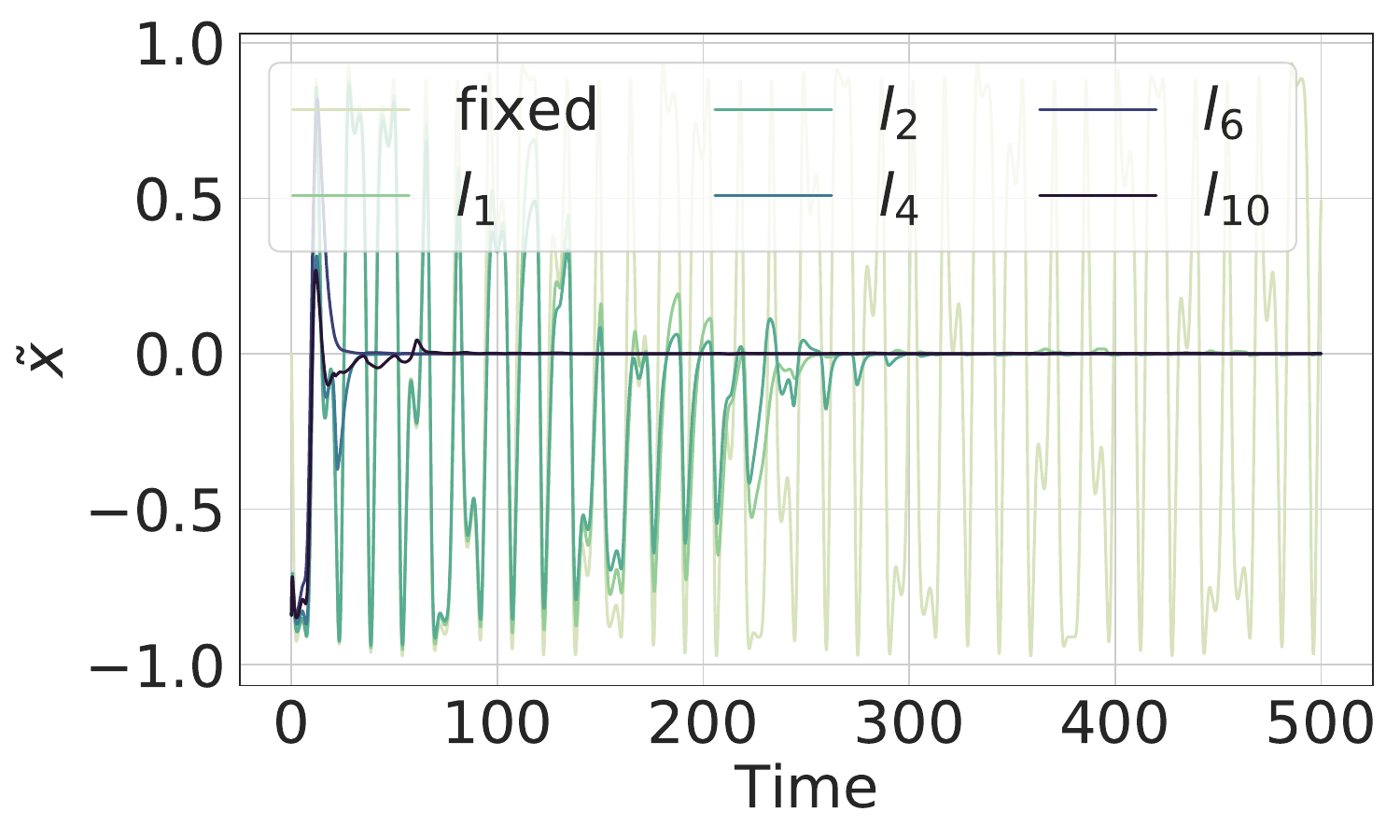}
        \put(5, 60){\textbf{A}}
        \end{overpic} &
         \begin{overpic}[width=.47\textwidth]{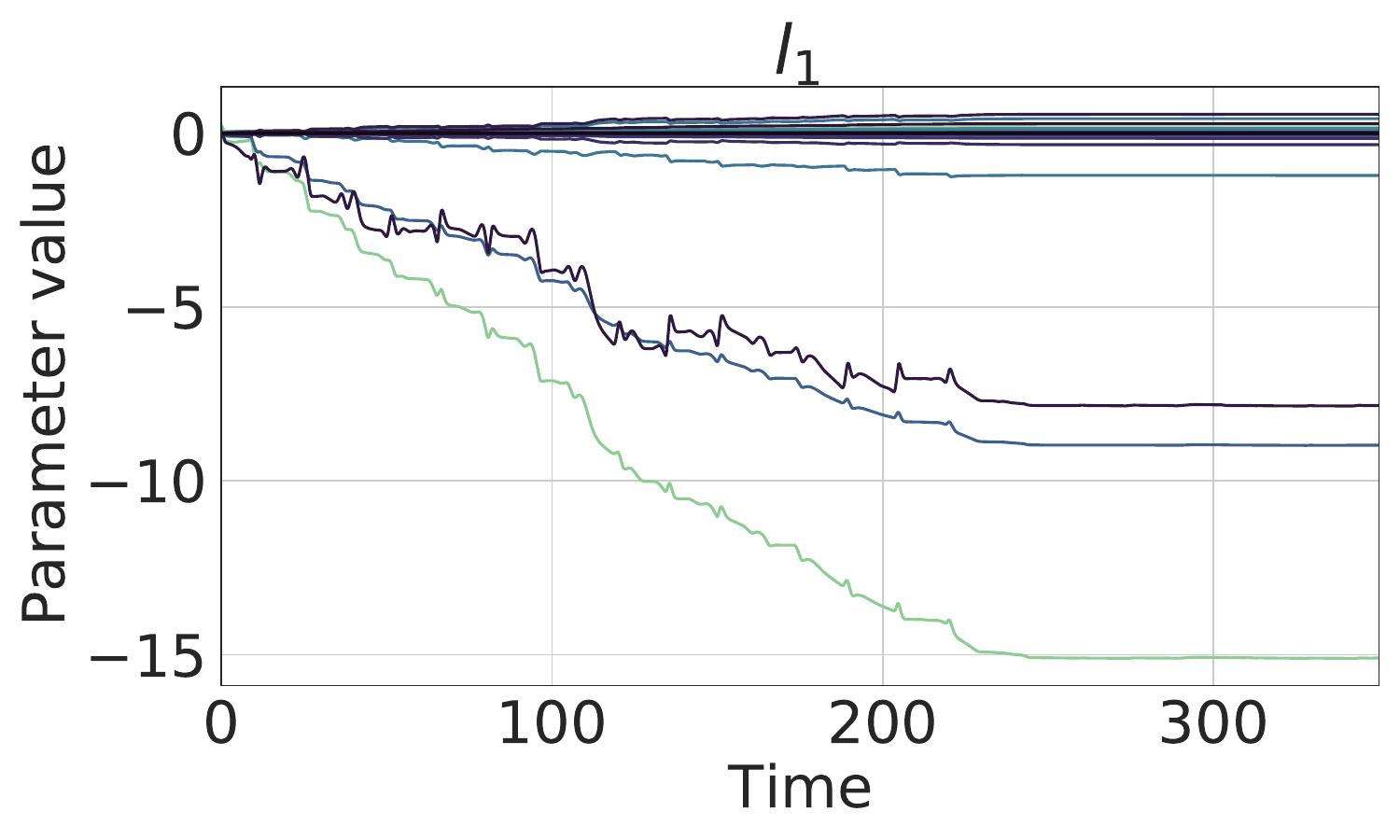}
         \put(5, 60){\textbf{B}}
         \end{overpic}
         \\
         \begin{overpic}[width=.47\textwidth]{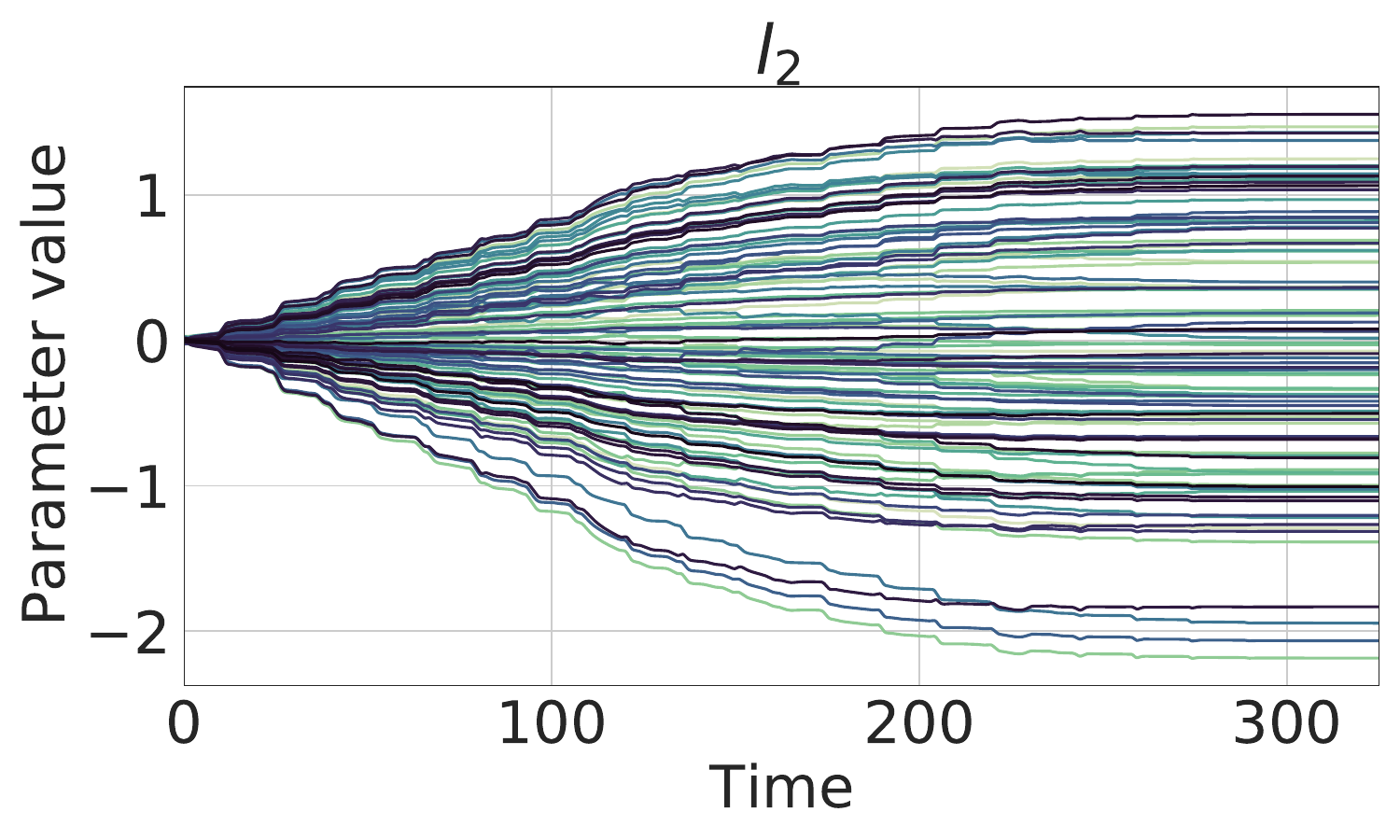} 
         \put(5, 60){\textbf{C}}
         \end{overpic}&
         \begin{overpic}[width=.47\textwidth]{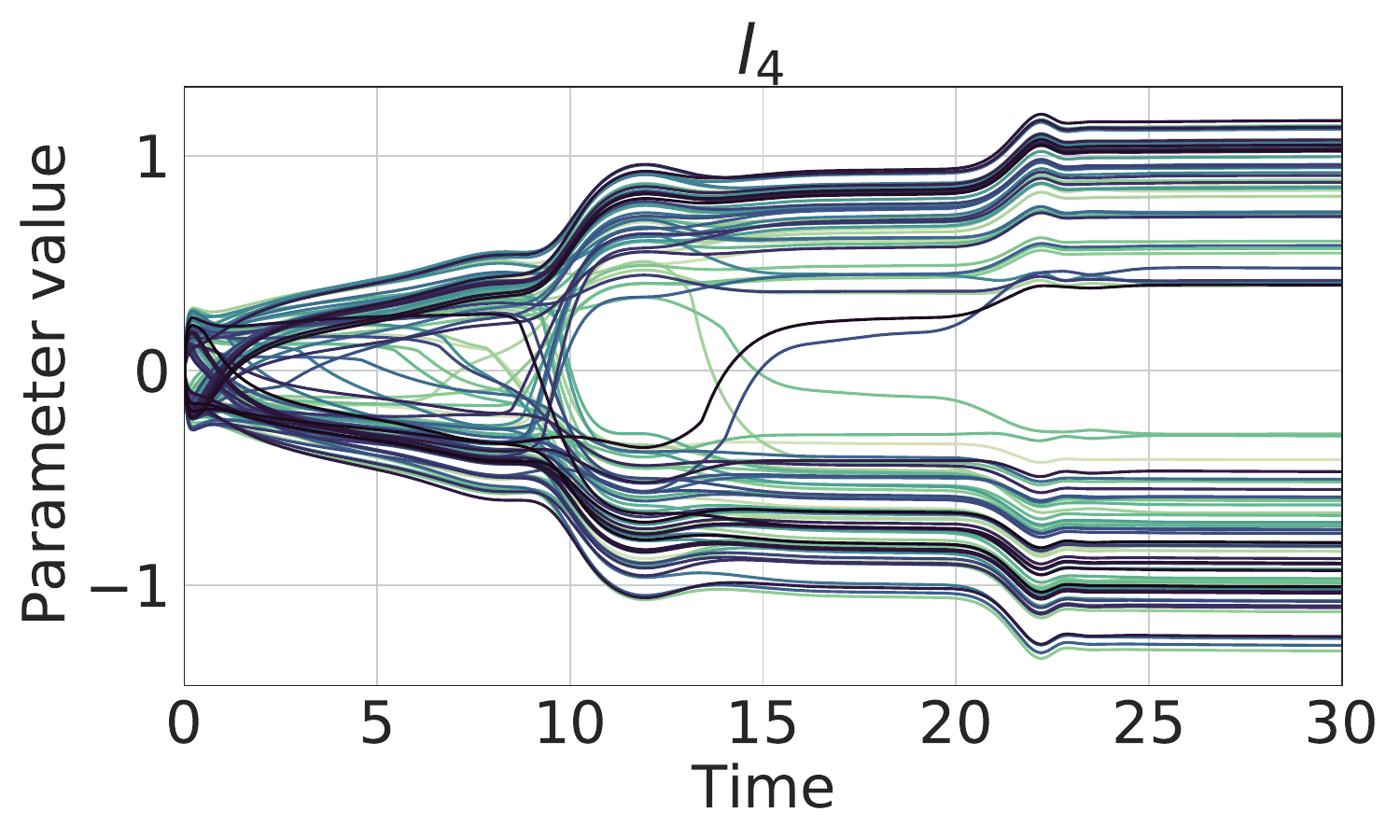} 
         \put(5, 60){\textbf{D}}
         \end{overpic}\\
         \begin{overpic}[width=.47\textwidth]{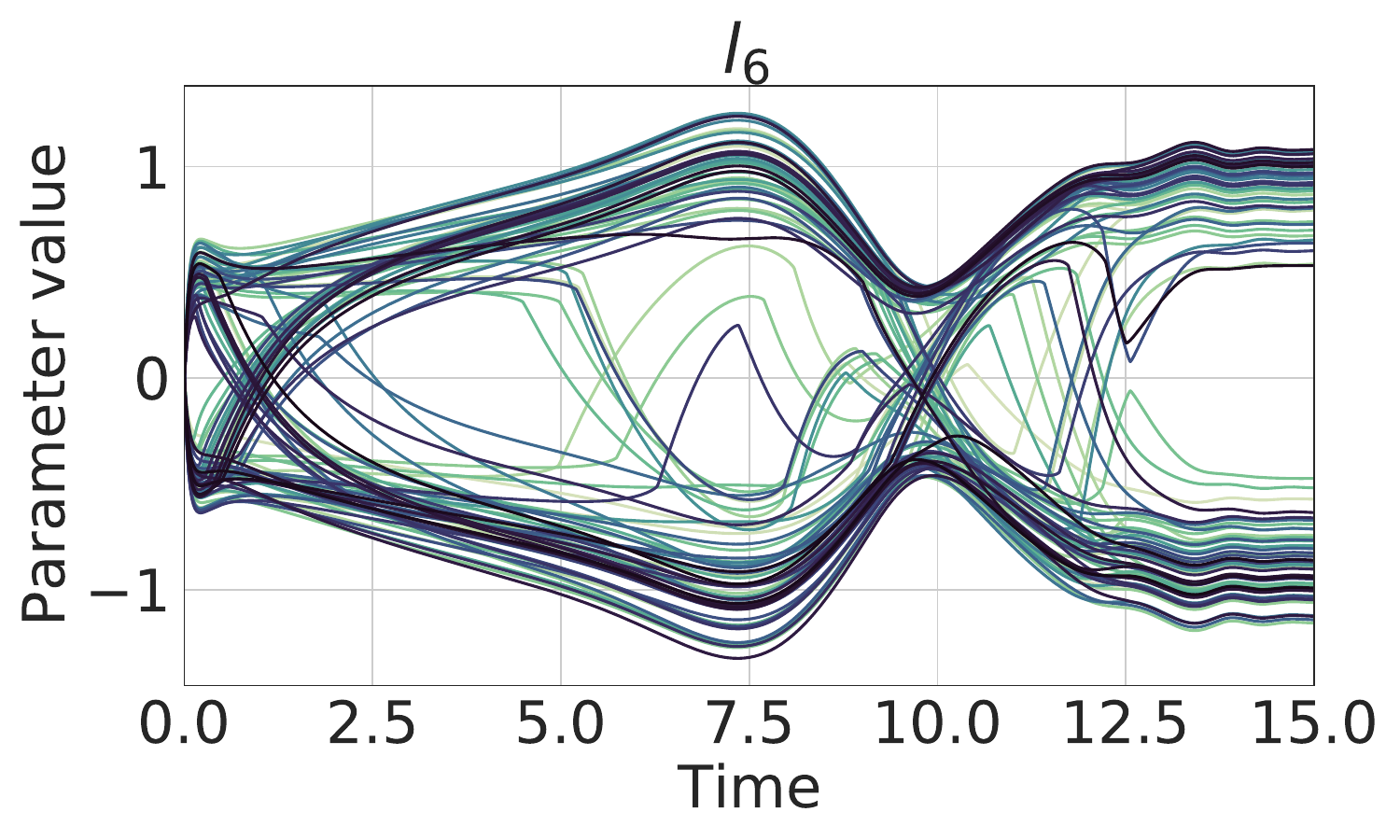} 
         \put(5, 60){\textbf{E}}
         \end{overpic}&
         \begin{overpic}[width=.47\textwidth]{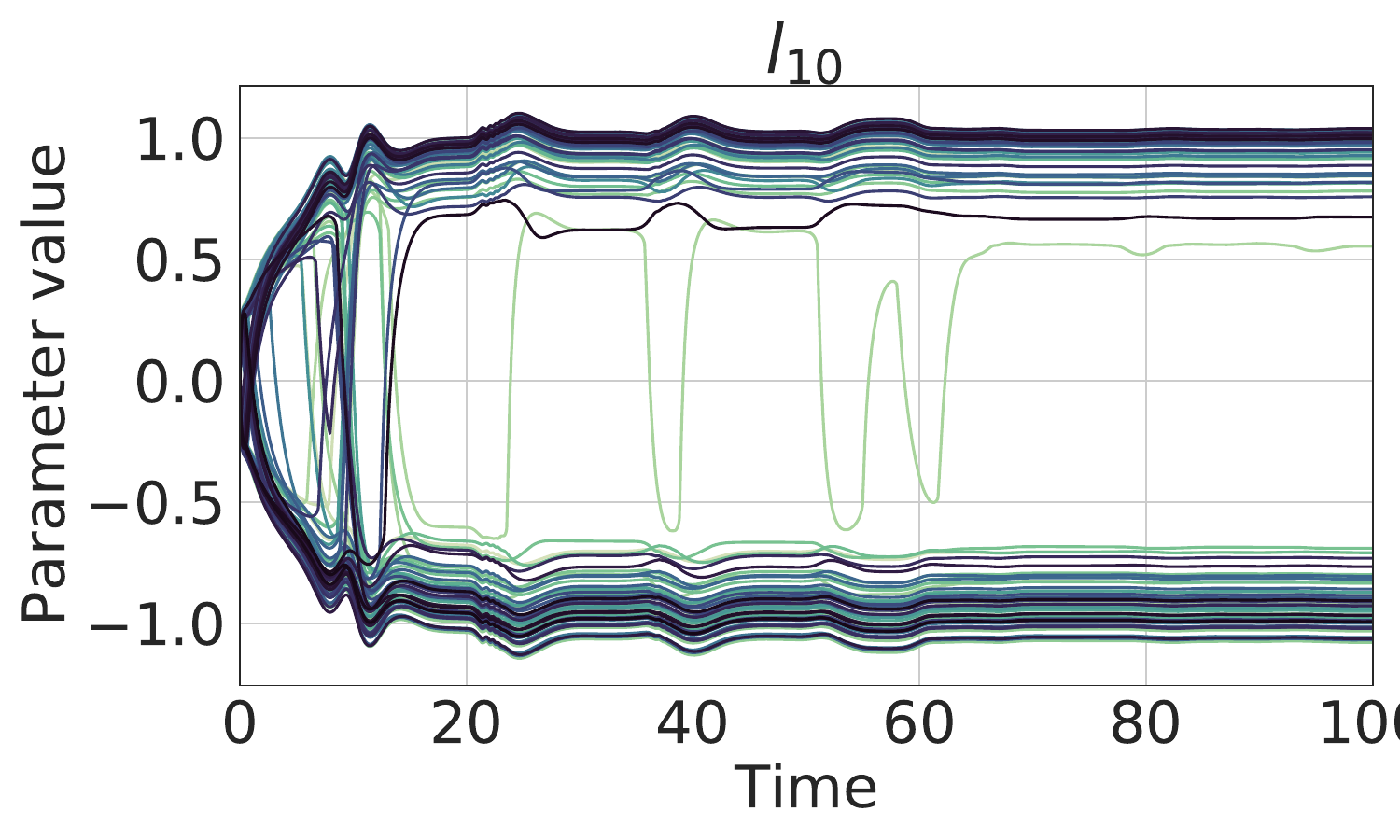}
         \put(5, 60){\textbf{F}}
         \end{overpic}
    \end{tabular}
    \caption{\textbf{Tracking error and parameter trajectories.} (A) Trajectory tracking error. All algorithms result in convergence $\bx\rightarrow\bx_d$, though transient performance differs between the algorithms. (B-F) Parameter trajectories for $100/500$ of the total parameters. Each algorithm results in remarkably different parameter trajectories and final values $\ha_\infty^\psi$.}
    \label{fig:ptraj}
\end{figure}

\begin{figure}
    \begin{tabular}{cc}
        \begin{overpic}[width=.47\textwidth]{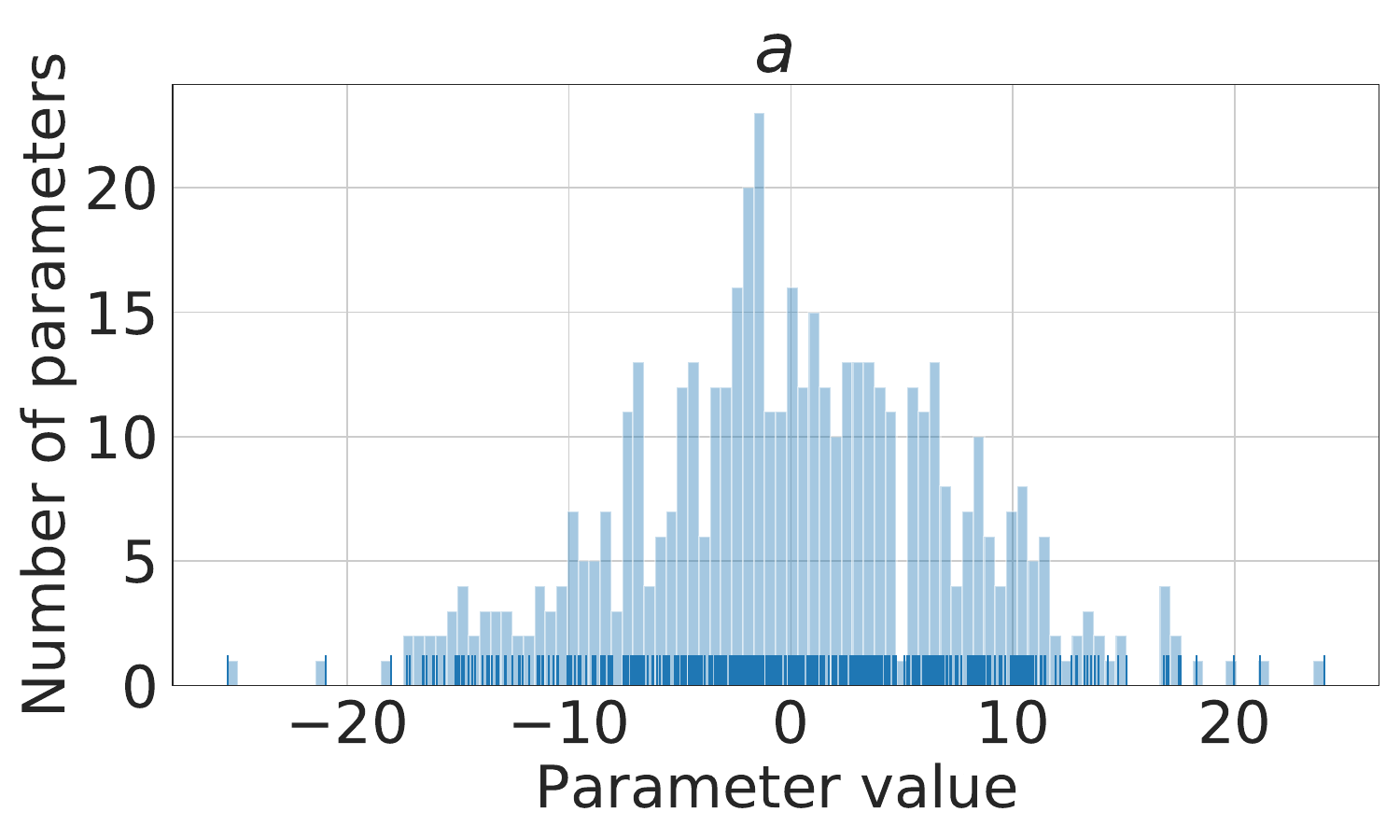}
        \put(5, 60){\textbf{A}}
        \end{overpic} &
         \begin{overpic}[width=.47\textwidth]{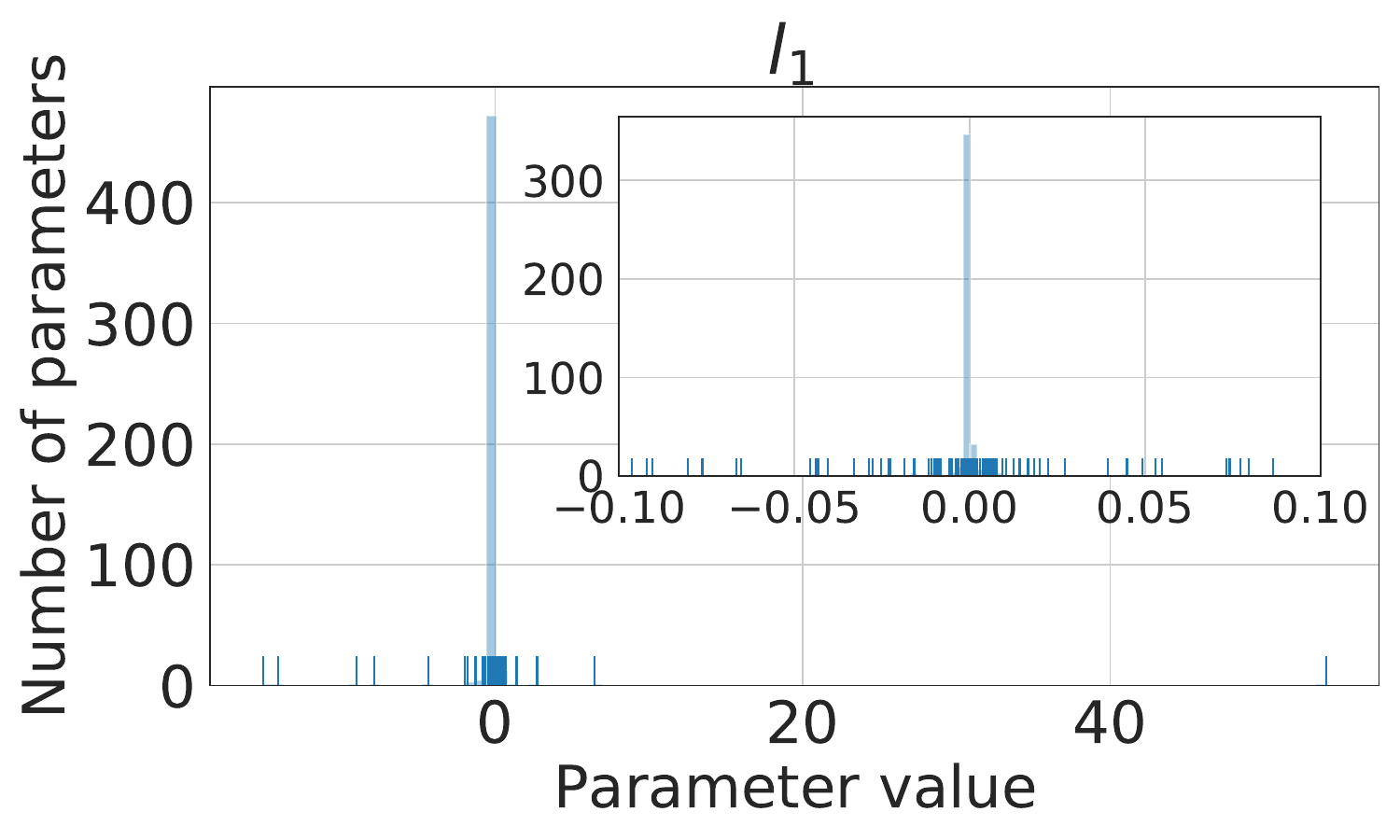}
         \put(5, 60){\textbf{B}}
         \end{overpic}
         \\
         \begin{overpic}[width=.47\textwidth]{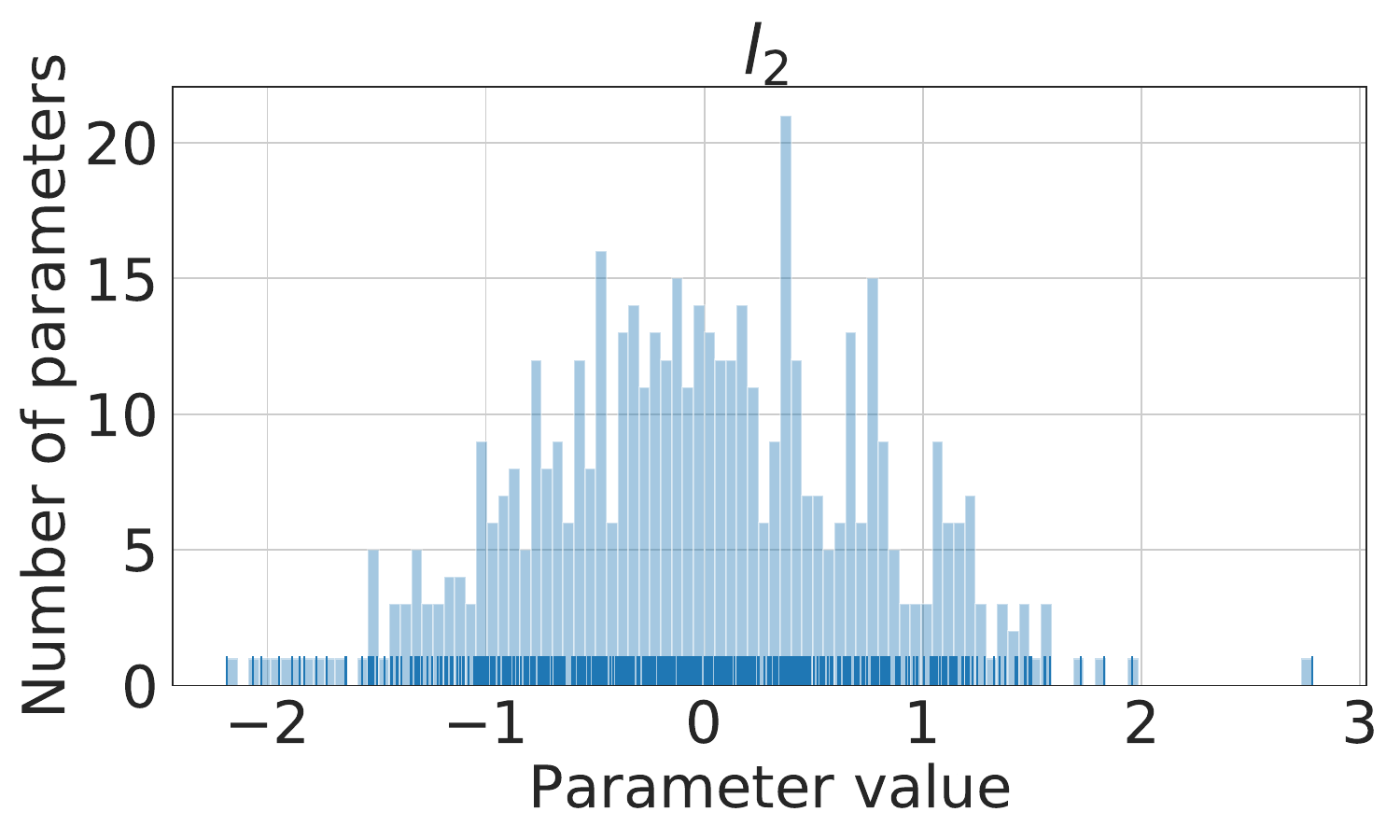} 
         \put(5, 60){\textbf{C}}
         \end{overpic}&
         \begin{overpic}[width=.47\textwidth]{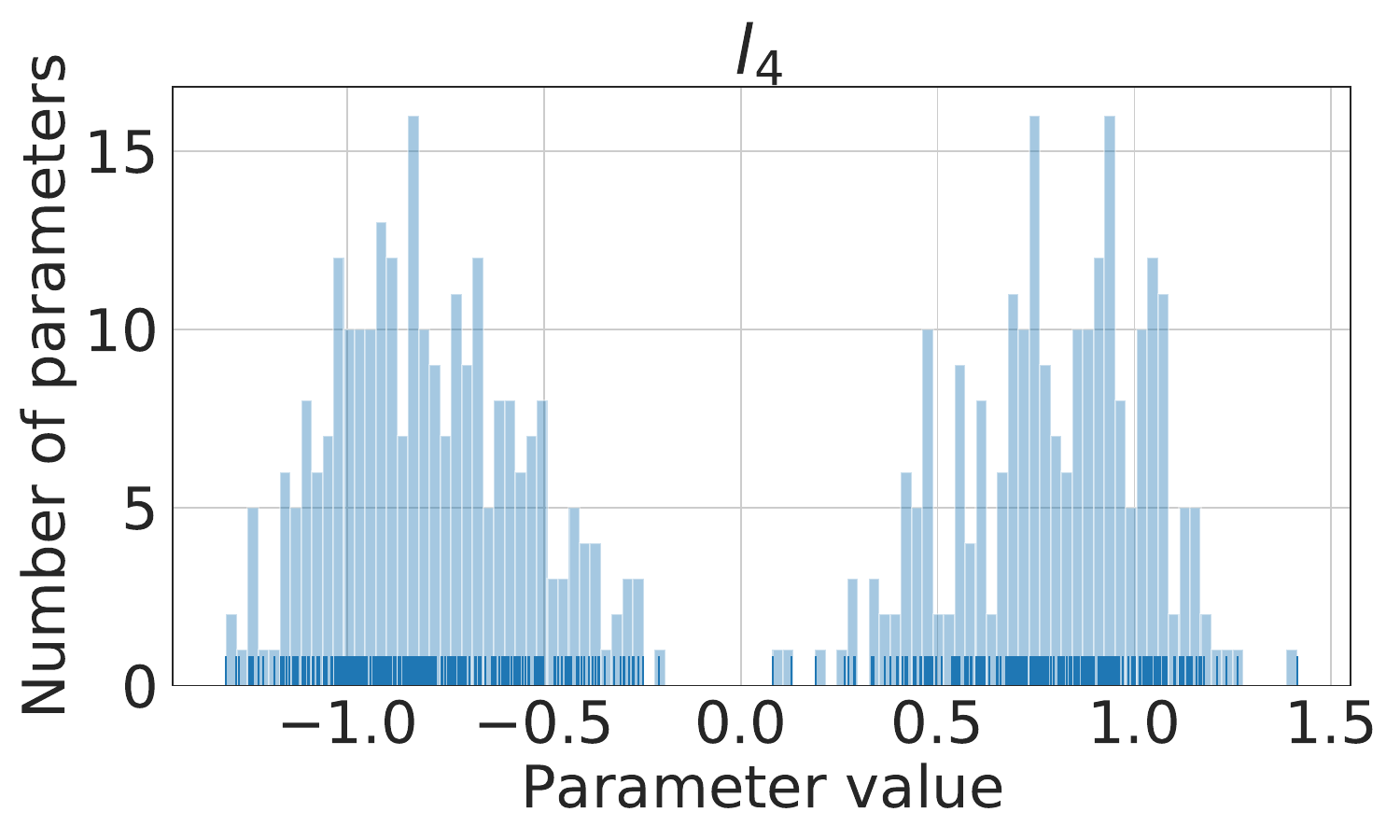} 
         \put(5, 60){\textbf{D}}
         \end{overpic}\\
         \begin{overpic}[width=.47\textwidth]{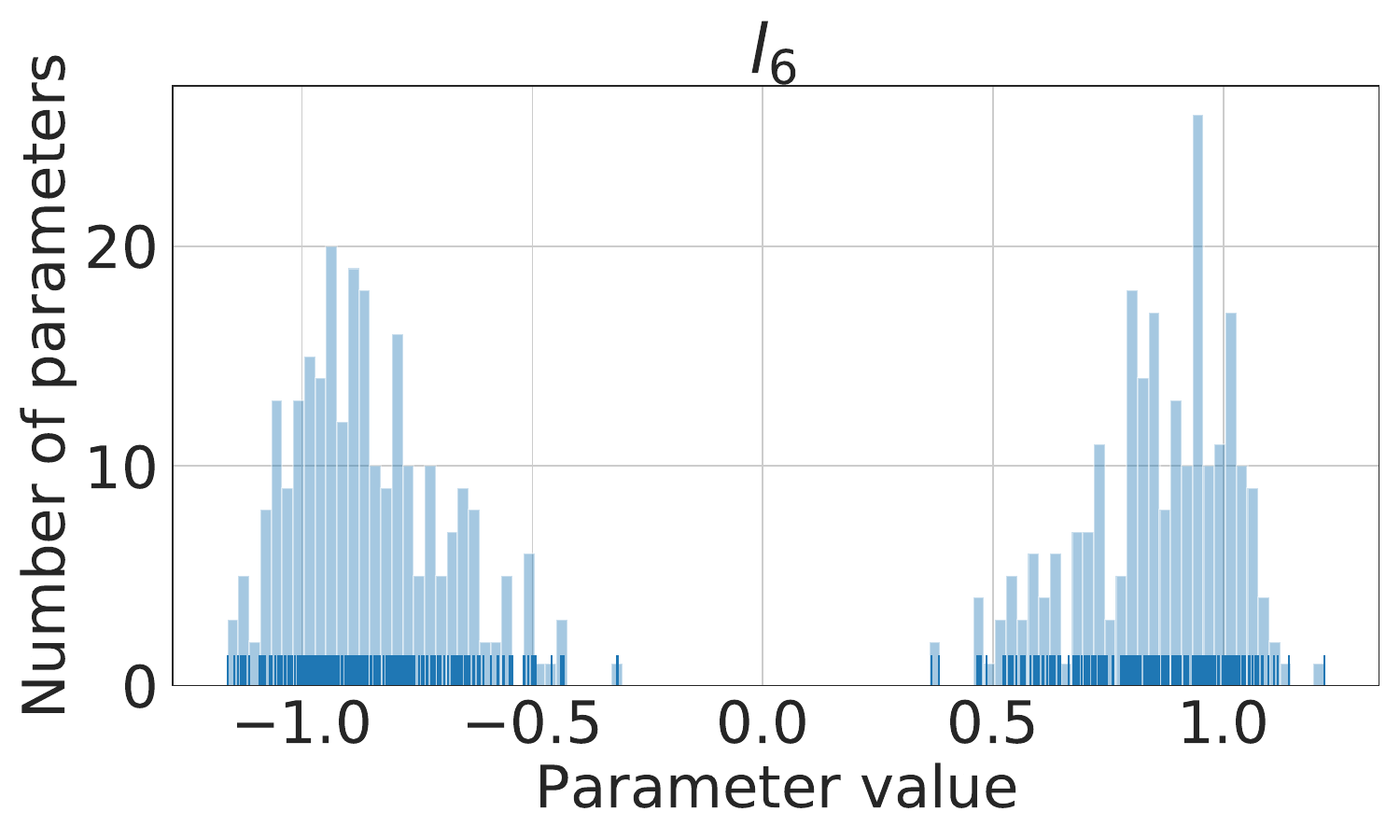} 
         \put(5, 60){\textbf{E}}
         \end{overpic}&
         \begin{overpic}[width=.47\textwidth]{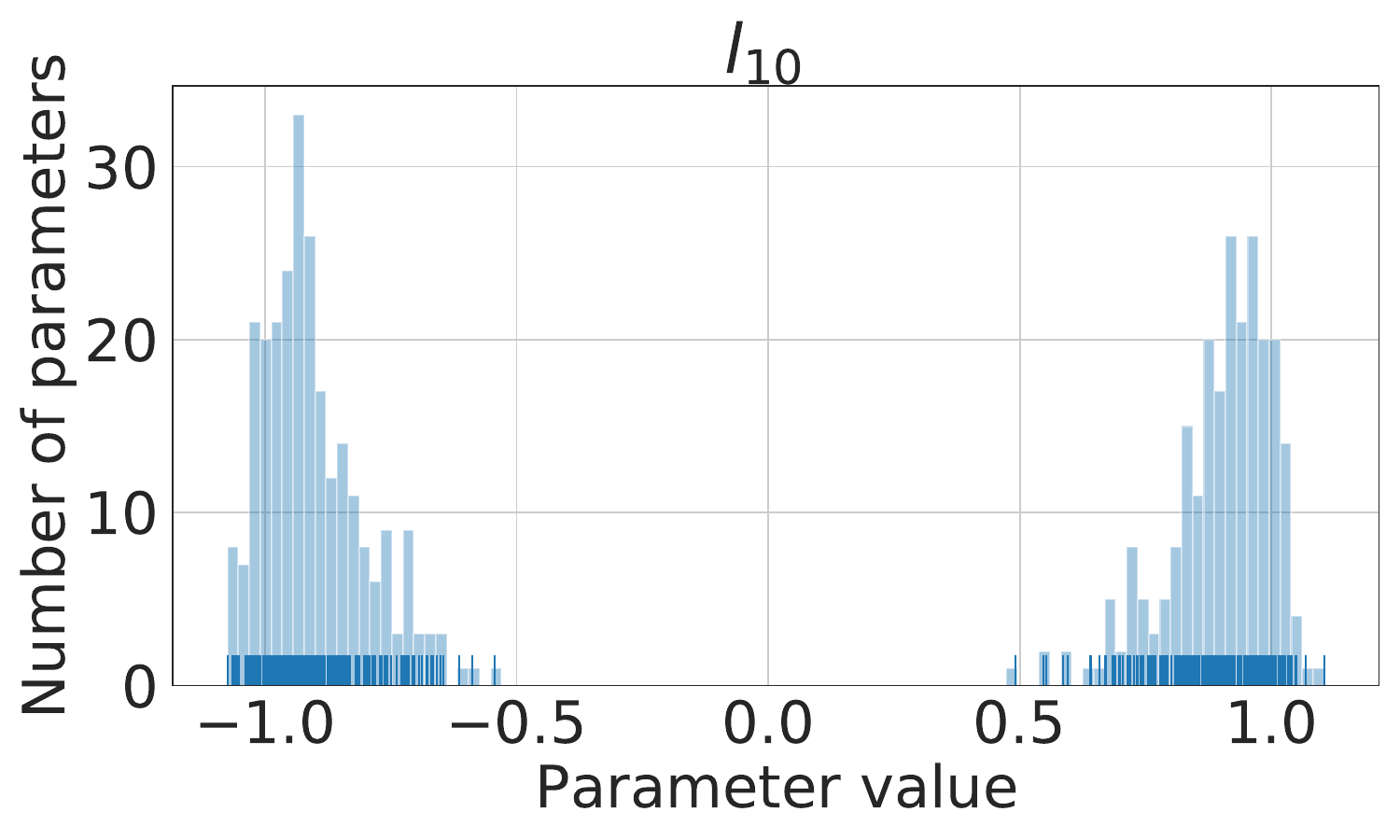}
         \put(5, 60){\textbf{F}}
         \end{overpic}
    \end{tabular}
    \caption{\textbf{Parameter histograms.} (A) True parameters $\ba$. (B) Parameter vector found by the algorithm with $\psi(\cdot) = \frac{1}{2}\Vert\cdot\Vert_{1.1}^2$. The resulting solution is extremely sparse, and has a few parameters with large magnitude, indicative of implicit $\ell_1$ regularization. (C) Parameter vector found by the standard Euclidean algorithm with $\psi(\cdot) = \frac{1}{2}\Vert\cdot\Vert_{2}^2$. The resulting parameter vector looks approximately Gaussian distributed, indicating $\ell_2$ regularization. (C)-(F) Parameter vectors found by $\psi(\cdot) = \frac{1}{2}\Vert\cdot\Vert_p^2$ with $p = 4, 6$, and $10$ respectively. The transition clearly indicates a trend towards $\ell_\infty$-norm regularization, with two bimodal peaks forming around $\pm 1$. The $\ell_\infty$ norm of the parameter vector decreases with increasing $p$.}
    \label{fig:hist}
\end{figure}

\begin{figure}
    \begin{tabular}{cc}
        \begin{overpic}[width=.47\textwidth]{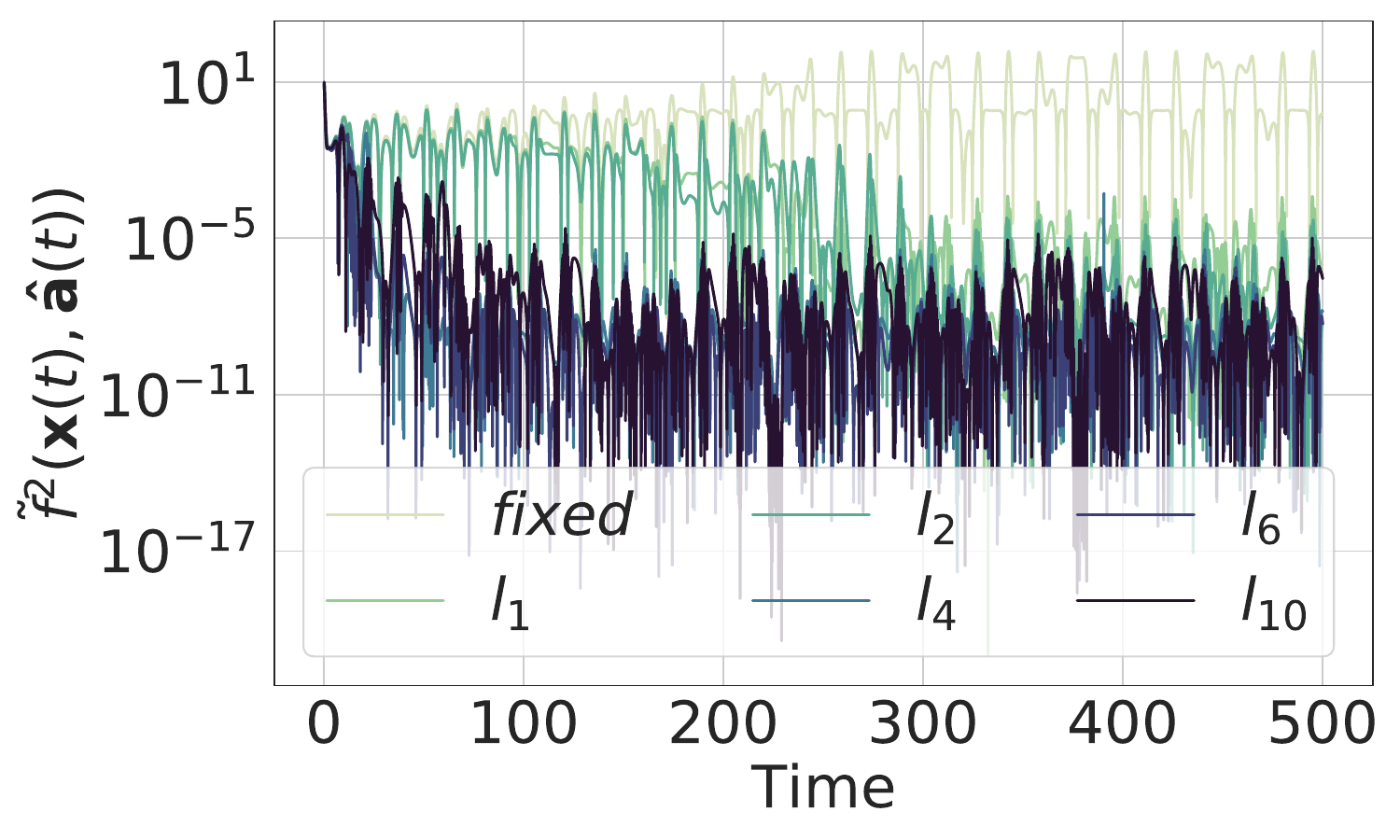}
        \put(5, 60){\textbf{A}}
        \end{overpic} &
         \begin{overpic}[width=.47\textwidth]{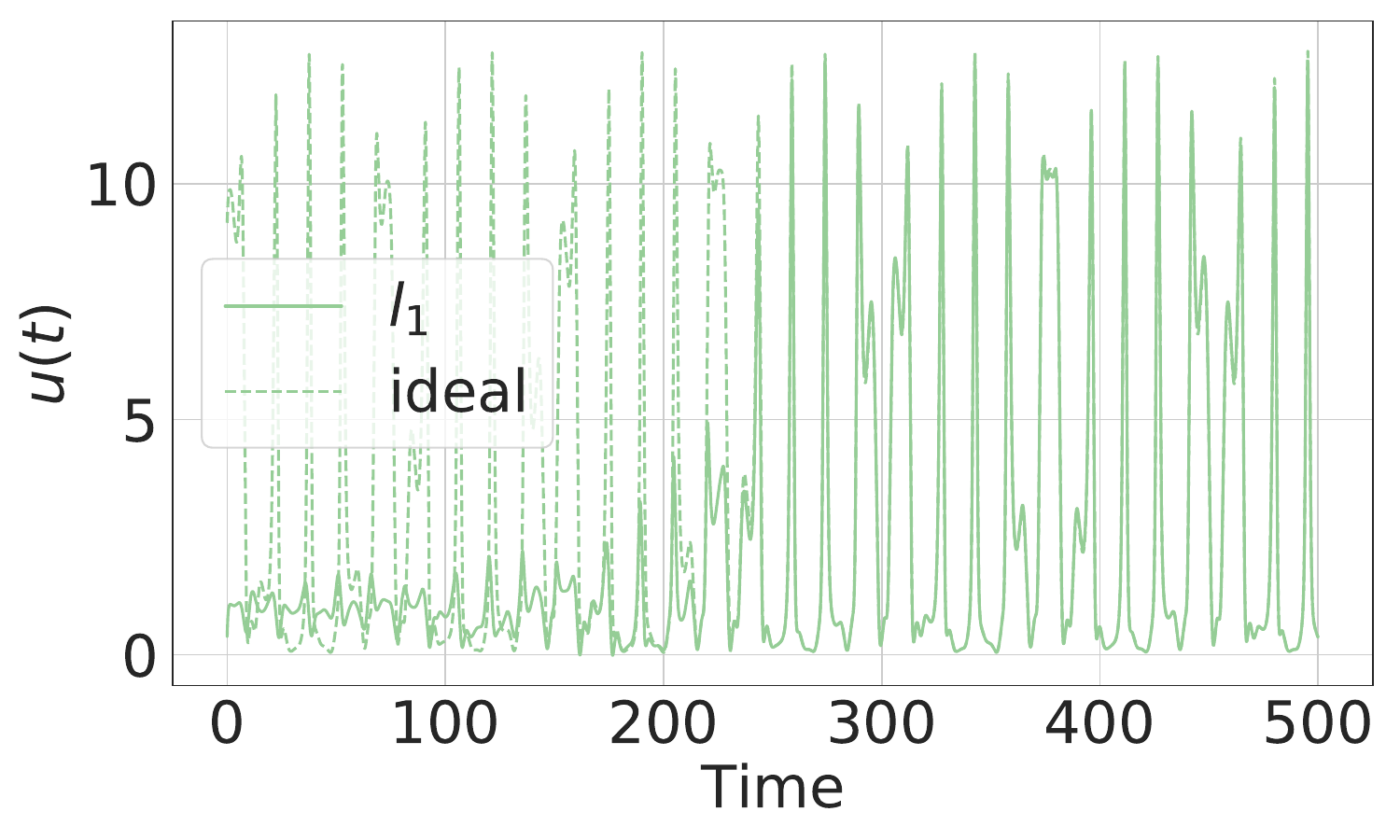}
         \put(5, 60){\textbf{B}}
         \end{overpic}
         \\
         \begin{overpic}[width=.47\textwidth]{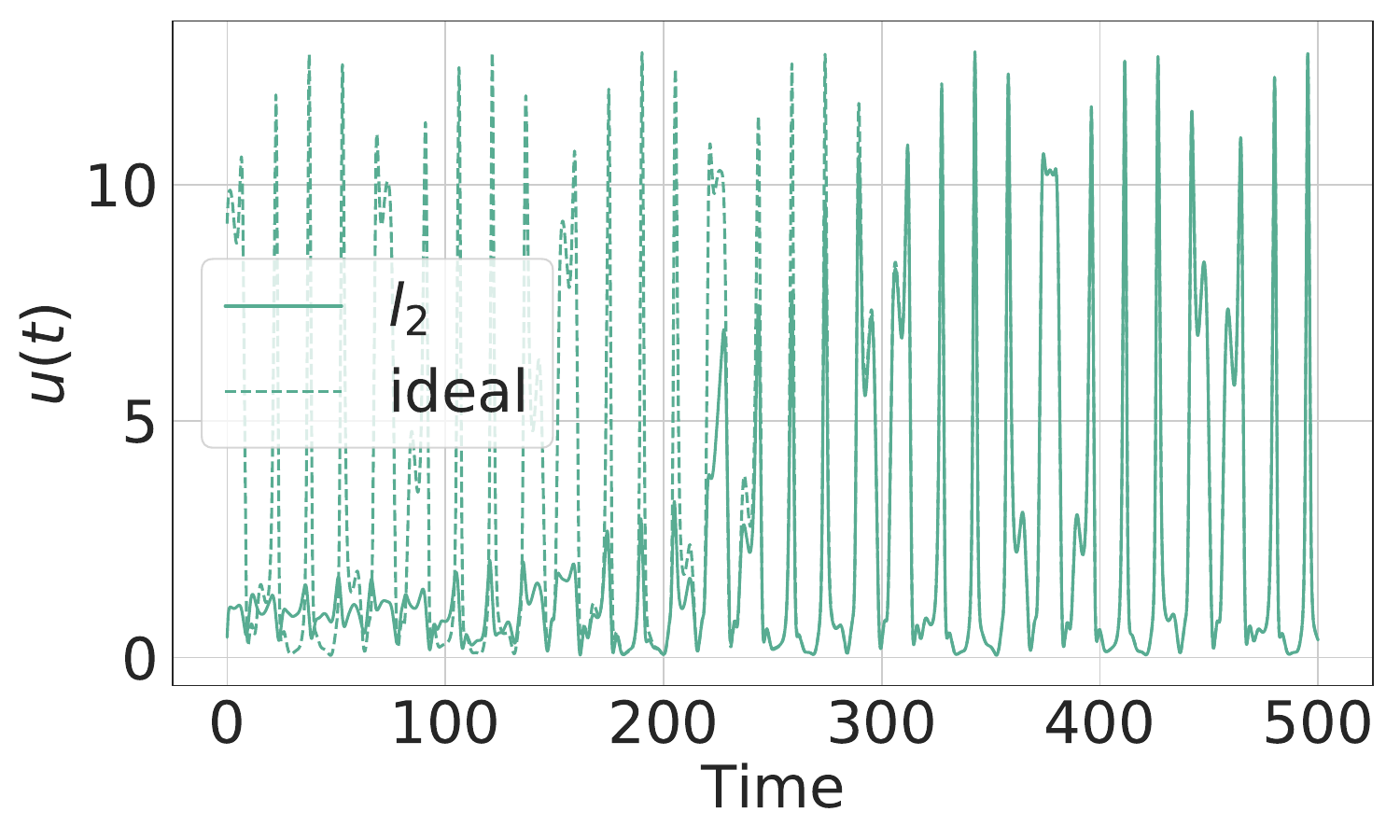} 
         \put(5, 60){\textbf{C}}
         \end{overpic}&
         \begin{overpic}[width=.47\textwidth]{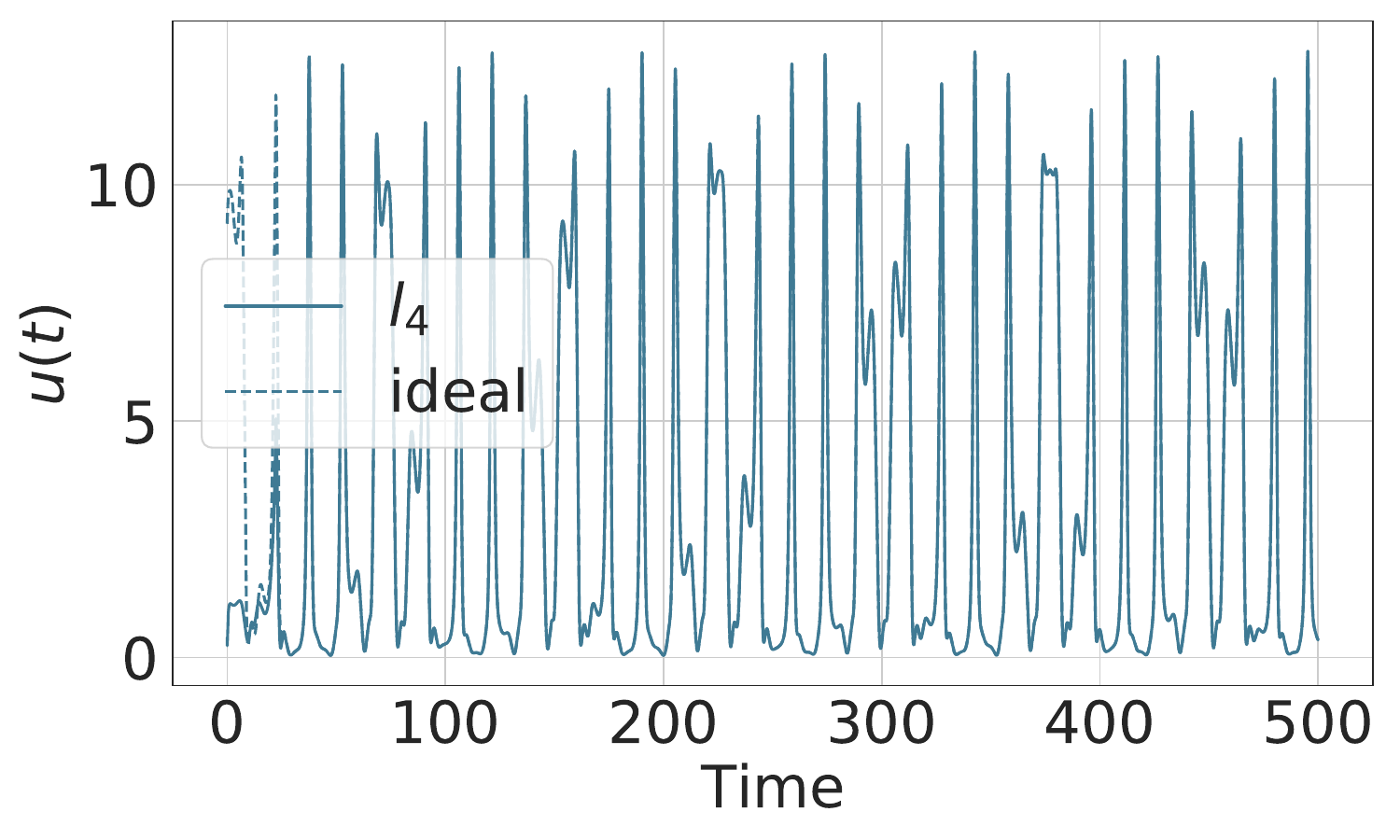} 
         \put(5, 60){\textbf{D}}
         \end{overpic}\\
         \begin{overpic}[width=.47\textwidth]{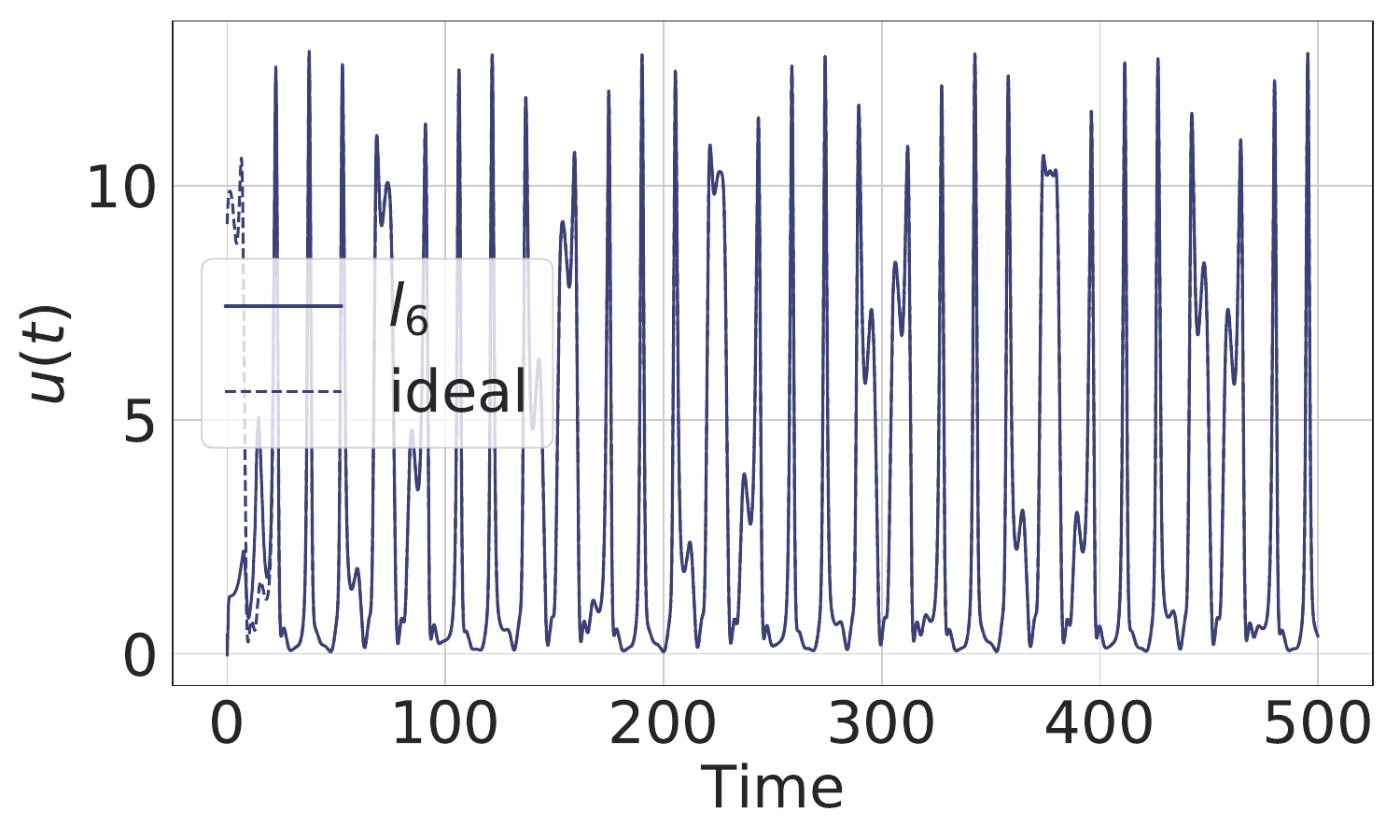} 
         \put(5, 60){\textbf{E}}
         \end{overpic}&
         \begin{overpic}[width=.47\textwidth]{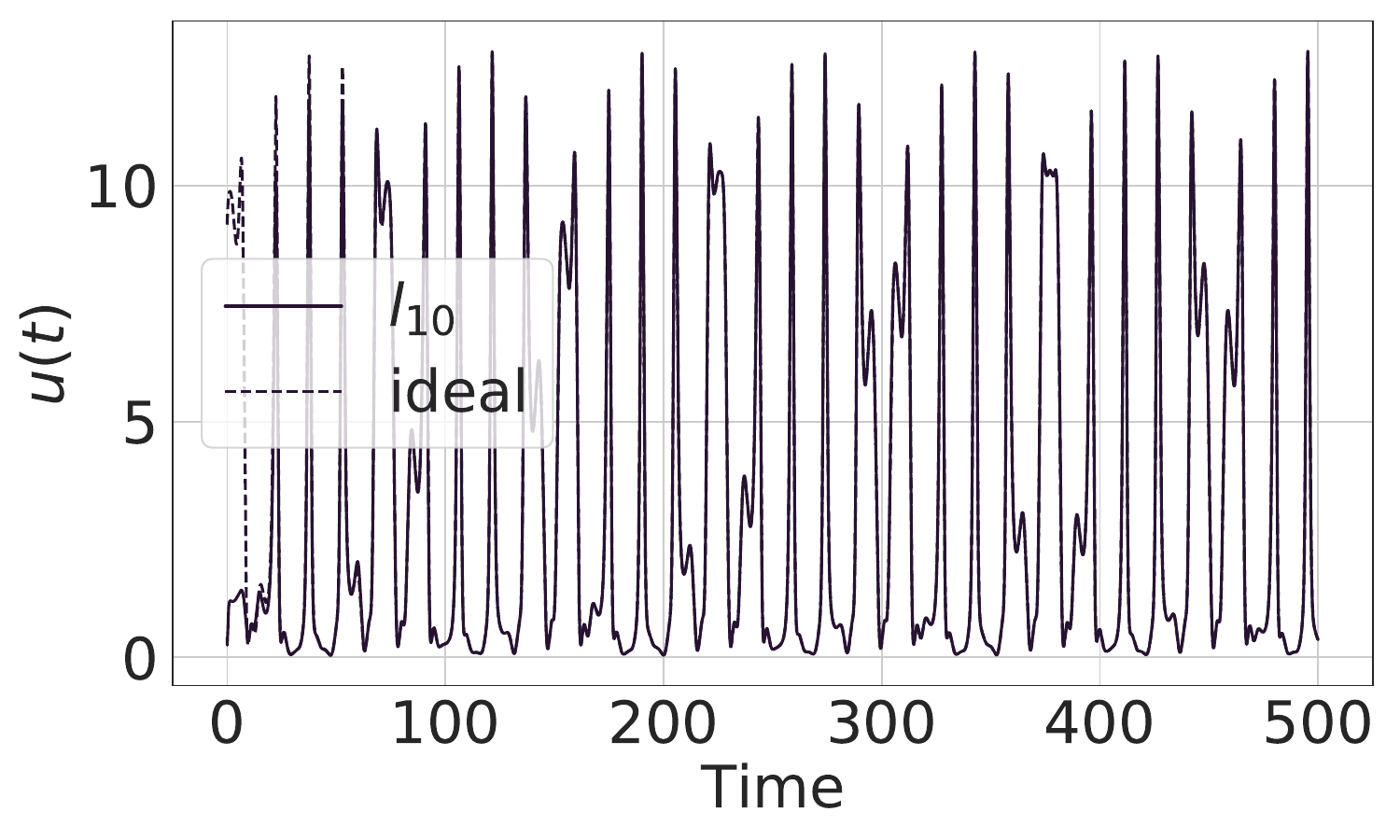}
         \put(5, 60){\textbf{F}}
         \end{overpic} 
    \end{tabular}
    \caption{\textbf{Function approximation error and control inputs.} (A) The function approximation error $\tilde{f}^2(\bx(t), \ha(t))$. All algorithms drive the error to zero. (B)-(F) Comparison of the control input $u^\psi(t)$ to the ``ideal'' control $u(t) = \ddot{x}_d(t) + f(\bx_d(t), \ba)$. All algorithms converge to the ideal control, though at a different rate. The control magnitude is kept to a reasonable level in every case.}
    \label{fig:control}
\end{figure}

\subsection{Learning to control with primitives}
\label{ssec:sim_entropy}
It is well-known that mirror descent with respect to the entropy $\psi\left(\hat{\boldsymbol\beta}\right) = \sum_i \hat{\beta}_i \log\hat{\beta}_i$ can improve the dimension-dependence of convergence rates in comparison to projected gradient descent when optimizing over the simplex $\Theta = \left\{\hat{\beta}_i : \hat{\beta}_i \geq 0, \: \: \: \sum_i\hat{\beta}_i = 1\right\}$~\citep{hazan}. Here we demonstrate that the same phenomenon appears in adaptive control.

As a model problem for this setting, we consider the second-order system 
\begin{align*}
    \dot{x}_1 &= x_2,\\
    \dot{x}_2 &= u(\bx, \ha) - \tanh\left(\bV\bx\right)^\T\ba,
\end{align*}
with $\ba \in \mathbb{R}^{p}$ a fixed vector of unknown parameters and $\bV \in \mathbb{R}^{p\times 2}$ a random matrix with $V_{ij} \sim \mathcal{N}\left(0, \frac{1}{p^2}\right)$. We assume access to control primitives $\left\{u_i(\bx, \ha^{(i)})\right\}_{i=1}^N$ capable of tracking desired trajectories $\left\{x_d^{(i)}\right\}_{i=1}^N$, and we consider tracking a desired trajectory $x_d(t)$ given piecewise by randomly drawn $x_d^{(i)}$. Details on the definition of the individual desired trajectory, the control inputs $u_i(\bx, \ha^{(i)})$, and the randomly specified desired trajectory used for the comparison can be found in Appendix~\ref{app2:sim_entropy}.

To leverage the available control primitives, we use the input
\begin{equation*}
    u\left(\bx, \hat{\boldsymbol{\beta}}, \left\{\ha^{(i)}\right\}_{i=1}^N\right)= \sum_{i=1}^N \hat{\beta}_i u_i(\bx, \ha^{(i)}) = \bu\left(\bx, \left\{\ha^{(i)}\right\}_{i=1}^N\right)\hat{\boldsymbol{\beta}}.
\end{equation*}
Above, $\bu\left(\bx, \left\{\ha^{(i)}\right\}_{i=1}^N\right)\in\mathbb{R}^{1\times N}$ is a row vector with components $u_i(\bx, \ha^{(i)})$.

We consider two adaptation laws,
\begin{align*}
    \dot{\hat{\boldsymbol\beta}} = -\gamma \bu\left(\bx, \left\{\ha^{(i)}\right\}_{i=1}^N\right)^\T s(\bx, \bx_d(t)),\\
    \frac{d}{dt}\nabla\psi\left(\hat{\boldsymbol\beta}\right) = -\gamma\bu\left(\bx, \left\{\ha^{(i)}\right\}_{i=1}^N\right)^\T s(\bx, \bx_d(t)),
\end{align*}
with projection of $\hat{\boldsymbol\beta}$ onto the simplex. In both cases, we initialize $\hat{\beta}_i = \frac{1}{N}$ and set $\gamma = .25$.

The results are shown in Figure~\ref{fig:entropy}. Figure~\ref{fig:entropy}A shows convergence of $s(\bx(t), \bx_d(t))$ for both adaptive laws. $s(\bx(t), \bx_d(t))$ jumps every $200$ units of time as the task changes discretely. While both converge, adaptation with respect to the entropy converges significantly faster and minimizes $s(\bx(t), \bx_d(t))$ to lower values. This effect is more prominently displayed in Figure~\ref{fig:entropy}B, which shows convergence of $s(\bx(t), \bx_d(t))$ on a logarithmic scale. Figures~\ref{fig:entropy}C and D show the parameter trajectories for the mirror descent and Euclidean laws, respectively. The mirror descent law displays trajectories in which fewer parameters stray from zero. Those that do drift from zero stray an order of magnitude further than the Euclidean law. The discrete changes of the desired trajectory are more visible in the parameter trajectories for the mirror descent law, indicating a faster reaction to the task.
\begin{figure}
    \begin{tabular}{cc}
        \begin{overpic}[width=.47\textwidth]{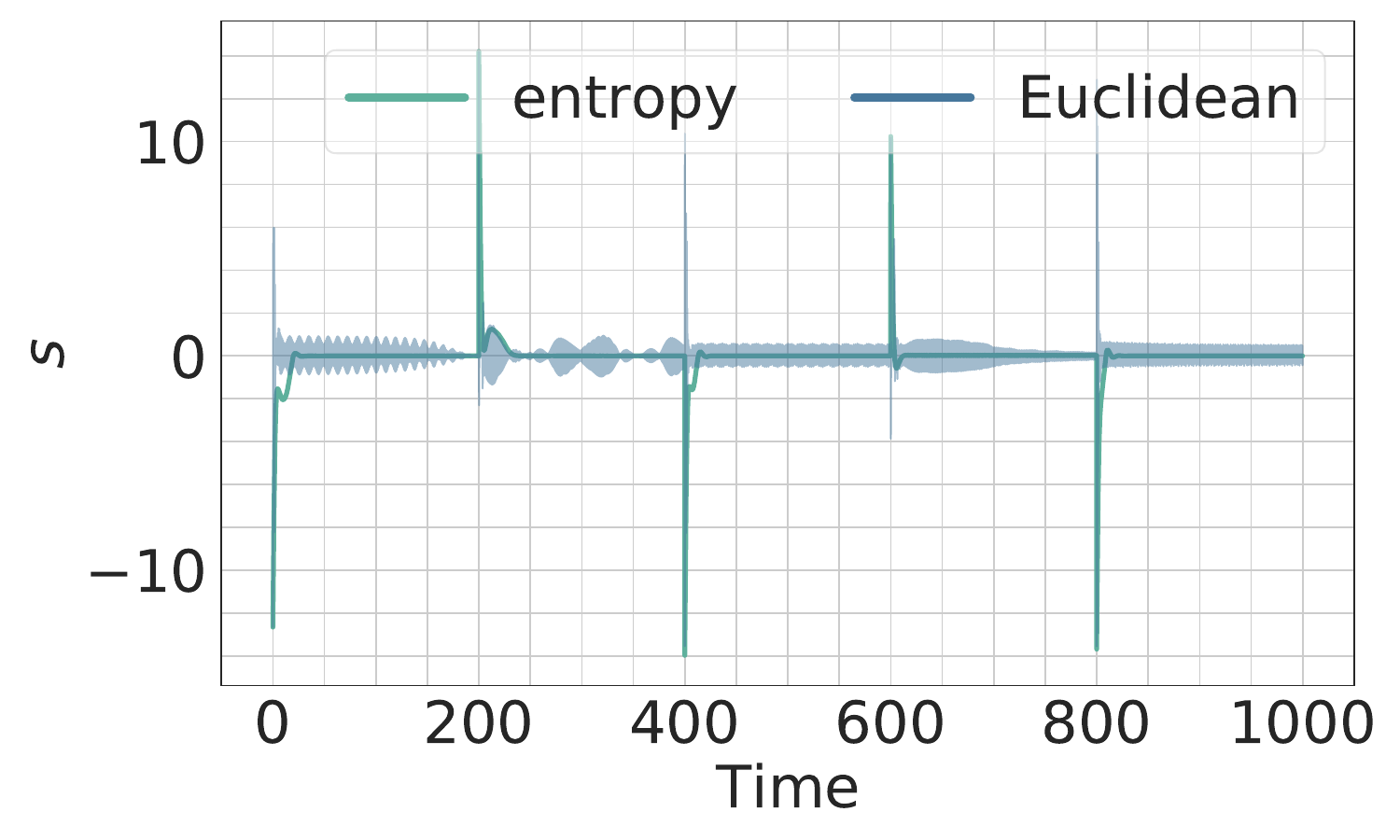}
        \put(5, 60){\textbf{A}}
        \end{overpic} &
         \begin{overpic}[width=.47\textwidth]{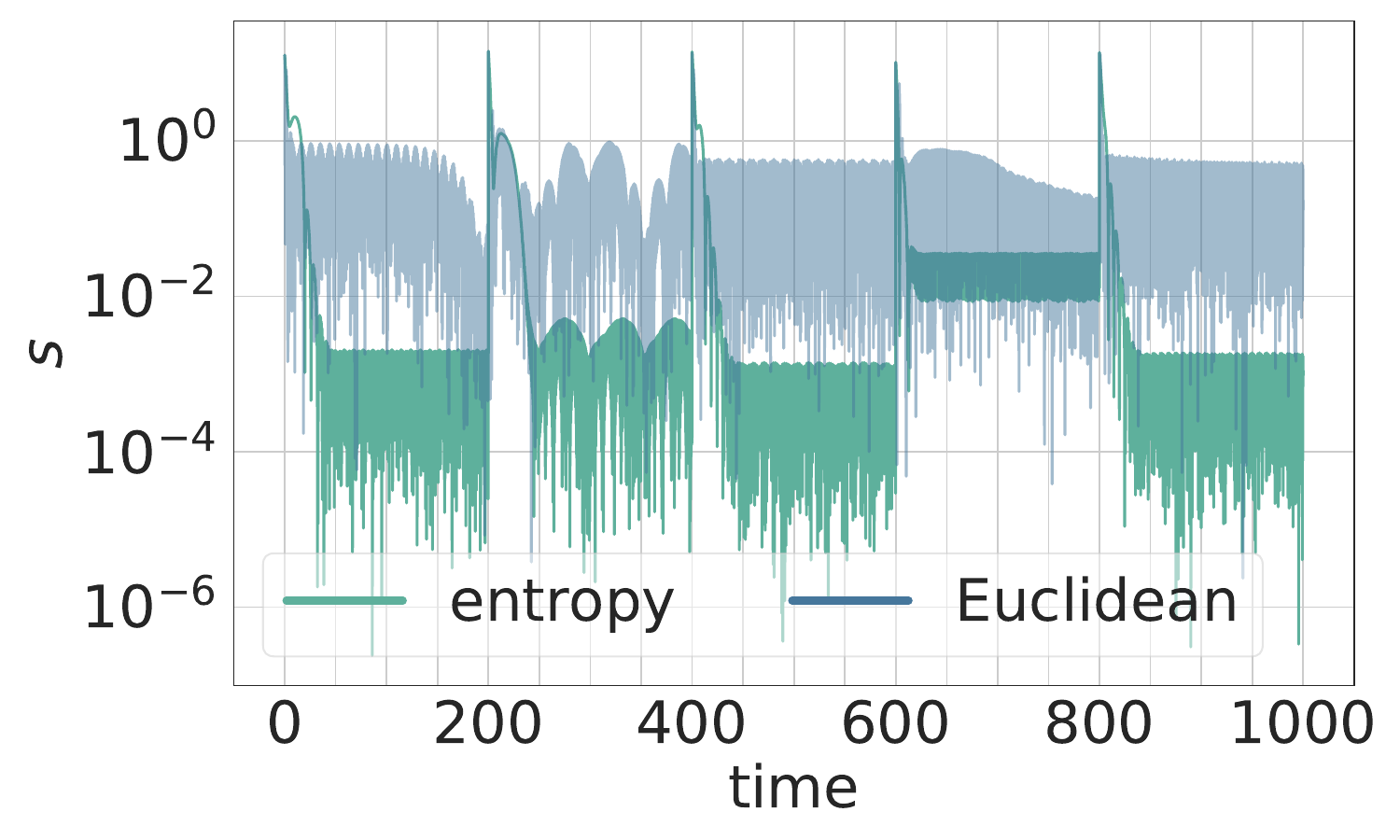}
         \put(5, 60){\textbf{B}}
         \end{overpic}
         \\
         \begin{overpic}[width=.47\textwidth]{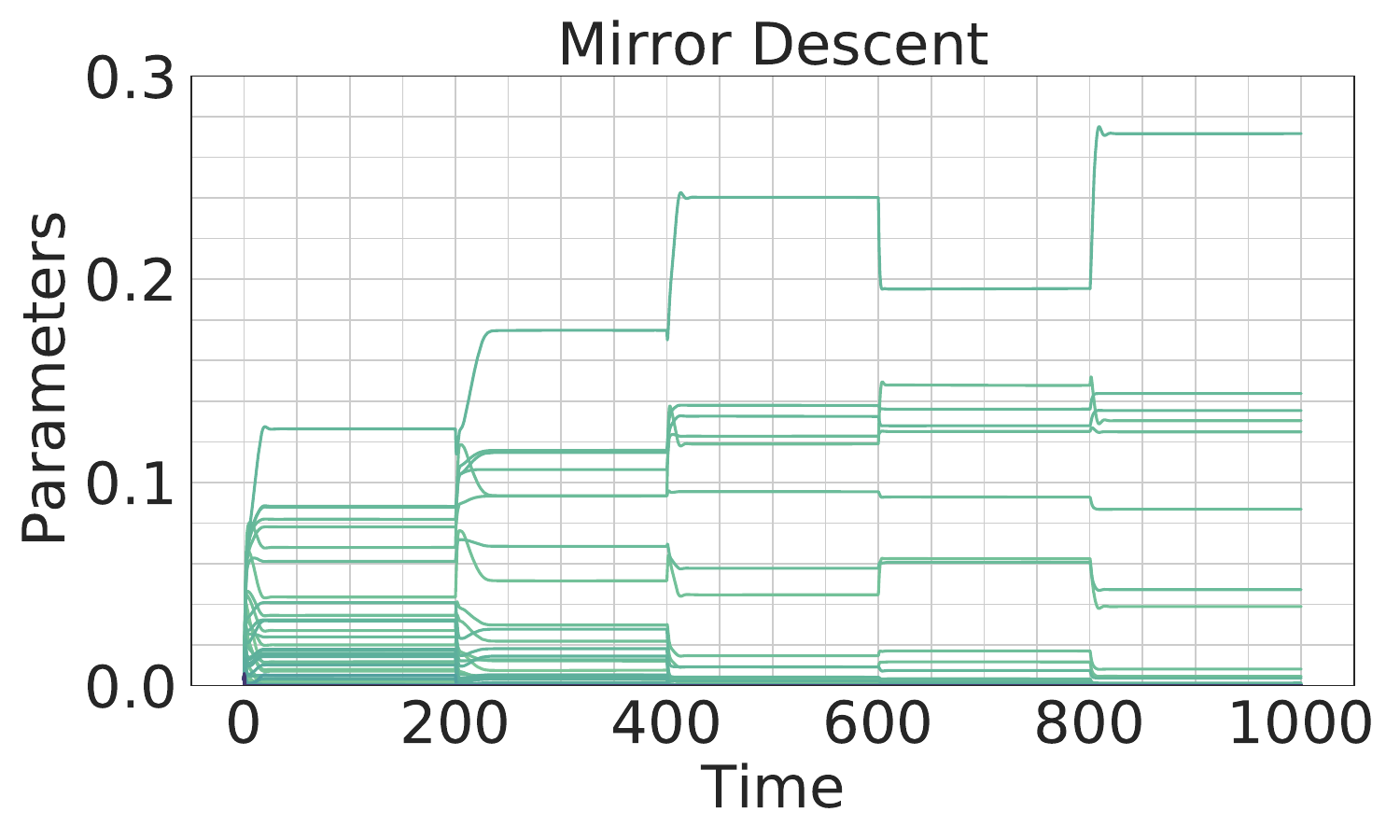}
         \put(5, 60){\textbf{C}}
         \end{overpic}&
         \begin{overpic}[width=.47\textwidth]{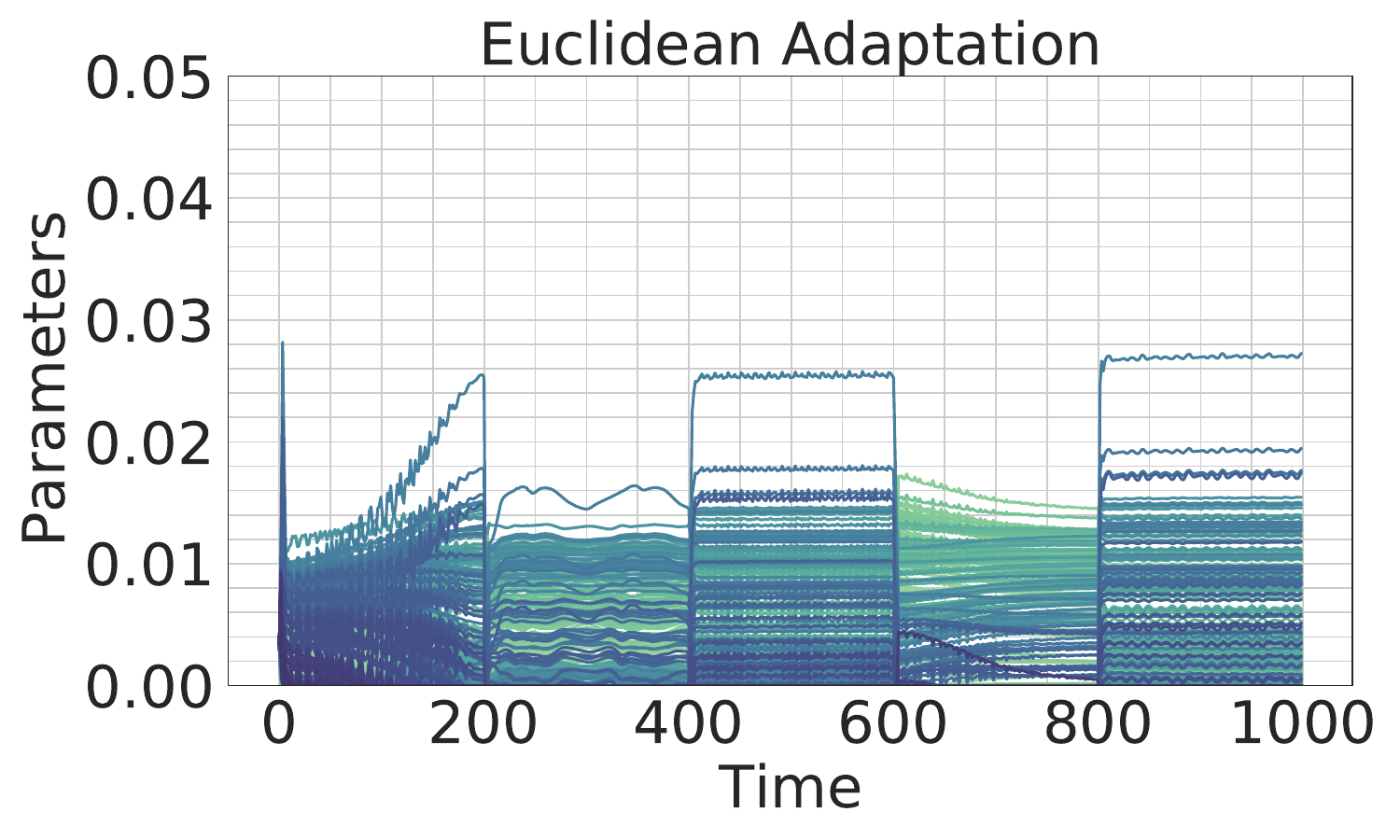}
         \put(5, 60){\textbf{D}}
         \end{overpic}
    \end{tabular}
    \caption{\textbf{Learning to control with primitives.} (A) $s$ on a linear scale for the Euclidean and mirror descent-like adaptation laws. Mirror descent converges faster for all tasks. (B) $s$ on a logarithmic scale for the Euclidean and mirror descent-like adaptation laws. Mirror descent converges faster and minimizes $s$ further for all tasks. (C/D) Parameter trajectories for the mirror descent and Euclidean adaptation laws. Mirror descent leads to smoother trajectories with fewer parameters straying from $0$.}
    \label{fig:entropy}
\end{figure}
\subsection{Dynamics prediction for Hamiltonian systems}
\label{ssec:exp_ham}
Similar to~\citet{bottou}, consider the Hamiltonian for three point masses interacting in $d=2$ dimensions via Newtonian gravitation (in units such that the gravitational constant $G =1$),
\begin{equation}
    \label{eqn:3body_ham}
    \mathcal{H} = \frac{1}{2m_1}\Vert\bp_1\Vert^2 + \frac{1}{2m_2}\Vert\bp_2\Vert^2 + \frac{1}{2m_3}\Vert\bp_3\Vert^2 - \frac{m_1 m_2}{\Vert \bq_1 - \bq_2\Vert} - \frac{m_1 m_3}{\Vert \bq_1 - \bq_3\Vert} - \frac{m_2 m_3}{\Vert \bq_2 - \bq_3\Vert},
\end{equation}
with $m_i$ the mass of body $i$, $\bp_i$ the momentum of body $i$, and $\bq_i$ the position of body $i$. Denote by $\bq$ the vector $\left(\bq_1^\T, \bq_2^\T, \bq_3^\T\right)^\T$ with similar notation for $\bp$. We consider estimating the Hamiltonian \eqref{eqn:3body_ham} directly with a physically motivated overparameterized basis
\begin{equation*}
    \hat{\mathcal{H}}(\ha) = \bY(\bq, \bp)\ha
\end{equation*}
to form the dynamics predictor \eqref{eqn:ham_pred_p}~\&~\eqref{eqn:ham_pred_q}, and choose $\psi(\cdot) = \frac{1}{2}\norm{\cdot}_{1.05}^2$ to promote sparsity. Further simulation details are given in Appendix~\ref{app2:exp_ham}

Results are shown in Figure~\ref{fig:ham}. Figure~\ref{fig:ham}A displays convergence of the predictor trajectory $\hat{\bx}(t)$ (solid) to the true trajectory $\bx(t)$ (open circles). The simulation is run open-loop and all parameters with magnitude less than $10^{-3}$ are set equal to zero past $t=10$, indicated with a vertical line. The open-loop predictor maintains the correct oscillatory behavior. Figure~\ref{fig:ham}B displays convergence of $\tilde{\bx}(t)$ to zero with adaptation (top) and demonstrates that adaptation is necessary for convergence (bottom). When switching to the open-loop predictor past $t=10$, the system without adaptation sustains large errors, while the learned predictor maintains good performance. Figure~\ref{fig:ham}C shows the qualitative structure of the predictor trajectory without learning (solid): when run closed-loop, it does not converge to the true trajectory (open circles), and when run open-loop, it spuriously converges to a fixed point. Figures~\ref{fig:ham}D and E show parameter trajectories and asymptotically converged parameters, respectively. Together, the two panes demonstrate that the implicit bias of the algorithm ensures convergence to a relatively sparse estimate of the system Hamiltonian.
\begin{figure}
    \begin{tabular}{cc}
        \begin{overpic}[width=.47\textwidth]{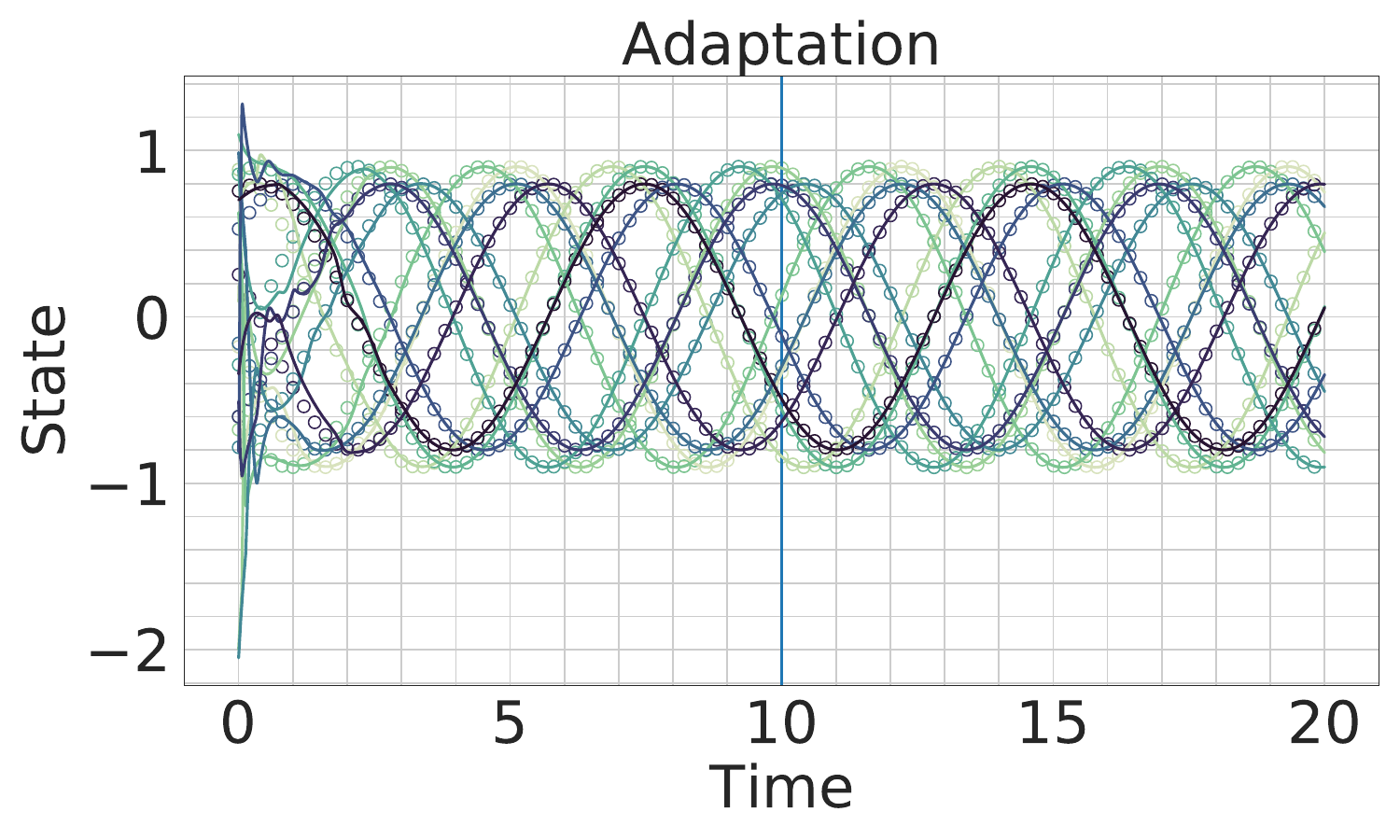}
        \put(5, 60){\textbf{A}}
        \end{overpic} &
         \begin{overpic}[width=.47\textwidth]{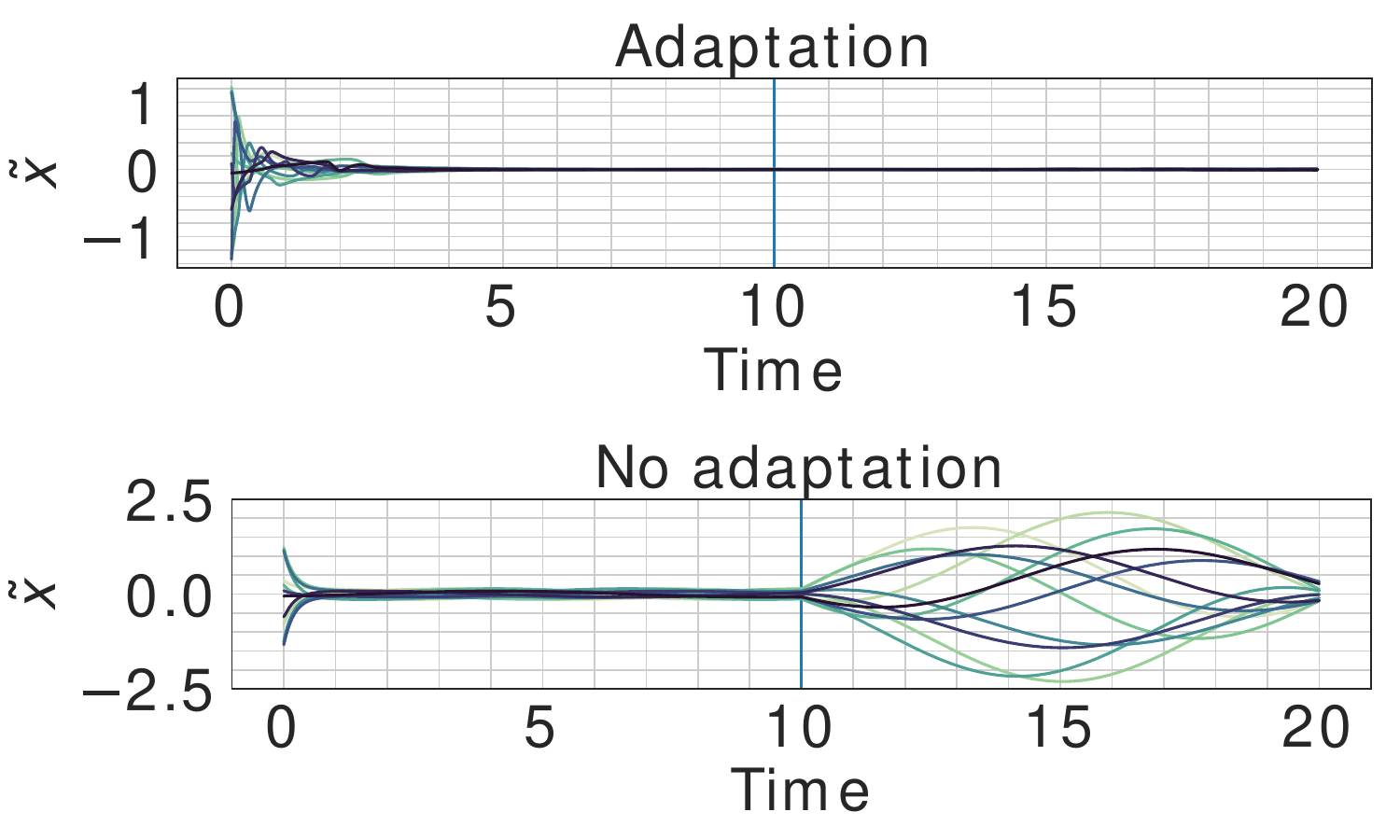}
         \put(5, 60){\textbf{B}}
         \end{overpic} \\
         \begin{overpic}[width=.47\textwidth]{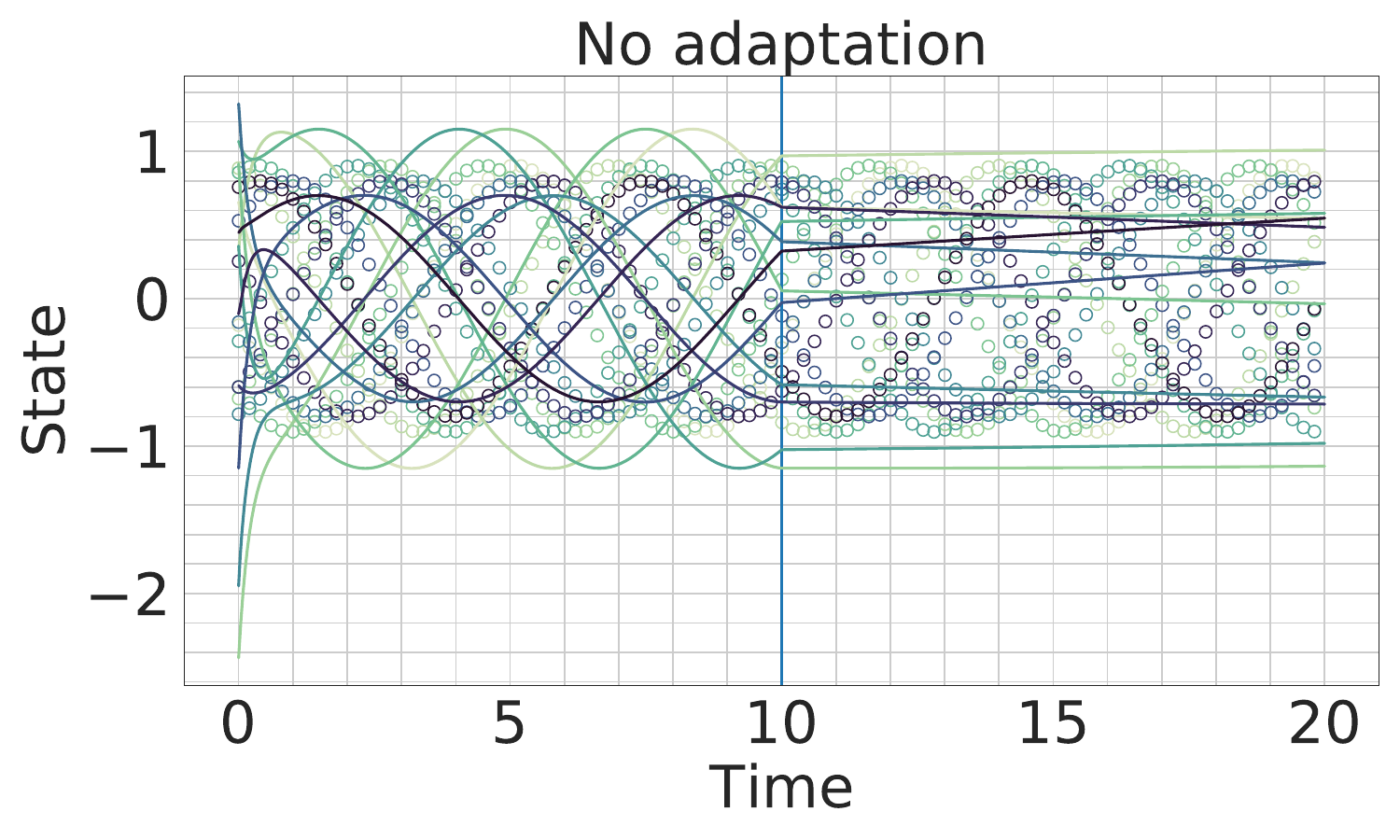}
         \put(5, 60){\textbf{C}}
         \end{overpic}&
         \begin{overpic}[width=.47\textwidth]{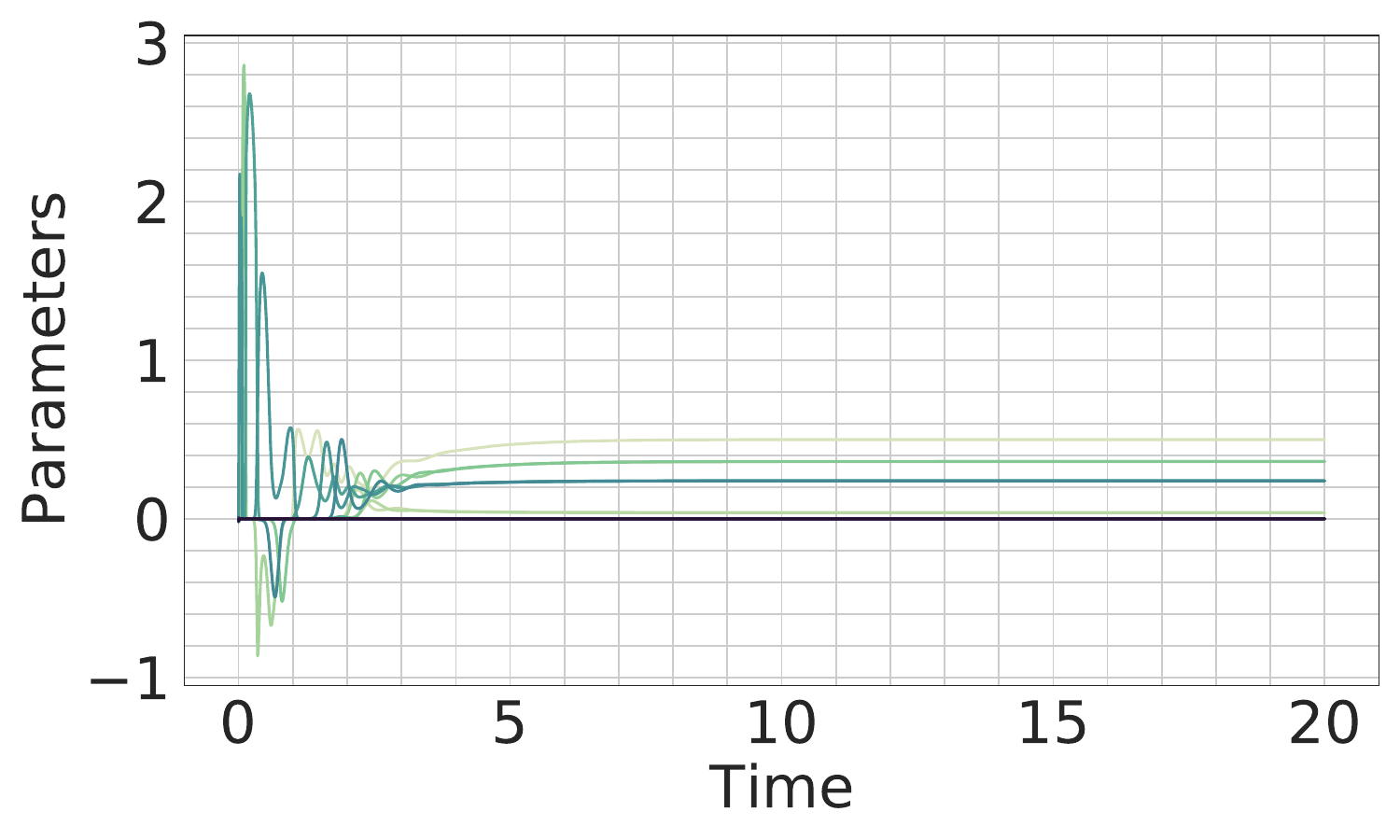}
         \put(5, 60){\textbf{D}}
         \end{overpic} \\
         \noalign{\vskip 5mm}
         \multicolumn{2}{c}{
         \begin{overpic}[width=.47\textwidth]{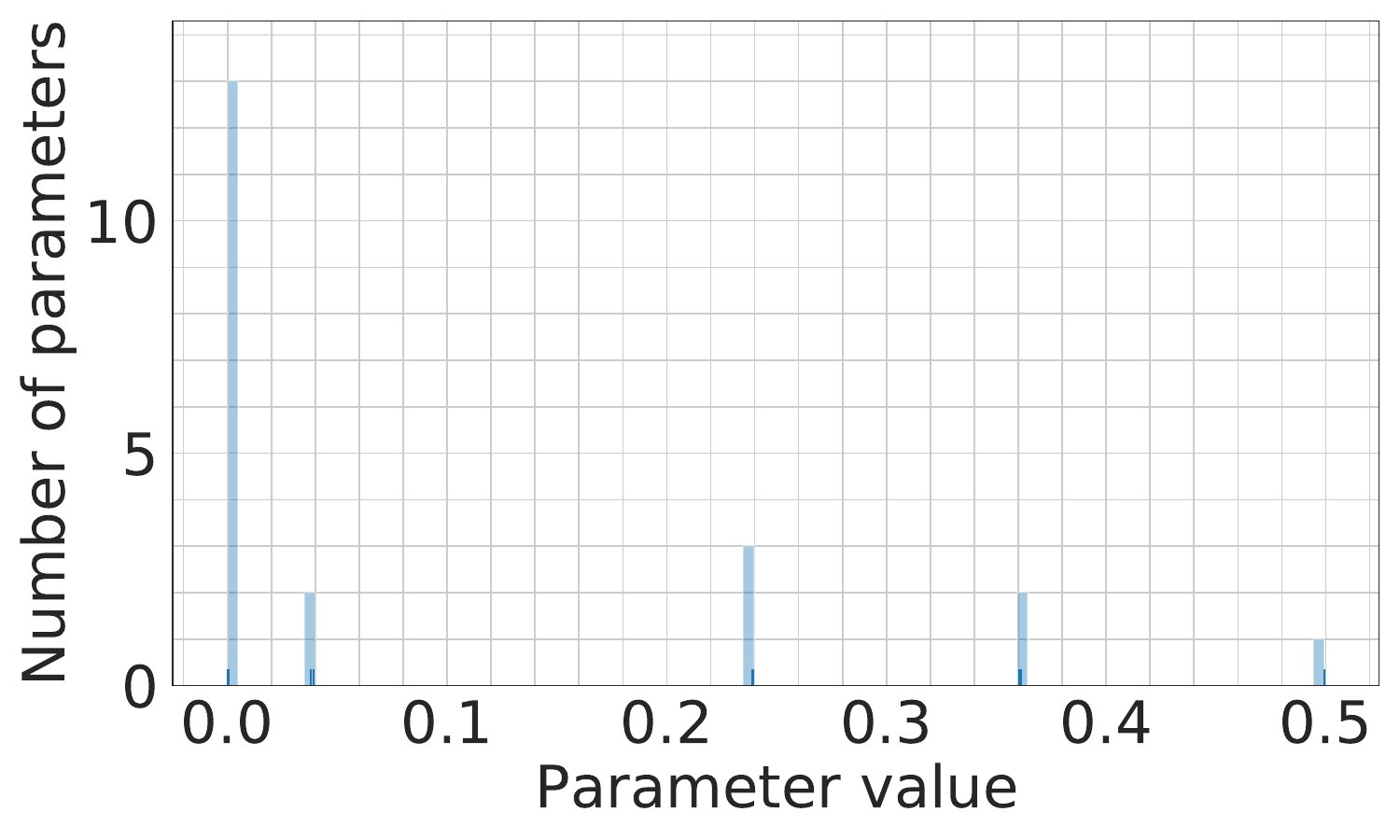}
         \put(5, 62){\textbf{E}}
         \end{overpic}
        }
    \end{tabular}
    \caption{\textbf{Three-body system.} (A) Convergence of the predictor with adaptation (solid) to the true system trajectory (open circles). Past the vertical line at $t=10$, coefficients below $10^{-3}$ are removed and the predictor is run open-loop, demonstrating that the limit cycle is correctly learned. (B) Predictor error $\tilde{\bx} = \hat{\bx} - \bx$ for the adaptive dynamics predictor \eqref{eqn:ham_pred_p}~\&~\eqref{eqn:ham_pred_q} with adaptation (top) and without adaptation (bottom). The system without adaptation is inaccurate when run open-loop, while the system with adaptation maintains accuracy. (C) Failure of the predictor trajectory without adaptation (solid) to converge to the true system trajectory (open circles). When run open-loop, the predictor without learning incorrectly tends to a fixed point. (D) Parameter trajectories for the adaptive dynamics predictor. Many parameters stay at or near zero, as predicted by Proposition~\ref{prop:implicit_reg}. (E) Histogram of final parameter values learned by the adaptive dynamics predictor.}
    \label{fig:ham}
\end{figure}
\subsection{Sparse identification of chemical reaction networks}
\label{ssec:exp_chem}
Consider a set of chemical reactions with $N$ distinct chemical species. Under the continuum hypothesis and the well-mixed assumption, mass-action kinetics dictates that the system dynamics can be described exactly in a monomial basis~\citep{network_obs},
\begin{align*}
    \nu_j(\bx) &= k_j\Pi_{i=1}^N x_i^{a_{ji}},\\
    \dot{\bx} &= \bGam \boldsymbol{\nu}(\bx),
\end{align*}
where $x_i$ is the concentration of chemical species $i$, $\bGam$ is the stoichiometric matrix, and the $a_{ji}$ are stoichiometric coefficients. Under the assumption that the full state of the network is measured, consider the adaptive dynamics predictor
\begin{align}
    \label{eqn:chem_obs_1}
    \dot{\hat{\bx}} &= \hat{\bGam}\hat{\boldsymbol{\nu}}(\hat{\bx}) + k\left(\bx - \hat{\bx}\right),\\
    \label{eqn:chem_obs_2}
    \frac{d}{dt}\nabla\psi\left(\hat{\bGam}\right) &= -\gamma (\hat{\bx} - \bx)\hat{\boldsymbol{\nu}}(\hat{\bx})^\T,
\end{align}
with $\gamma > 0$ a positive learning rate, $k > 0$ a measurement gain, $\psi$ a strongly convex function, $\hat{\bGam}$ an estimate of the stoichiometric matrix, and $\hat{\boldsymbol{\nu}}(\hat{\bx})$ a vector of monomial basis functions representing available knowledge of the system. Here we consider a four-species chemical reaction network (cf.~\citet{network_obs}, supplementary information),
\begin{align*}
    \dot{x}_1 &= -k_1 x_1 x_2,\\
    \dot{x}_2 &= -k_1 x_1 x_2 - k_2 x_2 x_3^2,\\
    \dot{x}_3 &= -k_1 x_1 x_2 - 2k_2 x_2 x_3^2,\\
    \dot{x}_4 &= k_2 x_2 x_3^2,
\end{align*}
with corresponding adaptive dynamics predictor \eqref{eqn:chem_obs_1}~\&~\eqref{eqn:chem_obs_2}. We set $\hat{\boldsymbol{\nu}}$ to be a vector of all monomials up to degree three comprising a total of $140$ candidate basis functions, and we set $\psi(\hat{\bGam}) = \frac{1}{2}\left\Vert\mathsf{vec}\left(\hat{\bGam}\right)\right\Vert_{1.01}^2$ to identify a sparse, parsimonious model consistent with the data. Searching over sparse models ensures that our learned predictor selects only a few relevant terms in the approximate system dynamics. We fix $k = 1.5$ and $\gamma = 0.25$ for $t < 10$. As in Section~\ref{ssec:exp_ham}, past $t = 10$ we set $k = \gamma = 0$ and run the predictor open-loop. We also perform shrinkage and set all coefficients with magnitude below $10^{-3}$ formally equal to zero, leaving $19$ remaining parameters.

Results are shown in Figure~\ref{fig:chemical}. In Figure~\ref{fig:chemical}A, we show convergence of the observer error to zero with adaptation (solid) and divergence away from zero without adaptation (dashed), demonstrating that adaptation is necessary for effective prediction. The inset displays a closer look at the asymptotic behavior of the open-loop dynamics predictor after shrinkage, which shows that the fixed point of the system is correctly learned. Figure~\ref{fig:chemical}B shows convergence of $\hat{\bx}(t)$ (solid) to $\bx(t)$ (dashed). Figure~\ref{fig:chemical}C displays parameter trajectories as a function of time. Many parameters stay at or near zero as predicted by Proposition~\ref{prop:implicit_reg}; use of a Euclidean law would lead to a Gaussian-like spread over parameter vectors as in Figure~\ref{fig:hist}C. The inset displays a finer-grained view around zero of the parameter trajectories. Figure~\ref{fig:chemical}D shows a histogram of the final parameter values learned by the adaptive dynamics predictor, demonstrating that only a few relevant terms are identified.

\begin{figure}
    \begin{tabular}{cc}
        \begin{overpic}[width=.47\textwidth]{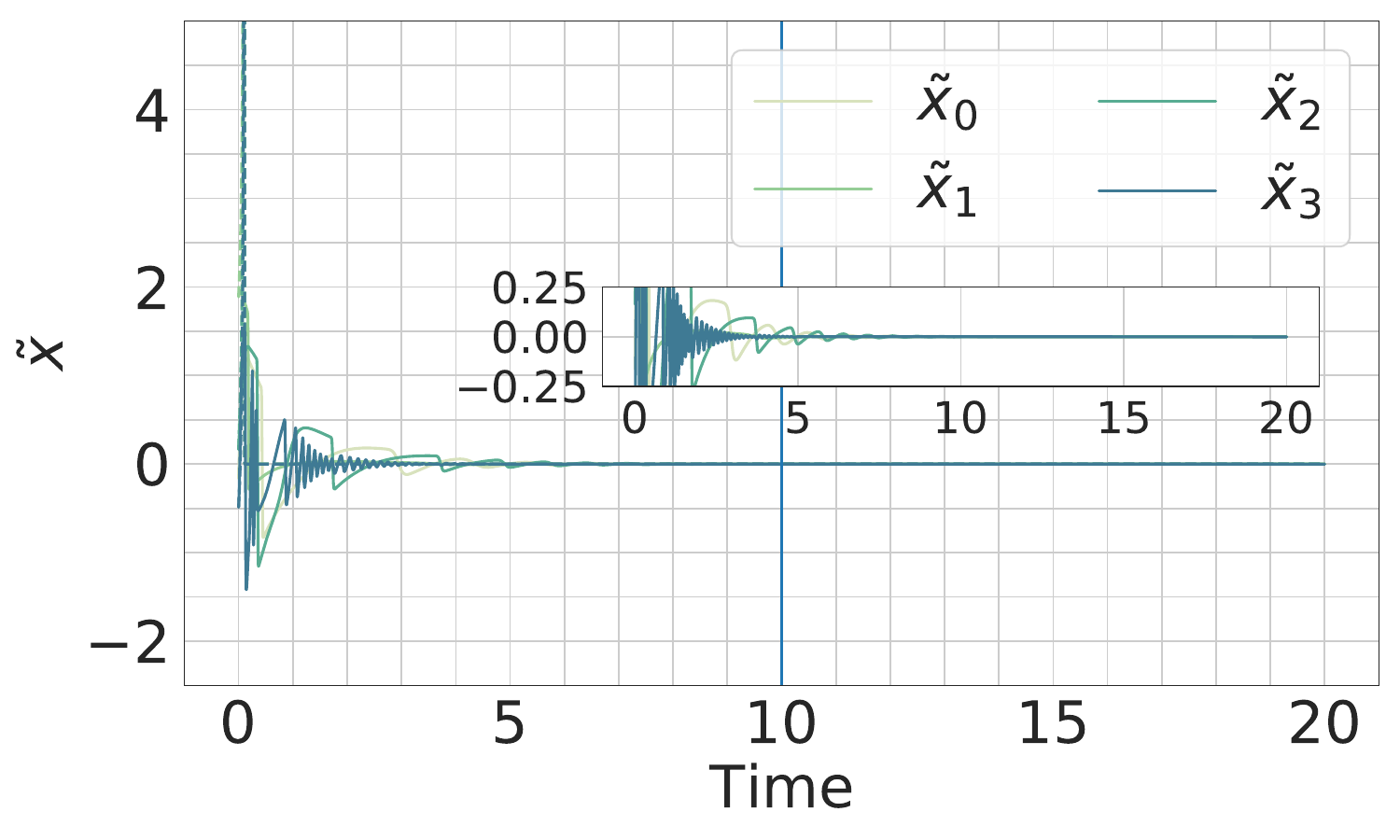}
        \put(5, 60){\textbf{A}}
        \end{overpic} &
         \begin{overpic}[width=.47\textwidth]{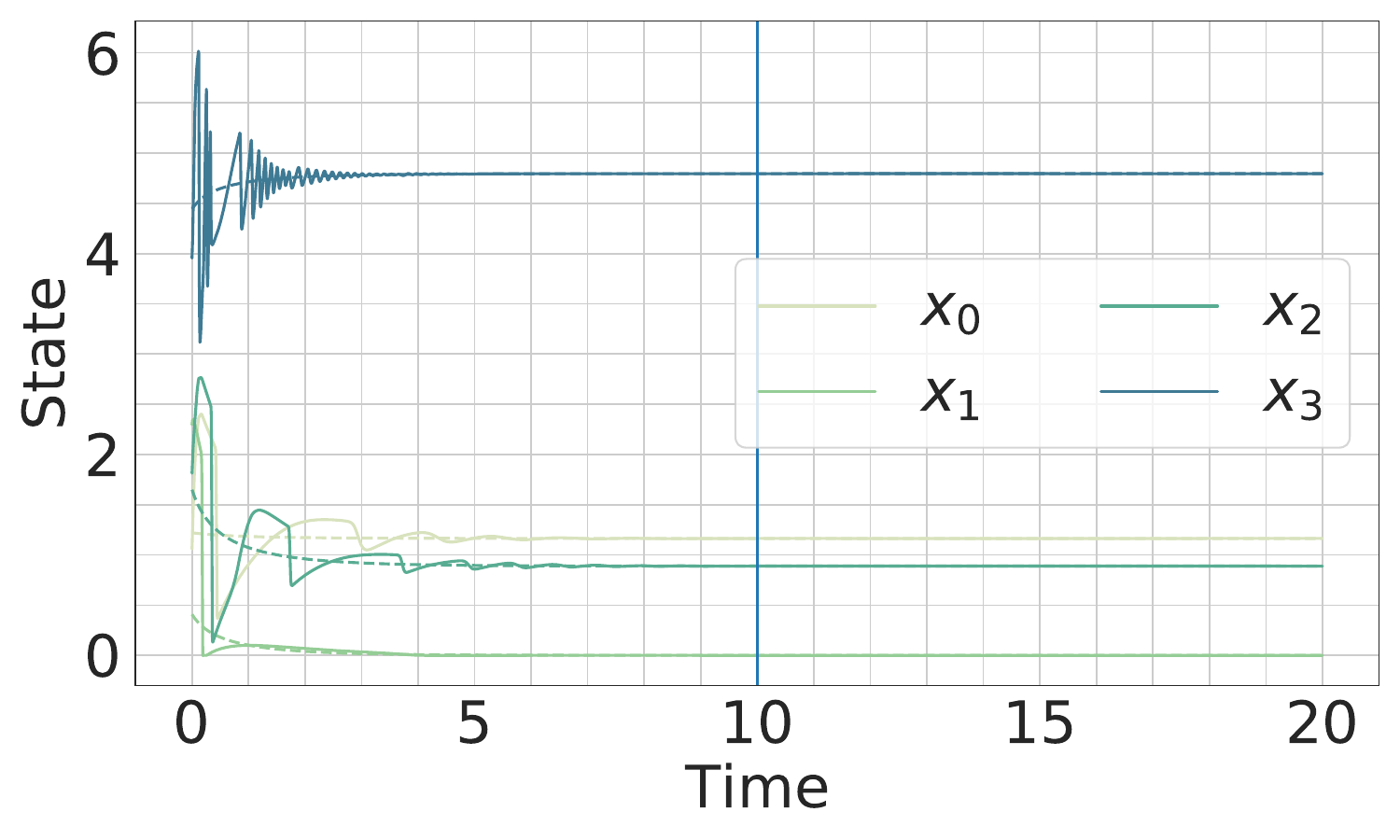}
         \put(5, 60){\textbf{B}}
         \end{overpic}
         \\
         \begin{overpic}[width=.47\textwidth]{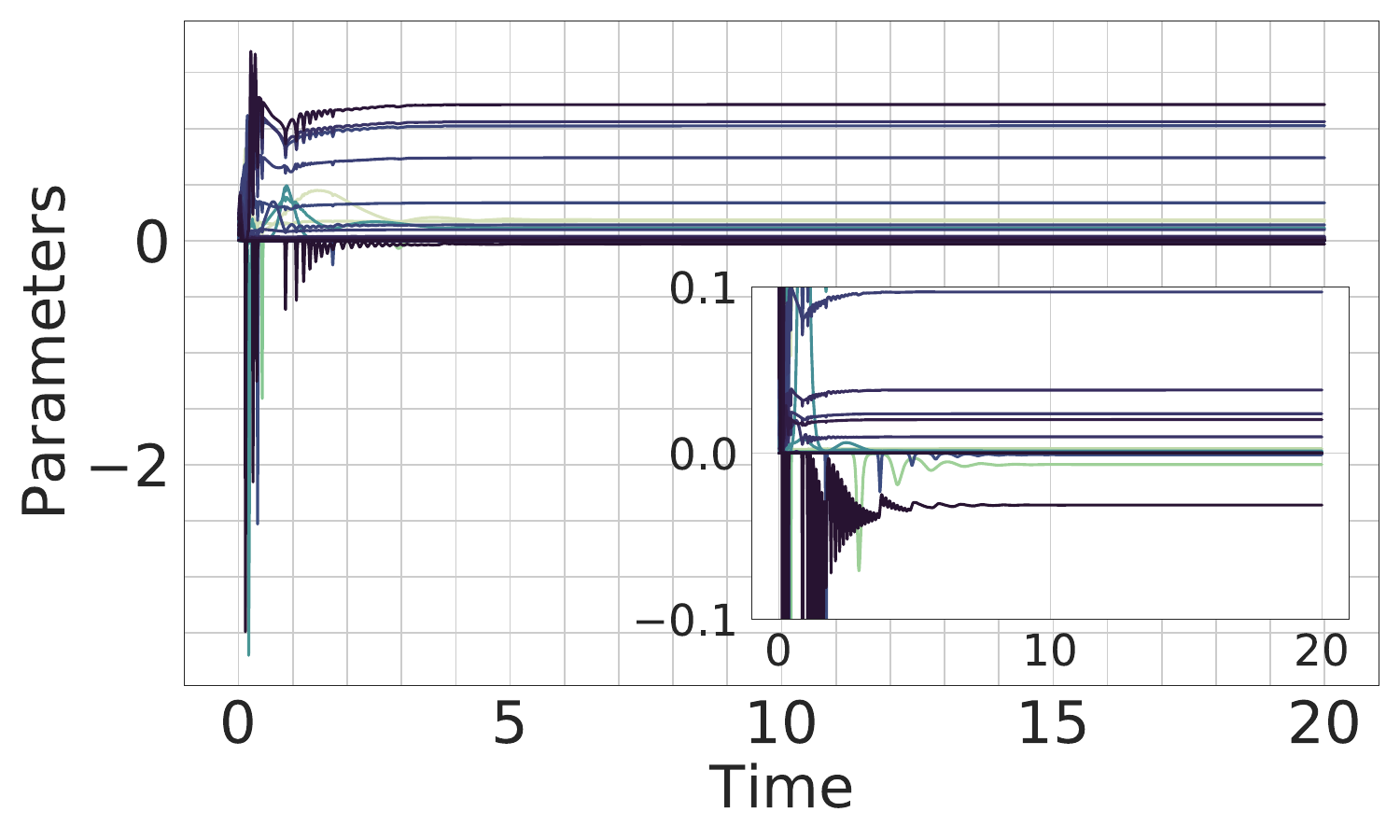}
         \put(5, 60){\textbf{C}}
         \end{overpic}&
         \begin{overpic}[width=.47\textwidth]{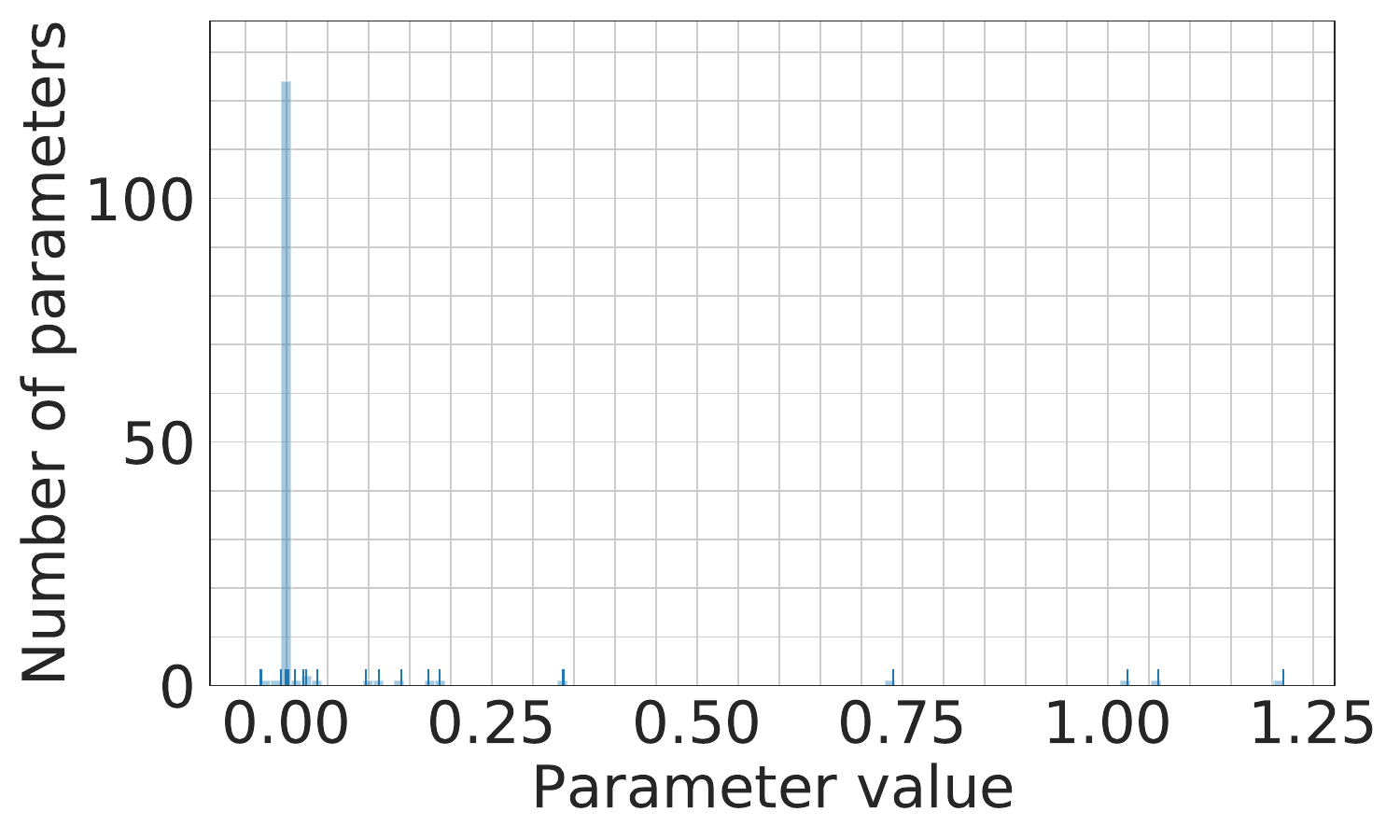}
         \put(5, 60){\textbf{D}}
         \end{overpic}
    \end{tabular}
    \caption{\textbf{Chemical reaction network.} (A) Observer error $\tilde{\bx} = \hat{\bx} - \bx$ for the adaptive dynamics predictor \eqref{eqn:chem_obs_1}~\&~\eqref{eqn:chem_obs_2} with adaptation (solid) and without adaptation (dashed). The predictor without adaptation diverges immediately. Past $t=10$, shrinkage is performed, all coefficients with magnitude below $10^{-3}$ are set to zero, and the predictor is run open-loop with $k = \gamma = 0$. (B) Convergence of $\hat{\bx}$ to $\bx$ for the adaptive dynamics predictor. The predictor accurately learns the fixed-point of the system and stays stationary when run open-loop. (C) Parameter trajectories for the adaptive dynamics predictor. Many parameters stay at or near zero, as predicted by Proposition~\ref{prop:implicit_reg}. (D) Histogram of final parameter values learned by the adaptive dynamics predictor. Only a few relevant terms are identified.}
    \label{fig:chemical}
\end{figure}
\section{Concluding remarks and future directions}
\label{sec:conc}
It is somewhat unusual in nonlinear control to have a choice between a large variety of algorithms that can all be proven to globally converge. Nevertheless, in this paper, we have presented a large set of new globally convergent adaptive control algorithms and analyzed their implicit bias properties. The general philosophy behind the development of these new algorithms is to leverage insight from recent developments in optimization and machine learning. As such, the results in this paper strengthen existing parallels between adaptive control, machine learning, and online optimization.

The new algorithms combine the velocity gradient methodology~\citep{fradkov-1979, fradkov-book} with the Bregman Lagrangian~\citep{ashia_1, ashia_2} to systematically generate velocity gradient algorithms with momentum. These algorithms can be seen as an adaptive control analogue of Nesterov's method, and admit an alternative interpretation as adding a form of Polyak-Ruppert averaging to classical first-order methods. Based on this latter analogy, using a parallel to distributed stochastic gradient descent algorithms~\citep{easgd, boffi-slot}, we developed a stable modification of our algorithms that consists of feedback coupling with an exponentially weighted average.

Throughout the paper, for simplicity of exposition, we gave concrete examples using the $n^{\text{th}}$ order system (\ref{eqn:gen_sys}). As discussed in Remark~\ref{rmk:error}, our results extend to more general systems that have an error model similar to (\ref{eqn:s_dyn}).

The most promising direction of future research is to understand when and whether or not there is an empirical advantage to using our proposed momentum algorithms. In optimization, accelerated algorithms generated by the Bregman Lagrangian provide faster convergence when properly discretized, and it is thus likely that a careful discretization is necessary to obtain optimal performance of our momentum algorithms. However, it is not possible to obtain finite-time convergence rates in adaptive control, and it would be necessary to prove such rates to understand analytically if there is an advantage. The analogy with Polyak-Ruppert averaging suggests that the momentum algorithms and the algorithms of Section~\ref{ssec:elastic} may be helpful in the presence of noise; here, the most likely angle of attack would be a stochastic discrete-time analysis.

Similar higher-order algorithms have appeared in the literature for adaptive control of linear systems of relative degree greater than one~\citep{morse}, where first-order adaptation laws provably cannot control the system. Here we have focused on feedback linearizable systems, and perhaps there exist classes of nonlinear systems that cannot be adaptively controlled with a first-order algorithm, but can with a momentum algorithm. We leave the investigation of these interesting and important questions to future work.

\bibliography{refs}
\appendix
\section{Simulation details}
\subsection{Further simulation details for Section~\ref{ssec:sim_nlin}}
\label{app2:sim_nlin}
The PI form of the non-Euclidean variant of algorithm (\ref{eqn:ho_tyukin}) is given by
\begin{align*}
    \nabla\psi(\hv) &= \overline{v} + \xi(\bx, t) + \brho(\bx, t),\\
    \xi(\bx, t) &= -\gamma s(\bx, \bx_d(t))\tanh(\mathbf{V}\bx),\\
    \rho(\bx, t) &= \gamma\left[\tanh\left(\mathbf{V}\bx\right)x_2 - \log\left(\cosh\left(\mathbf{V}\bx\right)\right)\oslash\mathbf{V}_2 + \left(\lambda\tilde{x} - x_{2, d}(t)\right)\tanh\left(\mathbf{V}\bx\right)\right],\\
    \dot{\overline{v}} &= \gamma\left(\dot{x}_{2, d}(t) - \lambda\left(x_2 - x_{2, d}(t)\right) - \eta s\right)\tanh(\mathbf{V}\bx) + \gamma x_2\tanh(\mathbf{V}\bx)\circ\mathbf{V}_1\oslash\mathbf{V}_2,\\
    \dot{\ha} &= \beta\left(1 + \mu\Vert\tanh(\mathbf{V}\bx)\Vert^2\right)\left(\hv - \ha\right),
\end{align*}
where $\circ$ and $\oslash$ denote elementwise multiplication and division respectively, where $\mathbf{V}_i$ is the $i^{\text{th}}$ column of $\mathbf{V}$, and where $\hv$ is obtained from $\nabla\psi(\hv)$ by inverting $\nabla\psi$. For the squared $p$ norm $\psi(\cdot) = \frac{1}{2}\Vert\cdot\Vert_p^2$, the inverse function can be analytically computed as
\begin{equation}
    \label{eqn:invert_sq_p}
    \left(\nabla\psi^{-1}\right)(\by) = \Vert\by\Vert_q^{2-q}|\by|^{q-1}\sign(\by)
\end{equation}
where $\frac{1}{q} + \frac{1}{p} = 1$, $|\cdot|$ denotes elementwise absolute value and $\sign(\cdot)$ denotes elementwise $\sign$~\citep{p-norm}. We consider the $\ell_1, \ell_2, \ell_4, \ell_6$, and $\ell_{10}$ norms for $\psi$. To approximate the $\ell_1$ norm, (\ref{eqn:invert_sq_p}) is used with $p = 1.1$. All other $p$ norms can be used directly. 

In all simulations we take $\lambda = .5$ in the definition of $s$ (\ref{eqn:s_dyn_def}) and $\eta = .5$ in the control input (\ref{eqn:control}). For the adaptation hyperparameters, we choose $\gamma = 1.5$ for the $\ell_2, \ell_4$, and $\ell_6$ norms. We take $\gamma = 50$ for the $\ell_1$ norm and $\gamma = .5$ for the $\ell_{10}$ norm\footnote{These values of $\gamma$ were chosen to ensure good control performance without excessively high control inputs or fast parameter adaptation. In particular, adaptation occurs very slowly with $\ell_1$ regularization, as small parameters are quickly eliminated to promote sparsity. A high adaptation gain was needed to ensure adaptation on a similar timescale to the other norms.}. In all cases, $\beta = 1$ and $\mu = \frac{3\gamma}{2\eta\beta}$. We set $\dim(\ba) = \dim(\ha) = 500$ and randomly initialize $\ha$ and $\hv$ around zero from a normal distribution with standard deviation $10^{-3}$. The true parameter vector $\ba$ is drawn from a normal distribution with mean zero and standard deviation $7.5$. The matrix $\bV$ is set to have normally distributed elements with standard deviation $\frac{1}{\sqrt{\dim{\ha}}}$. The state vector is initialized such that $\bx(0) = \bx_d(0)$. The desired trajectory is taken to be
\begin{equation*}
    x_d(t) = \sin\left(\frac{\sqrt{2}\pi}{12}t + \cos\left(\frac{\sqrt{3}\pi}{12}t\right)\right).
\end{equation*}
\subsection{Further simulation details for Section~\ref{ssec:sim_entropy}}
\label{app2:sim_entropy}
To define the control primitives, we consider a distribution over tasks specified by random desired trajectories
\begin{equation*}
    x_d^{(i)}(t) = M\sin(A_i t + B_i \cos(C_i t)) + D_i
\end{equation*}
with $D_i = 2i(-1)^i\times M$, $A_i \sim \mathsf{Unif}(0, 5\pi)$, $B_i \sim \mathsf{Unif}(0, 3)$, and $C_i \sim \mathsf{Unif}(0, 5\pi)$. The shift $D_i$ ensures that the desired trajectories occupy non-overlapping regions of state space. We then learn primitives $\left\{u_i(\bx, \ha^{(i)})\right\}_{i=1}^N$ to track $\left\{x_d^{(i)}(t)\right\}_{i=1}^N$ where each $u_i(\bx, \ha^{(i)})$ is given by \eqref{eqn:control} with parameter estimates $\ha^{(i)}$. The parameter estimates are found via the Slotine and Li adaptation law
\begin{align*}
    \dot{\ha}^{(i)} &= -\gamma \tanh(\bV\bx)^\T s(\bx, \bx_d(t)),
\end{align*}
which is allowed to run until the parameter estimates converge. We set $p = 15$, $N = 300$, $M = 0.1$, $\gamma = 5$, and $\eta = \lambda = 0.5$. Each vector of parameter estimates $\ha^{(i)}$ is initialized randomly, $\ha^{(i)}(0) \sim \mathcal{N}\left(0, \sigma^2_{\ha}\right)$ with $\sigma_{\ha} = 10^{-3}$. The state is initialized randomly for each task, $\bx_0 \sim \mathcal{N}\left(0, \sigma_\bx^2\right)$ with $\sigma_\bx = 5$. The true parameters $\ba$ are drawn randomly, $\ba \sim \mathcal{N}\left(0, \sigma_\ba^2\right)$ with $\sigma_\ba = 2$.

We then fix a time horizon $T$ and a number of tasks $k$, and set the desired trajectory in the experiment to be
\begin{equation*}
    x_d(t) = x_d^{(i_l)}(t) \text{    if    } t_{l-1} \leq t < t_l,
\end{equation*}
with $l = 1, \hdots, k$, $i_l$ drawn uniformly from $i=1, \hdots, N$, $t_0 = 0$, and $t_l = \frac{lT}{k}$. The above desired trajectory switches between desired trajectories $x_d^{(i_l)}$ that $u_{i_l}(\bx, \ha^{(i_l)})$ has learned to track. In the experiment, we fix $T = 1000$ and set $k = 5$.

To perform the experiment, we use the \texttt{lsoda} integrator in \texttt{scipy.integrate.ode}. For the Euclidean adaptation law, we zero any components if they become negative and project onto the simplex. For the non-Euclidean adaptation law, we integrate the mirror descent-like dynamics directly, and update $\nabla\psi\left(\hat{\boldsymbol\beta}\right)$ via $\texttt{lsoda}$. After each timestep, we compute $\hat{\boldsymbol\beta}$ by inverting the gradient of $\psi$, which ensures that each component $\hat{\beta}_i \geq 0$. We then project onto the simplex. In both cases, the projection is performed by dividing by the $\ell_1$ norm of the parameter vector, $\hat{\boldsymbol{\beta}} \mapsto \hat{\boldsymbol{\beta}}/\norm{\hat{\boldsymbol{\beta}}}_1$.
\section{Further simulation details for Section~\ref{ssec:exp_ham}}
\label{app2:exp_ham}
We define $\bY(\bq, \bp)\in\mathbb{R}^{1\times 21}$ to be a row vector of basis functions consisting of quadratics and quartics in $\bp_i$ and $\bq_i$, as well as $1/r_{ij}$, $1/r_{ij}^2$, and $1/r_{ij}^3$ potentials with $r_{ij} = \Vert\bq_i - \bq_j\Vert$ for $i \neq j$, comprising $21$ total basis functions. These choices represent standard expressions for kinetic energy, spring potentials, central force potentials, and higher-order terms; any basis functions can be chosen motivated by knowledge of the physical system at hand. 

We set $k = 5$, $\gamma = 3.5$, and choose $\psi(\cdot) = \frac{1}{2}\Vert\cdot\Vert_{1.05}^2$ to identify basis functions relevant to the observed trajectory. We fix $m_i = 1$ for all $i$ and initialize $\bq$ and $\bp$ to lock the system in an oscillatory mode. Past $t = 10$, we set $k = \gamma = 0$ and run the predictor open-loop, as well as perform shrinkage and set all coefficients with magnitude below $10^{-3}$ formally equal to zero, leaving $13$ remaining terms.
\section{Background results}
Barbalat's Lemma is a classical technique in adaptive control theory, which is used in conjuction with a Lyapunov-like analysis to prove convergence of a given signal.
\begin{lem}[Barbalat's Lemma,~\citep{slot_li_book}]
\label{lem:barbalat}
Assume that $\lim_{t\rightarrow \infty}\int_0^{t}|x(\tau)|d\tau < \infty$. If $x(t)$ is uniformly continuous, then $\lim_{t\rightarrow \infty}x(t) = 0$.
\end{lem}
A sufficient condition for uniform continuity of $x(t)$ is for $\dot{x}(t)$ to be bounded. It follows that for any signal $x(t) \in \LL_2\cap\LL_{\infty}$ with $\dot{x}(t) \in \LL_{\infty}$, we can apply Lemma~\ref{lem:barbalat} to the signal $x^2(t)$ and conclude that $x(t)\rightarrow 0$.
\begin{lem}
\label{lem:conv}
Assume that $\int_0^t\tilde{f}^2(\bx(t'), \hat{\ba}(t'), \ba, t')dt' < \infty$ where $[0, T)$ is the maximal interval of existence of $\bx(t)$. Further assume that $\hat{\ba}(t)$ is bounded over $[0, T)$, that both bounds are independent of $T$, and that $\tilde{f}(\bx, \ha, \ba, t)$ is locally bounded in $\bx$ and $\hat{\ba}$ uniformly in $t$. Then $\ha(\cdot)\in\LL_\infty$, $\tilde{f}(\bx(\cdot), \ha(\cdot), \ba, \cdot) \in \LL_2$, $s(\bx(\cdot), \bx_d(\cdot))\in\LL_2\cap\LL_\infty$, $s(\bx(t), \bx_d(t))\rightarrow 0$ and $\bx(t) \rightarrow \bx_d(t)$.
\end{lem}
\begin{proof}
By (\ref{eqn:s_dyn}), we can write explicitly
\begin{equation}
    \label{eqn:s_sol}
    s(\bx(t), \bx_d(t)) = \int_0^t e^{-\eta (t-\tau)}\tilde{f}(\bx(\tau), \hat{\ba}(\tau), \ba, \tau)d\tau.
\end{equation}
By the Cauchy-Schwarz inequality,
\begin{align*}
    s^2(\bx(T), \bx_d(T)) &\leq \left(\int_0^te^{-2\eta(T-\tau)}d\tau\right)\left(\int_0^t\tilde{f}^2(\bx(\tau), \ha(\tau), \ba, \tau)d\tau\right)\\
    &\leq \frac{1}{2\eta}\left(\int_0^t\tilde{f}^2(\bx(\tau), \ha(\tau), \ba, \tau)d\tau\right)\left(1 - e^{-2\eta T}\right)\\
    &\leq \frac{1}{2\eta}\left(\int_0^t\tilde{f}^2(\bx(\tau), \ha(\tau), \ba, \tau)d\tau\right) < \infty
\end{align*}
so that $\sup_{t\in[0, T)}|s(\bx(t), \bx_d(t))| < \infty$. Observe that this bound is independent of $T$. It immediately follows that $\sup_{t\in[0, T)}\Vert\bx(t)\Vert < \infty$, and that this bound is independent of $T$. This observation contradicts that $[0, T)$ is the maximal interval of existence of $\bx(t)$ for any $T$, and thus $\bx(t)$ must exist for all $t$. This shows that $\bx(\cdot)\in\LL_\infty$, $s(\bx(\cdot), \bx_d(\cdot))\in\LL_\infty$, and that the bounds on $\tilde{f}(\bx(t), \ha(t), \ba, t)$ and $\ha(t)$ for $t \in [0, T)$ can be extended for all $t$. From this we conclude $\tilde{f}(\bx(\cdot), \ha(\cdot), \ba, \cdot) \in \LL_2$ and $\ha\in\LL_\infty$. Similarly, Parseval's theorem applied to the low-pass filter (\ref{eqn:s_sol}) shows that $s(\bx(\cdot), \bx_d(\cdot)) \in \LL_2$. Because $\bx(\cdot) \in \LL_\infty$ and $\hat{\ba}(\cdot) \in \LL_\infty$, and because $\tilde{f}(\bx, \ha, \ba, t)$ is locally bounded in $\bx$ and $\hat{\ba}$ uniformly in $t$, $\tilde{f}(\bx(\cdot), \ha(\cdot), \ba, \cdot) \in \LL_\infty$. By (\ref{eqn:s_dyn}), $\dot{s}(\bx(\cdot), \bx_d(\cdot), \dot{\bx}(\cdot), \dot{\bx}_d(\cdot)) \in \LL_\infty$, and hence by Lemma~\ref{lem:barbalat}, $s(\bx(t), \bx_d(t)) \rightarrow 0$. By definition of $s$, we then conclude that $\bx(t) \rightarrow \bx_d(t)$.
\end{proof}
\section{Generalized linear models}
\label{app2:glms}
Generalized linear model (GLM) regression is an extension of linear regression where the data is assumed to be generated by a function of the form $f(\bx) = u(\bw^\T\bx)$ for a known ``link function'' $u$ and unknown parameters $\bw$. The first computationally and statistically efficient algorithm for this problem -- the GLM-Tron of~\citet{glmtron} -- assumes that $u$ is Lipschitz and monotonic, much like Assumption \ref{assmp:tyukin}. The GLM-Tron algorithm was recently extended to the setting of kernel methods, and was subsequently used to provably learn two hidden layer neural networks by~\citet{alphatron}; this extension is known as the Alphatron. In the kernel GLM setting handled by the Alphatron, the function to be approximated is assumed to be of the form $f(\bx) = u\left(\sum_{i=1}^m \mathcal{K}\left(\bx, \bx_i\right)w_i\right)$ where $\mathcal{K}$ is the kernel function for a Reproducing Kernel Hilbert Space (RKHS) $\mathcal{H}$.

The Alphatron initializes all weights to zero, and given a batch of labeled training data $\left(\bx_i, f(\bx_i)\right)_{i=1}^m$, updates with a learning rate $\lambda > 0$ according to the iteration
\begin{equation}
    \hat{w}_i^{t+1} = \hat{w}_i^t - \frac{\lambda}{m}\left(\hat{f}(\hat{\bw}^t, \bx_i) - f(\bx_i)\right).
    \label{eqn:alphatron_weights}
\end{equation}
We now demonstrate an equivalence between Tyukin's adaptation law (\ref{eqn:tyukin_alg}) and the Alphatron weight update (\ref{eqn:alphatron_weights}) in the following lemma. 

\begin{lem}
\label{lemma:alpha_tyuk}
The adaptation law (\ref{eqn:tyukin_alg}) is an application of the Alphatron algorithm (\ref{eqn:alphatron_weights}) to adaptive control.
\end{lem}
The proof is given in Appendix~\ref{app2:alpha_tyuk}.

Lemma \ref{lemma:alpha_tyuk} shows a convergence of techniques in nonlinearly parameterized adaptive control and provably correct learning algorithms with associated nonconvex optimization problems. It suggests that models amenable to gradient-based optimization in machine learning are good targets for nonlinear parameterizations suitable for adaptive control.
\section{Exponential forgetting least squares and bounded gain forgetting}
\label{ssec:bgf}
We now apply the exponential forgetting and bounded gain forgetting least squares techniques (cf. Chapter 9 of~\citet{slot_li_book}) to the algorithms developed in this work. These techniques are useful for estimation of time-varying parameters, as they rapidly discard previous information used for parameter estimation. Exponential forgetting least squares is described by a time-dependent learning rate matrix $\bP(t)$, which, in the linearly parameterized case $\tilde{f}(\bx, \ha, \ba, t) = \bY(\bx, t)\tilde{\ba}$ takes the form
\begin{equation}
    \label{eqn:efls}
    \dot{\bP} = \begin{cases}
    \lambda \bP - \bP \bY(\bx(t),t)^\T \bY(\bx(t), t) \bP & \text{if\ \ } \Vert\bP\Vert \leq P_0\\
    0 & \text{else}
    \end{cases}
\end{equation}
where $\lambda > 0$ is a constant forgetting factor, $P_0$ is a maximum bound on the norm, and $\Vert\bP\Vert$ is a matrix norm such as the opeator norm. Equation (\ref{eqn:efls}) implies for the inverse matrix
\begin{equation}
    \label{eqn:efls_inv}
    \frac{d}{dt}\bP^{-1} = \begin{cases}
    -\lambda \bP^{-1} + \bY(\bx(t), t)^\T\bY(\bx(t), t)& \text{if\ \ } \Vert\bP\Vert \leq P_0\\
    0 & \text{else}
    \end{cases}
\end{equation}
In the nonlinearly parameterized case described by Assumption \ref{assmp:tyukin}, we will replace $\bY(\bx(t), t)^\T$ in (\ref{eqn:efls}) \& (\ref{eqn:efls_inv}) by $\balf(\bx(t), t)$. In the bounded gain forgetting technique, $\lambda$ is a time-dependent function 
\begin{equation}
    \label{eqn:lamb_bgf}
    \lambda(t) = \lambda_0\left(1 - \frac{\Vert\bP\Vert}{P_0}\right),
\end{equation}
where $\lambda_0 > 0$ sets the forgetting factor when the norm of $\bP$ is small. It can be shown that this choice of $\lambda(t)$ ensures that $\Vert\bP\Vert \leq P_0$, and thus we may drop the case statement in (\ref{eqn:efls}) \& (\ref{eqn:efls_inv})~\citep{slot_li_book}. The choice of $\lambda(t)$ in bounded gain forgetting and the case statement used in (\ref{eqn:efls}) \& (\ref{eqn:efls_inv}) are both employed to prevent unboundedness of the learning rate matrix. 

We focus on algorithms without the elastic modification of Section~\ref{ssec:elastic}; extension to the elastic modification is simple. We also focus on the bounded gain forgetting technique: proofs for the exponential forgetting least squares technique are identical, with the addition of an appropriate case statement in the time derivative of the Lyapunov function. For simplicity, we include only the time-dependent gain $\bP(t)$ and set the scalar gains $\kappa = \gamma = 1$ where applicable.

We begin with the first-order non-filtered composite (\ref{eqn:comp_adapt}) with $\bP$ given by (\ref{eqn:efls}). In this case, the composite algorithm may be implemented via the PI form (\ref{eqn:pi_comp_1})-(\ref{eqn:pi_comp_4}) where now $\bP = \bP(t)$.
\begin{prop}
\label{prop:bgf_1}
Consider the adaptation algorithm (\ref{eqn:comp_adapt}) with $\bP(t)$ given by (\ref{eqn:efls}), $\lambda(t)$ given by (\ref{eqn:lamb_bgf}), and $\kappa = \gamma = 1$. Then all trajectories $(\bx(t), \hat{\ba}(t))$ remain bounded, $\tilde{f}(\bx(\cdot), \ha(\cdot), \cdot) \in \LL_2\cap\LL_\infty$, $s(\bx(\cdot), \bx_d(\cdot)) \in \LL_2\cap\LL_\infty$, $s(\bx(t), \bx_d(t)) \rightarrow 0$ and $\bx(t) \rightarrow \bx_d(t)$.
\end{prop}
The proof is given in Appendix~\ref{app2:bgf_1}

We can state a similar result for the higher-order non-filtered composite with time-dependent $\bP(t)$ given by (\ref{eqn:efls}),
\begin{equation}
    \ddot{\hat{\ba}} + \left(\beta \sN(t) - \frac{\dot{\sN}(t)}{\sN(t)} - \dot{\bP}(t)\bP^{-1}(t)\right)\dot{\hat{\ba}} = -\beta \sN(t) \bP(t)\bY(\bx, t)^\T\left(s(\bx, \bx_d(t)) + \tilde{f}(\bx, \ha, \ba, t)\right),
    \label{eqn:bgf_ho_comp}
\end{equation}
which admits a representation as two first-order equations,
\begin{align}
    \label{eqn:bgf_ho_comp_v}
    \dot{\hat{\bv}} &= -\bP(t)\bY(\bx, t)^\T\left(s(\bx, \bx_d(t)) + \tilde{f}(\bx, \ha, \ba, t)\right),\\
    \label{eqn:bgf_ho_comp_a}
    \dot{\hat{\ba}} &= \beta \sN(t) \bP(t)\left(\hat{\bv} - \hat{\ba}\right).
\end{align}
Equation (\ref{eqn:bgf_ho_comp_v}) can be implemented via the PI form $\hat{\bv} = \overline{\bv} + \bxi(\bx, t) + \brho(\bx, t)$ where $\bxi, \brho$, and $\dot{\overline{\bv}}$ are given by (\ref{eqn:pi_comp_1}), (\ref{eqn:pi_comp_2}), and (\ref{eqn:pi_comp_4}) respectively with $\gamma = \kappa = 1$.
\begin{prop}
\label{prop:bgf_2}
Consider the adaptation algorithm (\ref{eqn:bgf_ho_comp}) with $\bP(t)$ given by (\ref{eqn:efls}), $\lambda(t)$ given by (\ref{eqn:lamb_bgf}), $\sN(t) = 1 + \mu\Vert\bY(\bx(t), t)\Vert^2$ and $\mu > \frac{3\eta + 2}{2\eta\beta}$. Then all trajectories $(\bx(\cdot), \hat{\bv}(\cdot), \hat{\ba}(\cdot))$ remain bounded, $\tilde{f}(\bx(\cdot), \ha(\cdot), \cdot) \in \LL_2\cap\LL_\infty$, $s(\bx(\cdot), \bx_d(\cdot)) \in \LL_2\cap\LL_\infty$, $s(\bx(t), \bx_d(t)) \rightarrow 0$ and $\bx(t) \rightarrow \bx_d(t)$.
\end{prop}
The proof is given in Appendix~\ref{app2:bgf_2}. 

\begin{rmk}
\label{rmk:P_bound}
Because $\bP(t)$ is uniformly bounded in $t$, it is not necessary to include $\bP(t)$ in (\ref{eqn:bgf_ho_comp_a}); by a slight modification of the proof, it is easy to show that the modified higher-order law
\begin{equation*}
    \ddot{\hat{\ba}} + \left(\beta \sN(t) - \frac{\dot{\sN}(t)}{\sN(t)}\right)\dot{\hat{\ba}} = -\beta \sN(t) \bP(t)\bY(\bx(t), t)^\T\left(s(\bx, \bx_d(t)) + \tilde{f}(\bx, \ha, \ba, t)\right),
\end{equation*}
is also a stable adaptive law with a suitable choice of gains.$\qeddef$
\end{rmk}

We now consider Tyukin's first-order algorithm for nonlinearly parameterized systems (\ref{eqn:tyukin_alg}) with $\bP = \bP(t)$ given by (\ref{eqn:efls}). To do so, we require an additional assumption
\begin{assumption}
\label{assmp:lower_bound}
For the same function $\balf(\bx, t)$ as in Assumption \ref{assmp:tyukin}, there exists a constant $D_2$ such that
\begin{equation}
    \label{eqn:lower_bound}
    |\tilde{f}(\bx, \hat{\ba}, \ba, t)| \geq D_2|\balf(\bx, t)^\T\tilde{\ba}|.
\end{equation}
\end{assumption}

Together with Assumption \ref{assmp:tyukin}, Assumption~\ref{assmp:lower_bound} states that $\tilde{f}(\bx, \ha, \ba, t)$ lies between two linear functions. Given that the update (\ref{eqn:efls}) is derived based on recursive \textit{linear} least squares considerations, it is unsurprising that such an assumption is required in the nonlinearly parameterized setting. We are now in a position to state the following proposition.
\begin{prop}
\label{prop:bgf_3}
Consider the adaptation algorithm (\ref{eqn:tyukin_alg}) with $\bP(t)$ given by (\ref{eqn:efls}) and $\lambda(t)$ given by (\ref{eqn:lamb_bgf}) for $\tilde{f}(\bx, \ha, \ba, t)$ satisfying Assumptions \ref{assmp:tyukin} and \ref{assmp:lower_bound}. Further assume that $D_1 < 2D_2^2$ or that $D_2 > \frac{1}{2}$. Then, all trajectories $(\bx(t), \hat{\ba}(t))$ remain bounded, $\tilde{f}(\bx(\cdot), \ha(\cdot), \cdot) \in \LL_2$, $s(\bx(\cdot), \bx_d(\cdot)) \in \LL_2\cap\LL_\infty$, $s(\bx(t), \bx_d(t)) \rightarrow 0$ and $\bx(t) \rightarrow \bx_d(t)$.
\end{prop}
The proof is given in Appendix~\ref{app2:bgf_3}.

Last, we consider the momentum algorithm for nonlinearly parameterized systems
\begin{equation}
    \ddot{\hat{\ba}} + \left(\beta \sN(t) - \frac{\dot{\sN}(t)}{\sN(t)}- \dot{\bP}(t)\bP(t)^{-1}\right)\dot{\hat{\ba}} = -\beta \sN(t) \tilde{f}(\bx, \ha, \ba, t)\bP(t)\balf(\bx, t),
    \label{eqn:bgf_ho_nlin}
\end{equation}
which admits a representation as two first-order equations,
\begin{align}
    \label{eqn:bgf_ho_nlin_v}
    \dot{\hat{\bv}} &= -\tilde{f}(\bx, \ha, \ba, t)\bP(t)\balf(\bx, t),\\
    \label{eqn:bgf_ho_nlin_a}
    \dot{\hat{\ba}} &= \beta \sN(t) \bP(t)\left(\hat{\bv} - \hat{\ba}\right).
\end{align}
\begin{prop}
\label{prop:bgf_4}
Consider the adaptation algorithm (\ref{eqn:bgf_ho_nlin}) with $\bP(t)$ given by (\ref{eqn:efls}), $\lambda(t)$ given by (\ref{eqn:lamb_bgf}), $\sN(t) = 1 + \mu\Vert\balf(\bx(t), t)\Vert^2$, and $\mu > \frac{4D_2 - 2 + \left(2D_1 + 1\right)^2}{\beta\left(4D_2 - 1\right)}$. Suppose $\tilde{f}(\bx, \ha, \ba, t)$ satisfies Assumptions \ref{assmp:tyukin} and \ref{assmp:lower_bound}. Further assume that $D_2 > \frac{1}{2}$. Then, all trajectories $(\bx(t), \hat{\bv}(t), \hat{\ba}(t))$ remain bounded, $\tilde{f}(\bx(\cdot), \ha(\cdot), \cdot) \in \LL_2$, $s(\bx(\cdot), \bx_d(\cdot)) \in \LL_2\cap\LL_\infty$, $s(\bx(t), \bx_d(t))\rightarrow 0$ and $\bx(t) \rightarrow \bx_d(t)$.
\end{prop}
The proof is given in Appendix~\ref{app2:bgf_4}. 

\begin{rmk}
As in Remark~\ref{rmk:P_bound}, because $\bP(t)$ is uniformly bounded in $t$, it is not necessary to include $\bP(t)$ in (\ref{eqn:bgf_ho_nlin_a}). It is simple to show by modification of the proof that
\begin{equation*}
    \ddot{\hat{\ba}} + \left(\beta \sN(t) - \frac{\dot{\sN}(t)}{\sN(t)}\right)\dot{\hat{\ba}} = - \beta \sN(t) \tilde{f}(\bx, \ha, \ba, t)\bP(t)\balf(\bx, t)
\end{equation*}
is also a stable adaptive law with a suitable choice of gains.
\end{rmk}
\section{Proofs}
\subsection{Proof of Proposition~\ref{prop:implicit_reg}}
\label{app2:prop:implicit_reg}
\begin{proof}
Let $\bthet$ be any constant vector of parameters. The Bregman divergence $\bregd{\bthet}{\ha} = \psi(\bthet) - \psi(\ha) - \nabla\psi(\ha)^\T\left(\bthet - \ha\right)$ has time derivative
\begin{equation*}
    \frac{d}{dt}\bregd{\bthet}{\ha} = -\left(\frac{d}{dt}\nabla\psi(\ha)\right)^\T\left(\bthet - \ha\right).
\end{equation*}
From (\ref{eqn:natural_sl}), $\frac{d}{dt}\nabla\psi(\ha) = -\nabla_{\ha}\dot{Q}(\bx, \ha, t)$, so that
\begin{equation*}
    \frac{d}{dt}\bregd{\bthet}{\ha} = \nabla_{\ha}\dot{Q}(\bx, \ha, t)^\T\left(\bthet - \ha\right).
\end{equation*}
Integrating both sides of the above shows that
\begin{equation*}
    \bregd{\bthet}{\ha(0)} = \bregd{\bthet}{\ha(t)} + \int_0^t \nabla_{\ha}\dot{Q}(\bx(\tau), \ha(\tau), \tau)^\T\left(\ha(\tau) - \bthet\right)d\tau.
\end{equation*}
By the assumption of a linear parameterization, $\nabla_{\ha}\dot{Q}(\bx, \ha, t) = \bY(\bx, t)^\T\frac{\partial Q}{\partial \bx}(\bx, t)$. If we now take $\bthet \in \mathcal{A}$, $\bY(\bx(\tau), \tau)\bthet = f(\bx(\tau), \ba, \tau)$ and the integral term becomes independent of $\bthet$. Assuming that $\ha(t)\rightarrow\ha_\infty \in \mathcal{A}$, we can take the limit as $t\rightarrow \infty$ and say that for any $\bthet \in \mathcal{A}, \ha_\infty \in \mathcal{A}$,
\begin{equation*}
    \bregd{\bthet}{\ha(0)} = \bregd{\bthet}{\ha_\infty} + \int_0^\infty \frac{\partial Q}{\partial \bx}(\bx(\tau), \tau)^\T\left(\bY(\bx(\tau), \tau)\ha(\tau) - f(\bx(\tau), \ba, \tau)\right)d\tau.
\end{equation*}
Because the only dependence of the right-hand side on $\bthet$ is in the first term, and because this relation holds for any $\bthet$, the $\arg\min$ of the two Bregman divergences must be identical. The minimum of the right-hand side over $\bthet$ is clearly obtained at $\ha_\infty$, while the minimum of the left-hand side is by definition obtained at $\arg\min_{\bthet\in\mathcal{A}}d_\psi(\bthet, \ha(0))$. From this, we conclude that
\begin{equation*}
    \ha_\infty = \arg\min_{\bthet\in\mathcal{A}}\bregd{\bthet}{\ha(0)},
\end{equation*}
which completes the proof.
\end{proof}
\subsection{Proof of Proposition~\ref{prop:tyukin_reg}}
\label{app2:tyukin_reg}
\begin{proof}
The proof is much the same as Proposition~\ref{prop:implicit_reg}. The Bregman divergence $\bregd{\bthet}{\ha}$ for any fixed vector of parameters $\bthet$ verifies
\begin{equation*}
    \frac{d}{dt}\bregd{\bthet}{\ha} = \tilde{f}(\bx(t), \ha(t), \ba, t)\balf(\bx(t), t)^\T\left(\bthet - \ha\right),
\end{equation*}
so that, integrating both sides,
\begin{equation*}
    \bregd{\bthet}{\ha(0)} = \bregd{\bthet}{\ha(t)} - \int_0^t\tilde{f}(\bx(\tau), \ha(t), \ba, \tau)\balf(\bx(\tau), \tau)^\T\left(\bthet - \ha(\tau)\right)d\tau.
\end{equation*}
Now take $\bthet \in \mathcal{A}$. By the assumptions of the proposition, $\balf(\bx(\tau), \tau)^\T\bthet = \balf(\bx(\tau), \tau)^\T\ba$ is independent of $\bthet$. Hence, using that $\ha(t)\rightarrow\ha_\infty \in \mathcal{A}$, we can write
\begin{equation*}
    \bregd{\bthet}{\ha(0)} = \bregd{\bthet}{\ha_\infty} - \int_0^\infty \tilde{f}(\bx(\tau), \ha(\tau), \ba, \tau)\balf(\bx(\tau), \tau)^\T\left(\ba - \ha(\tau)\right)d\tau.
\end{equation*}
Optimizing both sides over $\bthet \in \mathcal{A}$ as in Proposition~\ref{prop:implicit_reg} yields the result.
\end{proof}
\subsection{Proof of Proposition~\ref{prop:dyn_predict_gen}}
\label{app2:dyn_predict_gen}
\begin{proof}
By adding and subtracting $\bff(\hat{\bx}) = \bY(\hat{\bx})\ba$, the error $\mathbf{e} = \hat{\bx} - \bx$ has dynamics
\begin{equation*}
    \dot{\mathbf{e}} = -k\mathbf{e} + \bY(\hat{\bx})\tilde{\ba} + \bff(\hat{\bx}) - \bff(\bx).
\end{equation*}
Consider the parameter estimator
\begin{equation*}
    \dot{\ha} = -\gamma\left[\nabla^2\psi(\ha)\right]^{-1}\bY(\hat{\bx})^\T\bGam\mathbf{e},
\end{equation*}
where $\gamma > 0$ is a constant learning rate, $\psi:\mathbb{R}^p\rightarrow\mathbb{R}$ is a strongly convex potential function, and $\bGam\in\mathbb{R}^{n\times n} > 0$ is a constant symmetric positive definite matrix. Now consider the Lyapunov-like function
\begin{equation*}
    V(\be, \ha) = \frac{1}{2}\be^\T\bGam\be + \frac{1}{\gamma}\bregd{\ba}{\ha}.
\end{equation*}
Let $\frac{\partial g}{\partial \bx}(\bx_1, \bx_2)$ denote the derivative of $\bg$ with respect to its first argument evaluated at $(\bx_1, \bx_2)$. $V(\be, \ha)$ then has time derivative
\begin{align}
    \dot{V}(\be, \ha, t) &= \be^\T\bGam\left(\bg\left(\be + \bx(t), \bx(t)\right) + \bY\left(\be + \bx(t)\right)\tilde{\ba} + \bff\left(\be + \bx(t)\right) - \bff\left(\bx(t)\right)\right)\nonumber\\
    &\phantom{=} - \tilde{\ba}^\T\bY^\T\left(\be + \bx(t)\right)\bGam\be,\nonumber\\
    &= \be^\T\bGam\left(\bg\left(\be + \bx(t), \bx(t)\right) - \bg\left(\bx(t), \bx(t)\right) + \bff\left(\be + \bx(t)\right) - \bff\left(\bx(t)\right)\right),\nonumber\\
    &= \be^\T\int_0^1\bGam\left[\frac{\p \bff}{\p \bx}\left(\bx(t) + s\be\right) + \frac{\partial \bg}{\p \bx}\left(\bx(t) + s\be, \bx(t)\right)\right]ds\be,
    \label{eqn:dyn_pred_lyap_dot}
\end{align}
where the second equality follows by application of $\bg(\bx, \bx) = \mathbf{0}$ for all $\bx$ and the last equality follows by Hadamard's Lemma. (\ref{eqn:dyn_pred_lyap_dot}) shows that $\be(t)\rightarrow 0$ as long as $\bff(\bx) + \bg(\bx, \bu(t))$ is contracting in the metric $\bGam$ for any external input $\bu(t)$~\citep{contraction, modular}, i.e., if
\begin{equation*}
    \left(\frac{\p \bff(\bx)}{\p \bx} + \frac{\p \bg}{\p \bx}\left(\bx, \bu(t)\right)\right)^\T\bGam + \bGam\left(\frac{\p \bff(\bx)}{\p \bx} + \frac{\p \bg}{\p \bx}\left(\bx, \bu(t)\right)\right) \leq -2\lambda\bGam
\end{equation*}
uniformly in $\bx$ for some contraction rate $\lambda > 0$. To see this, assume that $f(\bx) + \bg(\bx, \bu(t))$ is contracting for any external input $\bu(t)$. Then
\begin{equation*}
    \dot{V}(\be, \ha, t) \leq -2\lambda \be^\T \bGam \be \leq 0.
\end{equation*}
Because $V(\be, \ha) \geq 0$ for all $\be$ and $\ha$, and because $\dot{V}(\be, \ha, t) \leq 0$ for all $\be, \ha$, and $t$, this shows that $\be(\cdot) \in \mathcal{L}_\infty$, $\ha(\cdot) \in \mathcal{L}_\infty$ and that $\lim_{t\rightarrow\infty}V(\be(t), \ha(t)) < \infty$ exists, is finite, and is positive. Integrating both sides of \eqref{eqn:dyn_pred_lyap_dot} shows that
\begin{equation*}
    \int_0^\infty \be^\T \bGam \be \leq \frac{V(\be(0), \ha(0))}{2\lambda}
\end{equation*}
so that $\be(\cdot) \in \mathcal{L}_2$. It is straightforward to show that $\dot{\be}(\cdot) \in \mathcal{L}_\infty$, so that an application of Lemma~\ref{lem:barbalat} leads to the conclusion that $\be(t) \rightarrow \mathbf{0}$.
\end{proof}
\subsection{Proof of Proposition~\ref{prop:ham_estimate}}
\label{app2:ham_estimate}
After subtracting the true dynamics $\dot{\bp}$ and $\dot{\bq}$ from the dynamics predictor, consider the decomposition of the error dynamics
\begin{align*}
    \dot{\tilde{\bp}} &= -\left(\nabla_{\hq}\bY(\hp, \hq)\right)\tilde{\ba} - k_p \tilde{\bp} - \left(\nabla_{\hq}\HH(\hat{\bp}, \hat{\bq}) - \nabla_{\hq}\HH(\bp, \hat{\bq})\right) - \left(\nabla_{\hq}\HH(\bp, \hat{\bq}) - \nabla_{\bq}\HH(\bp, \bq)\right),\\
    \dot{\tilde{\bq}} &= \left(\nabla_{\hp}\bY(\hp, \hq)\right)\tilde{\ba} - k_q\tilde{\bq} + \left(\nabla_{\hp}\HH(\hp, \hq) - \nabla_{\hp}\HH(\hp, \bq)\right) + \left(\nabla_{\hp}\HH(\hp, \bq) - \nabla_{\bp}\HH(\bp, \bq)\right),
\end{align*}
The Lyapunov-like function
\begin{equation*}
    V(\tilde{\bp}, \tilde{\bq}, \ha) = \frac{1}{2}\tilde{\bp}^\T\tilde{\bp} + \frac{1}{2}\tilde{\bq}^\T\tilde{\bq} + \bregd{\ba}{\ha}
\end{equation*}
then has time derivative
\begin{align*}
    \dot{V}(\tilde{\bp}, \tilde{\bq}, \ha) &= \tilde{\bp}^\T\left[-\left(\nabla_{\hq}\bY(\hp, \hq)\right)\tilde{\ba} - k_p \tilde{\bp} - \left(\nabla_{\hq}\HH(\hat{\bp}, \hat{\bq}) - \nabla_{\hq}\HH(\bp, \hat{\bq})\right) - \left(\nabla_{\hq}\HH(\bp, \hat{\bq}) - \nabla_{\bq}\HH(\bp, \bq)\right)\right]\\
    &\phantom{=} + \tilde{\bq}^\T\left[\left(\nabla_{\hp}\bY(\hp, \hq)\right)\tilde{\ba} - k_q\tilde{\bq} + \left(\nabla_{\hp}\HH(\hp, \hq) - \nabla_{\hp}\HH(\hp, \bq)\right) + \left(\nabla_{\hp}\HH(\hp, \bq) - \nabla_{\bp}\HH(\bp, \bq)\right)\right]\\
    &\phantom{=} + \tilde{\ba}^\T\left(\left(\nabla_{\hq}\bY(\hp, \hq)\right)^\T\tilde{\bp} - \left(\nabla_{\hp}\bY(\hp, \hq)\right)^\T\tilde{\bq}\right)\\
    &= \tilde{\bp}^\T\left[- k_p \tilde{\bp} - \left(\nabla_{\hq}\HH(\hat{\bp}, \hat{\bq}) - \nabla_{\hq}\HH(\bp, \hat{\bq})\right) - \left(\nabla_{\hq}\HH(\bp, \hat{\bq}) - \nabla_{\bq}\HH(\bp, \bq)\right)\right]\\
    &\phantom{=} + \tilde{\bq}^\T\left[- k_q\tilde{\bq} + \left(\nabla_{\hp}\HH(\hp, \hq) - \nabla_{\hp}\HH(\hp, \bq)\right) + \left(\nabla_{\hp}\HH(\hp, \bq) - \nabla_{\bp}\HH(\bp, \bq)\right)\right]\\
    &= \begin{pmatrix} \tilde{\bp}^\T & \tilde{\bq}^\T\end{pmatrix}\bM(\tilde{\bp}, \tilde{\bq}) \begin{pmatrix} \tilde{\bp} \\ \tilde{\bq} \end{pmatrix},
\end{align*}
where the matrix $\bM(\tilde{\bp}, \tilde{\bq})$ is given by, after an application of Hadamard's Lemma,
\begin{equation*}
    \begin{pmatrix} -k_p\bI - \int_0^1\nabla_{\bp + s\tilde{\bp}}\nabla_{\hq}\HH(\bp + s\tilde{\bp}, \hq)ds & -\int_0^1\nabla_{\bq + s\tilde{\bq}}^2\HH(\bp, \bq + s\tilde{\bq})ds \\ \int_0^1\nabla_{\bp + s\tilde{\bp}}^2\HH(\bp + s\tilde{\bp}, \bq)ds & -k_q \bI + \int_0^1\nabla_{\bq + s\tilde{\bq}}\nabla_{\hp}\HH(\hp, \bq + s\tilde{\bq})ds \end{pmatrix}.
\end{equation*}
As in the proof of Proposition~\ref{prop:dyn_predict_gen}, a sufficient condition for convergence of $\tilde{\bp}(t)\rightarrow \mathbf{0}$ and $\tilde{\bq}(t) \rightarrow \mathbf{0}$ is uniform negative definiteness of the Jacobian matrix
\begin{equation*}
    \bJ = \begin{pmatrix} -k_p\bI - \nabla_{\bp}\nabla_{\bq}\HH(\bp, \bq) & -\nabla_{\bq}^2\HH(\bq, \bp) \\ \nabla_{\bp}^2\HH(\bp, \bq) & -k_q \bI + \nabla_{\bq}\nabla_{\bp}\HH(\bp, \bq) \end{pmatrix},
\end{equation*}
in $\bp$ and $\bq$, i.e., contraction of the nominal $\dot{\hp}$ and $\dot{\hq}$ system in the Euclidean metric. Sufficient conditions for uniform negative definiteness are then given by (see, for example, Example~3.8 in~\citet{modular})
\begin{align*}
    k_p &> -\frac{1}{2}\lambda_{\min}\left(\nabla_{\bp}\nabla_{\bq}\HH(\bp, \bq) + \nabla_{\bq}\nabla_{\bp}\HH(\bp, \bq)\right),\\
    k_q &> \frac{1}{2}\lambda_{\max}\left(\nabla_{\bp}\nabla_{\bq}\HH(\bp, \bq) + \nabla_{\bq}\nabla_{\bp}\HH(\bp, \bq)\right),\\
    \lambda_p\lambda_q &> \frac{1}{4}\lambda_{\max}^2\left[\nabla^2_{\bp}\HH(\bp, \bq) - \nabla^2_{\bq}\HH(\bp, \bq)\right],
\end{align*}
where $\lambda_p$ and $\lambda_q$ are the contraction rates of the $\hat{\bp}$ and $\hat{\bq}$ systems respectively, given by the difference of the left- and right-hand sides of the first two inequalities above.
\subsection{Proof of Proposition~\ref{prop:lagr_est}}
\label{app2:lagr_est}
The Euler-Lagrange equations of motion $\frac{d}{dt}\frac{\p \LL}{\p\dot{\bq}} - \frac{\p\LL}{\p\bq} = \bu$ with $\bu$ a control input give the dynamics
\begin{equation*}
    \sum_{lj}a_l^{(K)}M_{ij}^l(\bq)\ddot{q}_j + \sum_{lkj}a_l^{(K)}\dot{q}_k\dot{q}_j\left[\frac{\p M_{ij}^l(\bq)}{\p q_k} - \frac{1}{2}\frac{\p M_{kj}^l(\bq)}{\p q_i}\right] + \sum_l a_l^{(P)}\frac{\p \phi^l(\bq)}{\p q_i} = u_i.
\end{equation*}
Above, the second term
\begin{equation*}
    \sum_{lkj}a_l^{(K)}\dot{q}_k\dot{q}_j\left[\frac{\p M_{ij}^l(\bq)}{\p q_k} - \frac{1}{2}\frac{\p M_{kj}^l(\bq)}{\p q_i}\right]
\end{equation*}
uniquely defines the centripetal and Coriolis forces (traditionally written as $\bC(\bq, \dot{\bq})\dot{\bq}$ with $\bC(\bq, \dot{\bq})$ the Coriolis matrix), but does not uniquely define the Coriolis matrix (see, e.g., Section 9.1 of~\citet{slot_li_book}). Choosing
\begin{equation*}
    C_{ij}(\bq, \dot{\bq}) = \sum_{kl} a_l^{(K)}\frac{1}{2}\left[\frac{\p M^l_{ij}(\bq)}{\p q_k} - \left(\frac{\p M^l_{kj}(\bq)}{\p q_i} - \frac{\p M^l_{ki}(\bq)}{\p q_j}\right)\right]\dot{q}_k
\end{equation*}
preserves the Coriolis force $\bC(\bq, \dot{\bq})\dot{\bq}$ and ensures that $\dot{\bH}(\bq) - 2\bC(\bq, \dot{\bq})$ is a skew-symmetric matrix. In matrix notation, the dynamics are then given by
\begin{equation*}
    \bH(\bq)\ddot{\bq} + \bC(\bq, \dot{\bq})\dot{\bq} + \bg(\bq) = \bu
\end{equation*}
with the potential force $\bg(\bq) = \sum_l a^{(P)}_l\nabla_{\bq}\phi^l(\bq)$. Defining $\bs$ and $\dot{\bq}_r$ as $\bs(\bq, \dot{\bq}, \bq_d(t), \dot{\bq}_d(t)) = \left(\frac{d}{dt} + \lambda\right)\left(\bq - \bq_d(t)\right) = \dot{\bq} - \dot{\bq}_r$ with $\dot{\bq}_r = \dot{\bq}_d(t) - \lambda \left(\bq - \bq_d(t)\right)$, these dynamics can be equivalently rewritten, dropping the arguments of $\bs$ for visual clarity,
\begin{equation}
    \label{eqn:lag_dyn_s}
    \bH(\bq)\dot{\bs} + \bC(\bq, \dot{\bq})\bs = \bu - \left(\bH(\bq)\ddot{\bq}_r + \bC(\bq, \dot{\bq})\dot{\bq}_r + \bg(\bq)\right).
\end{equation}
Observe that because the Lagrangian was linearly parameterized, the resulting dynamics inherent a linear parameterization. Defining the \textit{known} basis functions
\begin{align*}
    Y^{(P)}_{il}(\bq) &= \frac{\p\phi^l(\bq)}{\p q_i},\\
    Y^{(K)}_{il}(\bq, \dot{\bq}, \bq_d(t), \dot{\bq}_d(t)) &= \sum_{kj} \frac{1}{2}\left[\frac{\p M^l_{ij}(\bq)}{\p q_k} - \left(\frac{\p M^l_{kj}(\bq)}{\p q_i} - \frac{\p M^l_{ki}(\bq)}{\p q_j}\right)\right]\dot{q}_k\dot{q}_{r, j} + \sum_{j}M_{ij}^l(\bq)\ddot{q}_{r, j},
\end{align*}
we can write (\ref{eqn:lag_dyn_s}) as
\begin{equation*}
    \bH(\bq)\dot{\bs} + \bC(\bq, \dot{\bq})\bs = \bu - \bY^{(P)}(\bq)\ba^{(P)} - \bY^{(K)}(\bq, \dot{\bq}, \bq_d(t), \dot{\bq}_d(t))\ba^{(K)}.
\end{equation*}
For $\bK > 0$ a constant positive definite matrix and for parameter estimates $\hat{\ba}^{(P)}$ and $\hat{\ba}^{(K)}$, taking $\bu = -\bK\bs + \bY^{(p)}(\bq)\ha^{(P)} + \bY^{(K)}(\bq, \dot{\bq}, \bq_d(t), \dot{\bq}_d(t))\ha^{(K)}$ leads to
\begin{equation*}
    \bH(\bq)\dot{\bs} + \bC(\bq, \dot{\bq})\bs = -\bK\bs + \bY^{(P)}(\bq)\tilde{\ba}^{(P)} + \bY^{(K)}(\bq, \dot{\bq}, \bq_d(t), \dot{\bq}_d(t))\tilde{\ba}^{(K)}.
\end{equation*}
The proof in~\citet{slot_li_robot} can now be directly extended. For $\psi^{(K)}$, $\psi^{(P)}$ strongly convex functions and $\gamma_K > 0$, $\gamma_P > 0$ positive gains, the Lyapunov-like function
\begin{equation*}
    V(\bs, \bq, \ha^{(K)}, \ha^{(P)}) = \frac{1}{2}\bs^\T\bH(\bq)\bs + \frac{1}{\gamma_K}\bregdg{{\psi^{(K)}}}{\ba^{(K)}}{\ha^{(K)}} +  \frac{1}{\gamma_P}\bregdg{{\psi^{(P)}}}{\ba^{(P)}}{\ha^{(P)}},
\end{equation*}
shows stability of the adaptation laws
\begin{align*}
    \dot{\ha}^{(K)} &= -\gamma_K\nabla^2\psi^{(K)}\left(\ha^{(K)}\right)^{-1}\bY^{(K)}(\bq, \dot{\bq}, \bq_d(t), \dot{\bq}_d(t))^\T\bs,\\
    \dot{\ha}^{(P)} &= -\gamma_P\nabla^2\psi^{(P)}\left(\ha^{(P)}\right)^{-1}\bY^{(P)}(\bq)^\T\bs,
\end{align*}
after an application of Barbalat's Lemma (Lemma~\ref{lem:barbalat}) and using skew-symmetry of $\dot{\bH} - 2\bC$ to eliminate $\frac{1}{2}\bs^\T\dot{\bH}\bs$.
\subsection{Proof of Proposition~\ref{prop:adaptive_observer}}
\label{app2:adaptive_observer}
Consider the Lyapunov-like function
\begin{equation*}
    V(\tilde{\by}, \ha) = \frac{1}{2}\tilde{\by}^\T\bGam\tilde{\by} + \frac{1}{\gamma}\bregd{\ba}{\ha}.
\end{equation*}
Let $\frac{\partial \bY\ba}{\partial \by}(\by)$ denote the derivative of $\bY(\by)\ba$ with respect to its argument evaluated at $\by$, and $\frac{\partial \bg}{\partial \by}(\by_1, \by_2)$ denote the derivative of $\bg$ with respect to its first argument evaluated at $(\by_1, \by_2)$. Using that $\bg(\by, \by) = \mathbf{0}$ for all $\by \in \mathbb{R}^m$, this Lyapunov-like function has time derivative
\begin{align*}
    \dot{V}(\tilde{\by}, \ha, t) &= \tilde{\by}^\T\bGam\bC\left(\bY(\tilde{\by} + \by(t))\tilde{\ba} + \left[\bY(\tilde{\by} + \by(t)) - \bY(\by(t))\right]\ba + \bg(\tilde{\by} + \by(t), \by(t)) - \bg(\by(t), \by(t))\right)\\
    &\phantom{=} - \tilde{\by}^\T\bGam\bC\bY(\tilde{\by} + \by(t))\tilde{\ba},\\
    &= \tilde{\by}^\T\bGam\bC\left(\left[\bY(\tilde{\by} + \by(t)) - \bY(\by(t))\right]\ba + \bg(\tilde{\by} + \by(t), \by(t)) - \bg(\by(t), \by(t))\right),\\
    &= \tilde{\by}^\T\left(\bGam\bC\int_0^1\left(\frac{\p\bY\ba}{\p \by}(\by(t) + s\tilde{\by}) + \frac{\p \bg}{\p \by}(\by(t) + s\tilde{\by}, \by(t))\right)ds\right)\tilde{\by},
\end{align*}
where the last equality follows by Hadamard's Lemma. As in the proof of Proposition~\ref{prop:dyn_predict_gen}, this shows that a sufficient condition for convergence of $\tilde{\by}(t)\rightarrow 0$ is
\begin{equation*}
    \bGam\bC\left(\frac{\p \bY\ba}{\p \by}(\by) + \frac{\p \bg}{\p \by}(\by, \bu(t)\right) + \left(\frac{\p \bY\ba}{\p \by}(\by) + \frac{\p \bg}{\p \by}(\by, \bu(t))\right)^\T\bC^\T\bGam \leq -\lambda\bGam
\end{equation*}
uniformly in $\by$ for some contraction rate $\lambda > 0$ and any external input $\bu(t)$. Under suitable observability assumptions on the system, convergence of $\hat{\by}(t)$ to $\by(t)$ ensures that $\hat{\bx}(t)$ converges to $\bx(t)$~\citep{luenberger}. 
\subsection{Proof of Proposition~\ref{prop:rec_net}}
\label{app2:rec_net}
The Lyapunov-like function
\begin{equation*}
    V(\widehat{\bThet}) = \frac{1}{\gamma}\bregd{\bThet}{\widehat{\bThet}},
\end{equation*}
has time derivative 
\begin{align*}
    \dot{V}(\widehat{\bThet}, t) &= -\sum_{ij}\widetilde{\Theta}_{ij}\left(\bsig\left(\widehat{\bThet}\bx(t)\right) - \bsig\left(\bThet\bx(t)\right)\right)_ix_j(t),\\
    &= -\sum_{ij}\widetilde{\Theta}_{ij}\left(\sigma_i\left(\sum_k\widehat{\Theta}_{ik}x_k(t)\right) - \sigma_i\left(\sum_k \Theta_{ik}x_k(t)\right)\right)x_j(t),\\
    &= -\sum_i \left(\sum_k \widetilde{\Theta}_{ik}x_k(t)\right)\left(\sigma_i\left(\sum_k\widehat{\Theta}_{ik}x_k(t)\right) - \sigma_i\left(\sum_k\Theta_{ik}x_k(t)\right)\right),\\
    &\leq -\sum_i \frac{1}{L_i}\left(\sigma_i\left(\sum_k\widehat{\Theta}_{ik}x_k(t)\right) - \sigma_i\left(\sum_k\Theta_{ik}x_k(t)\right)\right)^2,\\
    &\leq -\frac{1}{\max_k L_k}\left\Vert\bsig\left(\widehat{\bThet}\bx(t)\right) - \bsig\left(\bThet\bx(t)\right)\right\Vert_2^2 \leq 0.
\end{align*}
Integrating the above inequality shows that $\left[\bsig\left(\widehat{\bThet}\bx(t)\right) - \bsig\left(\bThet\bx(t)\right)\right]$ is an $\LL_2$ signal and hence that each component $\left[\sigma_i\left(\widehat{\bThet}\bx(t)\right) - \sigma_i\left(\bThet\bx(t)\right)\right]$ is also an $\LL_2$ signal. The error dynamics
\begin{equation*}
    \dot{\be} = -(k+1)\be + \bsig\left(\widehat{\bThet}\bx(t)\right) - \bsig\left(\bThet\bx(t)\right),
\end{equation*}
shows that each component $e_i$ is a low-pass filter of each component of the function approximation error $\left[\sigma_i\left(\widehat{\bThet}\bx(t)\right) - \sigma_i\left(\bThet\bx(t)\right)\right]$. Applying Lemma~\ref{lem:conv} shows that $\be(t)\rightarrow \mathbf{0}$.
\subsection{Proof of Proposition~\ref{prop:local_ho_sg}}
\label{app2:prop:local_ho_sg}
\begin{proof}
Consider the Lyapunov-like function
\begin{equation}
    V(\bx, \hv, \ha, t) = Q(\bx, t)  + \frac{1}{2\gamma}\tilde{\bv}^\T\tilde{\bv} + \frac{1}{2\gamma}\left(\hat{\ba} - \hat{\bv}\right)^\T\left(\hat{\ba} - \hat{\bv}\right).
    \label{eqn:ho_local_proof_v}
\end{equation}
Equation (\ref{eqn:ho_local_proof_v}) implies that, with $\mathcal{N}(t) = 1 + \mu N(t)$,
\begin{align}
    \dot{V}(\bx, \hv, \ha, t) &= \dot{Q}(\bx, \hat{\ba}, t) - \tilde{\ba}^\T\nabla_{\ha}\dot{Q}(\bx, \hat{\ba}, t) - \frac{\beta}{\gamma} \Vert\hat{\ba} - \hat{\bv}\Vert^2 - \frac{\beta\mu}{\gamma} N(t) \Vert\hat{\ba} - \hat{\bv}\Vert^2\nonumber\\
    &\phantom{=} + 2\left(\hat{\ba} - \hat{\bv}\right)^\T\nabla_{\ha}\dot{Q}(\bx, \hat{\ba}, t), \nonumber\\
    &\leq \dot{Q}(\bx, \ba, t) - \frac{\beta}{\gamma}\Vert\hat{\ba} - \hat{\bv}\Vert^2 - \frac{\beta\mu}{\gamma} N(t) \Vert\hat{\ba} - \hat{\bv}\Vert^2 + 2 \left(\hat{\ba} - \hat{\bv}\right)^\T\nabla_{\ha}\dot{Q}(\bx, \hat{\ba}, t),\nonumber\\
    &\leq -\rho(Q(\bx, t)) - \frac{\beta}{\gamma} \Vert\hat{\ba} - \hat{\bv}\Vert^2 \leq 0.
    \label{eqn:ho_local_proof_vdot}
\end{align}
By radial unboundedness of $Q(\bx, t)$ in $\bx$, (\ref{eqn:ho_local_proof_v}) \& (\ref{eqn:ho_local_proof_vdot}) show that $\bx$ remains bounded. Similarly, radial unboundedness of $V(\bx, \hv, \ha, t)$ in $\tilde{\bv}$ and $\hat{\ba} - \hat{\bv}$ show that $\hat{\bv}(t)$ and $\hat{\ba}(t)$ remain bounded. Integrating (\ref{eqn:ho_local_proof_vdot}) shows that $\frac{\beta}{\gamma}\int_0^\infty\Vert\hat{\ba}(t)-\hat{\bv}(t)\Vert^2dt \leq V(\bx(0), \bv(0), \ha(0), 0) < \infty$, so that $\left(\hat{\ba}(\cdot) - \hat{\bv}(\cdot)\right) \in \LL_2$. An identical argument shows that $\int_0^\infty \rho\left(Q(\bx(t), t)\right)dt < \infty$. Now, because $\bx$ and $\hat{\ba}$ are bounded, and because $\tilde{f}(\bx, \hat{\ba}, \ba, t)$ is locally bounded in $\bx$ and $\hat{\ba}$ uniformly in $t$ by assumption, writing $\bx(t) - \bx(s) = \int_s^t\left(\bff(\bx(t'), \ba, t') + \bu(\ha(t'), t')\right)dt'$ shows that $\bx(t)$ is uniformly continuous in $t$. Because $Q(\bx, t)$ is uniformly continuous in $t$ when $\bx$ is bounded, because $Q(\bx(t), t)$ is bounded for all $t$, and because $\rho$ is continuous in its argument, we conclude $\rho$ is uniformly continuous in $t$ and $\lim_{t\rightarrow \infty} \rho(t) = \lim_{t\rightarrow\infty} \rho(Q(\bx(t), t)) = 0$ by Lemma~\ref{lem:barbalat}. This shows that $\lim_{t\rightarrow \infty} Q(\bx(t), t) = 0$.
\end{proof}
\subsection{Proof of Proposition~\ref{prop:int_ho}}
\label{app2:prop:int_ho}
\begin{proof}
Consider the Lyapunov-like function
\begin{equation}
    V(\hv, \ha) = \frac{1}{2\gamma}\tilde{\bv}^\T\tilde{\bv} + \frac{1}{2\gamma}\left(\hat{\ba} - \hat{\bv}\right)^\T\left(\hat{\ba} - \hat{\bv}\right).
    \label{eqn:v_ho_int}
\end{equation}
Equation (\ref{eqn:v_ho_int}) implies that, with $\mathcal{N}(t) = 1 + \mu N(t)$,
\begin{align}
    \dot{V}(\bx, \hv, \ha, t) &= -\tilde{\ba}^\T\nabla_{\ha}R(\bx, \hat{\ba}, t) - \frac{\beta}{\gamma}\Vert\hat{\ba} - \hat{\bv}\Vert^2 - \frac{\beta\mu}{\gamma}N(t)\Vert\hat{\ba} - \hat{\bv}\Vert^2\nonumber\\
    &\phantom{=} + 2 \left(\hat{\ba} - \hat{\bv}\right)^\T\nabla_{\ha}R(\bx, \hat{\ba}, t),\nonumber\\
    &\leq R(\bx, \ba, t) - R(\bx, \hat{\ba}, t) - \frac{\beta}{\gamma}\Vert\hat{\ba} - \hat{\bv}\Vert^2 - \frac{\beta\mu}{\gamma}N(t)\Vert\hat{\ba} - \hat{\bv}\Vert^2\nonumber\\
    &\phantom{=} + 2 \left(\hat{\ba} - \hat{\bv}\right)^\T\nabla_{\ha}R(\bx, \hat{\ba}, t),\nonumber\\
    &\leq -kR(\bx, \hat{\ba}, t) - \frac{\beta}{\gamma}\Vert\hat{\ba} - \hat{\bv}\Vert^2 \leq 0.
    \label{eqn:vdot_ho_int}
\end{align}
(\ref{eqn:v_ho_int}) \& (\ref{eqn:vdot_ho_int}) show boundedness of $\hv$ and $\ha$ over $[0, T_x]$. Furthermore, integrating (\ref{eqn:vdot_ho_int}) shows that $\int_0^{T_x}\Vert\ha(t')-\hv(t')\Vert^2dt' < \infty$ and $\int_0^{T_x}R(\bx(t'), \ha(t'), t')dt' < \infty$. For any bounded solution $\bx(t)$, these integrals may be extended to infinity, and we conclude that $\left(\ha-\hv\right)\in\mathcal{L}_2$, $\ha\in\LL_{\infty}$, and $\hv\in\LL_{\infty}$. Writing $\bx(t) - \bx(s)$ in integral form as in the proof of Proposition~\ref{prop:local_ho_sg} shows that $\bx(t)$ is uniformly continuous in $t$, and in light of the local boundedness assumption on $\nabla_{\ha}R$, the same procedure can be applied to $\hv$ and $\ha$. Because $R(\bx(t), \ha(t), t)$ is uniformly continuous in $t$ for bounded $\bx$ and $\ha$, and because $\bx(t)$ and $\ha(t)$ are both uniformly continuous in $t$, we conclude that $R(\bx(t), \ha(t), t)$ is uniformly continuous in $t$ and $R(\bx(t), \ha(t), t)\rightarrow 0$ by Barbalat's Lemma (Lemma~\ref{lem:barbalat}).
\end{proof}
\subsection{Proof of Proposition~\ref{prop:comp_adapt}}
\label{app2:comp_adapt}
\begin{proof}
Consider the Lyapunov-like function
\begin{equation*}
    V(\bx, \ha, t) = \frac{1}{2}s(\bx, \bx_d(t))^2 + \frac{1}{2\gamma}\tilde{\ba}^\T\bP^{-1}\tilde{\ba}.
\end{equation*}
$V(\bx, \ha, t)$ has time derivative
\begin{equation*}
    \dot{V}(\bx, \ha, t) = -\eta s(\bx, \bx_d(t))^2 - \frac{\kappa}{\gamma}\tilde{f}(\bx, \ha, \ba, t)^2 \leq 0.
\end{equation*}
This immediately shows $s(\bx(\cdot), \bx_d(\cdot)) \in \LL_\infty$ and that $\hat{\ba}(\cdot) \in \LL_{\infty}$. Because $s(\bx(\cdot), \bx_d(\cdot)) \in \LL_\infty$, $\bx(\cdot) \in \LL_\infty$ by definition of the sliding variable. Integrating $\dot{V}(\bx, \ha, t)$ with respect to time shows that $s(\bx(\cdot), \bx_d(\cdot)) \in \LL_2$ and $\tilde{f}(\bx(\cdot), \ha(\cdot), \ba, \cdot) \in \LL_2$. The result then follows by application of Lemma~\ref{lem:conv} or directly by Barbalat's Lemma (Lemma~\ref{lem:barbalat}).
\end{proof}
\subsection{Proof of Proposition~\ref{prop:ho_comp}}
\label{app2:ho_comp}
\begin{proof}
Consider the Lyapunov function
\begin{equation*}
    V(\bx, \ha, \hv, t) = \frac{1}{2}s(\bx, \bx_d(t))^2 + \frac{1}{2\gamma}\left(\Vert\tilde{\bv}\Vert^2 + \Vert \hat{\ba} - \hat{\bv}\Vert^2\right).
\end{equation*}
$V(\bx, \ha, \hv, t)$ has time derivative
\begin{align}
    \dot{V}(\bx, \ha, \hv, t) &= - \eta s(\bx, \bx_d(t))^2 + s(\bx, \bx_d(t)) \tilde{f}(\bx, \ha, \ba, t) \nonumber\\
    &\phantom{=} + \frac{1}{\gamma}\tilde{\bv}^\T\left(-\kappa \tilde{f}(\bx, \ha, \ba, t) - \gamma s(\bx, \bx_d(t))\right)\bY(\bx, t)^\T \nonumber\\
    &\phantom{=} + \frac{1}{\gamma}\left(\hat{\ba} - \hat{\bv}\right)^\T\left(\beta \sN(t) \left(\hat{\bv} - \hat{\ba}\right) + \gamma s(\bx, \bx_d(t)) \bY(\bx, t)^\T + \kappa \tilde{f}(\bx, \ha, \ba, t)\bY(\bx, t)^\T\right)\nonumber\\
    &= -\eta s(\bx, \bx_d(t))^2 - \frac{\kappa}{\gamma}\tilde{f}(\bx, \ha, \ba, t)^2 - \frac{\beta}{\gamma}\Vert\hat{\ba}-\hat{\bv}\Vert^2 - \frac{\beta\mu}{\gamma}\Vert\hat{\ba} - \hat{\bv}\Vert\Vert\bY(\bx, t)\Vert^2 \nonumber\\
    &\phantom{=} + 2s(\bx, \bx_d(t))\left(\hat{\ba} - \hat{\bv}\right)^\T\bY(\bx, t)^\T + 2\frac{\kappa}{\gamma}\tilde{f}(\bx, \ha, \ba, t)\left(\hat{\ba} - \hat{\bv}\right)^\T\bY(\bx, t)^\T\nonumber\\
    &\leq -\eta s(\bx, \bx_d(t))^2 - \frac{\kappa}{\gamma}\tilde{f}(\bx, \ha, \ba, t)^2 - \frac{\beta}{\gamma}\Vert\hat{\ba}-\hat{\bv}\Vert^2 - \frac{\beta\mu}{\gamma}\Vert\hat{\ba} - \hat{\bv}\Vert\Vert\bY(\bx, t)\Vert^2 \nonumber\\
    &\phantom{=} + 2|s(\bx, \bx_d(t))|\Vert\hat{\ba} - \hat{\bv}\Vert\Vert\bY(\bx, t)\Vert + 2\frac{\kappa}{\gamma}|\tilde{f}(\bx, \ha, \ba, t)|\Vert\hat{\ba} - \hat{\bv}\Vert\Vert\bY(\bx, t)\Vert\nonumber\\
    &\leq -\epsilon_1\eta s(\bx, \bx_d(t))^2 - \epsilon_2\frac{\kappa}{\gamma}\tilde{f}(\bx, \ha, \ba, t)^2 - \frac{\beta}{\gamma}\Vert\hat{\ba}-\hat{\bv}\Vert^2 \nonumber\\
    &\phantom{=} - \left(\sqrt{(1-\epsilon_1)\eta}|s(\bx, \bx_d(t))| - \frac{1}{\sqrt{(1-\epsilon_1)\eta}}\Vert\hat{\ba} - \hat{\bv}\Vert\Vert\bY(\bx, t)\Vert\right)^2 \nonumber\\
    &\phantom{=} - \left(\sqrt{\frac{(1-\epsilon_2)\kappa}{\gamma}}|\tilde{f}(\bx, \ha, \ba, t)| - \frac{\kappa}{\gamma}\sqrt{\frac{\gamma}{(1-\epsilon_2)\kappa}}\Vert\hat{\ba} - \hat{\bv}\Vert\Vert\bY(\bx, t)\Vert\right)^2\nonumber\\
    &\leq 0.\nonumber
\end{align}
Above, $0 < \epsilon_1 < 1$ and  $0 < \epsilon_2 < 1$ are arbitrary and we have taken $\mu = \frac{\gamma}{\beta}\left(\frac{1}{(1-\epsilon_1)\eta} + \frac{\kappa}{(1-\epsilon_2)\gamma}\right)$. Because $\epsilon_1$ and $\epsilon_2$ are arbitrary, $\dot{V}(\bx, \ha, \hv, t)$ is negative semi-definite for $\mu > \frac{\gamma}{\beta}\left(\frac{1}{\eta} + \frac{\kappa}{\gamma}\right)$. Hence $\hat{\bv}(\cdot) \in \LL_\infty$, $\hat{\ba}(\cdot) \in \LL_\infty$, and $s(\bx(\cdot), \bx_d(\cdot)) \in \LL_\infty$. Because $s(\bx(\cdot), \bx_d(\cdot)) \in \LL_\infty$, we automatically have $\bx(\cdot) \in \LL_\infty$, which shows that $\dot{s}(\bx(\cdot), \bx_d(\cdot)) \in \LL_\infty$ by local boundedness of $\tilde{f}(\bx, \ha, \ba, t)$ in $\bx$ and $\hat{\ba}$ uniformly in $t$. Integrating $\dot{V}(\bx, \ha, \hv, t)$ with respect to time shows that $s(\bx(\cdot), \bx_d(\cdot), \dot{\bx}(\cdot), \dot{\bx}_d(\cdot)) \in \LL_2$ and hence by Lemma~\ref{lem:barbalat}, $s(\bx(t), \bx_d(t)) \rightarrow 0$ and $\bx(t) \rightarrow \bx_d(t)$.
\end{proof}

\subsection{Proof of Proposition~\ref{prop:ho_tyukin}}
\label{app2:ho_tyukin}
\begin{proof}
Consider the Lyapunov-like function
\begin{equation*}
V(\ha, \hv) = \frac{1}{2\gamma}\Vert\tilde{\bv}\Vert^2 + \frac{1}{2\gamma}\Vert\hat{\ba} - \hat{\bv}\Vert^2.
\end{equation*}
$V(\ha, \hv)$ has time derivative
\begin{align*}
    \dot{V}(\bx, \ha, \hv, t) &= \frac{1}{\gamma}\tilde{\bv}^\T\left(-\gamma\tilde{f}(\bx, \ha, \ba, t)\balf(\bx, t)\right) + \frac{1}{\gamma}\left(\hat{\ba} -  \hat{\bv}\right)^\T\left(\beta\sN(t)\left(\hat{\bv} - \hat{\ba}\right) + \gamma\tilde{f}(\bx, \ha, \ba, t)\balf(\bx, t)\right)\\
    &= -\left(\tilde{\ba}^\T\balf(\bx, t)\right)\tilde{f}(\bx, \ha, \ba, t) - \frac{\beta}{\gamma} \sN(t) \Vert\hat{\ba} - \hat{\bv}\Vert^2 + 2\left(\hat{\ba} - \hat{\bv}\right)^\T\balf(\bx, t)\tilde{f}(\bx, \ha, \ba, t)\\
    &\leq -\frac{\tilde{f}(\bx, \ha, \ba, t)^2}{D_1} - \frac{\beta}{\gamma}\Vert\hat{\ba} - \hat{\bv}\Vert^2 - \frac{\beta\mu}{\gamma}\Vert\hat{\ba} - \hat{\bv}\Vert^2\Vert\balf(\bx, t)\Vert^2\\
    &\phantom{=} + 2\Vert\balf(\bx, t)\Vert\Vert\hat{\ba} - \hat{\bv}\Vert|\tilde{f}(\bx, \ha, \ba, t)|\\
    &\leq -\frac{\epsilon}{D_1}\tilde{f}(\bx, \ha, \ba, t)^2 - \beta\Vert\hat{\ba} - \hat{\bv}\Vert^2 \\
    &\phantom{=} - \left(\sqrt{\frac{1-\epsilon}{D_1}}|\tilde{f}(\bx, \ha, \ba, t)| - \sqrt{\frac{D_1}{1-\epsilon}}\Vert\balf(\bx, t)\Vert\Vert\hat{\ba} - \hat{\bv}\Vert\right)^2 \leq 0.
\end{align*}
Above, $0 < \epsilon < 1$ is arbitrary and we have chosen $\mu = \frac{\gamma D_1}{(1-\epsilon)\beta}$. Because $\epsilon$ is arbitrary, this shows that $\hat{\bv}(t)$ and $\hat{\ba}(t)$ remain bounded for $\mu > \frac{\gamma D_1}{\beta}$ over the maximal interval of existence of $\bx(t)$. By integrating $\dot{V}(\bx, \ha, \hv, t)$ with respect to time shows that $\tilde{f}(\bx(\cdot), \ha(\cdot), \ba, \cdot) \in \LL_2$ over this same interval. Observe now that these bounds are independent of the length of the interval, as they depend only on the initial value of $V(\ha, \hv)$. Application of Lemma~\ref{lem:conv} completes the proof.
\end{proof}

\subsection{Proof of Proposition~\ref{prop:sg_local_elastic}}
\label{app2:prop:sg_local_elastic}
\begin{proof}
Consider the Lyapunov-like function
\begin{equation*}
    V(\bx, \ha, \overline{\ba}, t) = Q(\bx, t) + \frac{1}{2}\left(\tilde{\ba}^\T\bP^{-1}\tilde{\ba} + \tilde{\overline{\ba}}^\T\bP^{-1}\tilde{\overline{\ba}}\right).
\end{equation*}
$V(\bx, \ha, \overline{\ba}, t)$ has time derivative
\begin{align*}
    \dot{V}(\bx, \ha, \overline{\ba}, t) &= \dot{Q}(\bx, \ha, t) - \tilde{\ba}^\T\nabla_{\ha}\dot{Q}(\bx, \ha, t) - k\left(\overline{\ba} - \hat{\ba}\right)^\T\bP^{-1}\left(\overline{\ba} - \hat{\ba}\right)\\
    &\leq \dot{Q}(\bx, \ba, t) - k\left(\overline{\ba} - \hat{\ba}\right)^\T\bP^{-1}\left(\overline{\ba} - \hat{\ba}\right)\\
    &\leq -\rho\left(\bQ(\bx, t)\right) - k\left(\overline{\ba} - \hat{\ba}\right)^\T\bP^{-1}\left(\overline{\ba} - \hat{\ba}\right)\\
    &\leq 0.
\end{align*}
Because $\dot{V}(\bx, \ha, \overline{\ba}, t) \leq 0$ and $V(\bx, \ha, \overline{\ba}, t) \geq 0$, $\ha(t)$ and $\overline{\ba}(t)$ remain bounded. Radial unboundedness of $Q(\bx, t)$ in $\bx$ ensures that $\bx(t)$ remains bounded. Integrating $\dot{V}(\bx, \ha, \overline{\ba}, t)$ with respect to time shows that $\left(\ha(\cdot) - \overline{\ba}(\cdot)\right) \in \LL_2$. The remainder of the argument follows the proof of Proposition~\ref{prop:local_ho_sg}.
\end{proof}

\subsection{Proof of Proposition~\ref{prop:sg_local_mom_elastic}}
\label{app2:prop:sg_local_mom_elastic}
\begin{proof}
Consider the Lyapunov-like function
\begin{equation*}
    V(\bx, \ha, \hv, \overline{\ba}, \overline{\bv}, t) = Q(\bx, t) + \frac{1}{2\gamma}\tilde{\bv}^\T\tilde{\bv} + \frac{1}{2\gamma}\Vert\hat{\ba} - \hat{\bv}\Vert^2 + \frac{1}{2}\Vert\hat{\ba} - \overline{\ba}\Vert^2 + \frac{1}{2\gamma}\tilde{\overline{\bv}}^\T\tilde{\overline{\bv}}.
\end{equation*}
Observe that we can write:
\begin{align*}
    \frac{d}{dt}\left(Q(\bx, t) + \frac{1}{2\gamma}\tilde{\bv}^\T\tilde{\bv}\right) &= \dot{Q}(\bx, \ha, t) - \tilde{\bv}^\T\nabla_{\ha}\dot{Q}(\bx, \ha, t) + \frac{k_v}{\gamma}\tilde{\bv}^\T\left(\overline{\bv} - \hv\right)\\
    &= \dot{Q}(\bx, \ha, t) - \tilde{\ba}^\T\nabla_{\ha}\dot{Q}(\bx, \ha, t) - \left(\hv - \ha\right)^\T\nabla_{\ha}\dot{Q}(\bx, \ha, t)+ \frac{k_v}{\gamma}\tilde{\bv}^\T\left(\overline{\bv} - \hv\right)\\
    &\leq \dot{Q}(\bx, \ba, t) - \left(\hv - \ha\right)^\T\nabla_{\ha}\dot{Q}(\bx, \ha, t) + \frac{k_v}{\gamma}\tilde{\bv}^\T\left(\overline{\bv} - \hv\right)
\end{align*}
Where the inequality follows by convexity of $\dot{Q}(\bx, \ha, t)$ in its second argument. Now,
\begin{align*}
    \frac{d}{dt} \frac{1}{2\gamma}\norm{\ha-\hv}^2 &= \frac{1}{\gamma}\left(\ha - \hv\right)^\T\left(\beta\sN(t)\left[\hv - \ha\right] + k_a\beta\sN(t)\left[\overline{\ba} - \ha\right] + \gamma\nabla_{\ha}\dot{Q}(\bx, \ha, t) - k_v\left[\overline{\bv}-\hv\right]\right)\\
    &= -\frac{\beta\sN(t)}{\gamma}\norm{\ha-\hv}^2 + \frac{k_a\beta\sN(t)}{\gamma}\left(\ha-\hv\right)^\T\left(\overline{\ba}-\ha\right) + \left(\ha-\hv\right)^\T\nabla_{\ha}\dot{Q}(\bx, \ha, t)\\
    &\phantom{=} - \frac{k_v}{\gamma}\left(\ha-\hv\right)^\T\left(\overline{\bv}-\hv\right)
\end{align*}
Summing the above, we find that
\begin{align*}
    \frac{d}{dt}\left(Q(\bx, t) + \frac{1}{2\gamma}\tilde{\bv}^\T\tilde{\bv} + \frac{1}{2\gamma}\norm{\ha-\hv}^2\right) &\leq \dot{Q}(\bx, \ba, t) + 2\left(\ha - \hv\right)^\T\nabla_{\ha}\dot{Q}(\bx, \ha, t) + \frac{k_v}{\gamma}\tilde{\bv}^\T\left(\overline{\bv} - \hv\right)\\
    &\phantom{=} - \frac{\beta\sN(t)}{\gamma}\norm{\ha-\hv}^2 + \frac{k_a\beta\sN(t)}{\gamma}\left(\ha-\hv\right)^\T\left(\overline{\ba}-\ha\right)\\
    &\phantom{=} - \frac{k_v}{\gamma}\left(\ha-\hv\right)^\T\left(\overline{\bv}-\hv\right)
\end{align*}
Now, take $\sN(t) = 1 + \mu N(t)$ and write that
\begin{align*}
    \dot{Q}(\bx, \ba, t) + 2\left(\ha - \hv\right)^\T\nabla_{\ha}\dot{Q}(\bx, \ha, t) - \frac{\beta\sN(t)}{\gamma}\norm{\ha-\hv}^2 \leq -\rho(Q(\bx, t)) - \frac{\beta}{\gamma}\norm{\ha-\hv}^2
\end{align*}
by Assumption~\ref{assmp:local_ho_cond}. Hence,
\begin{align*}
    \frac{d}{dt}\left(Q(\bx, t) + \frac{1}{2\gamma}\tilde{\bv}^\T\tilde{\bv} + \frac{1}{2\gamma}\norm{\ha-\hv}^2\right) &\leq -\rho(Q(\bx, t)) - \frac{\beta}{\gamma}\norm{\ha-\hv}^2 + \frac{k_v}{\gamma}\tilde{\bv}^\T\left(\overline{\bv} - \hv\right)\\
    &\phantom{=} - \frac{\beta\sN(t)}{\gamma}\norm{\ha-\hv}^2 + \frac{k_a\beta\sN(t)}{\gamma}\left(\ha-\hv\right)^\T\left(\overline{\ba}-\ha\right)\\
    &\phantom{=} - \frac{k_v}{\gamma}\left(\ha-\hv\right)^\T\left(\overline{\bv}-\hv\right)
\end{align*}
Consider
\begin{equation*}
    \frac{d}{dt}\frac{1}{2\gamma}\tilde{\overline{\bv}}^\T\tilde{\overline{\bv}} = \frac{k_v}{\gamma}\tilde{\overline{\bv}}^\T\left(\hv - \overline{\bv}\right)
\end{equation*}
and then note that
\begin{equation*}
    \frac{k_v}{\gamma}\tilde{\overline{\bv}}^\T\left(\hv - \overline{\bv}\right) + \frac{k_v}{\gamma}\tilde{\bv}^\T\left(\overline{\bv} - \hv\right) = -\frac{k_v}{\gamma}\norm{\overline{\bv}-\hv}^2.
\end{equation*}
The final term in the Lyapunov function satisfies
\begin{align*}
    \frac{d}{dt} \frac{1}{2\gamma}\norm{\ha-\overline{\ba}}^2 = \frac{\beta \sN(t)}{\gamma}\left(\ha - \overline{\ba}\right)^\T\left(\hv - \ha\right) - \frac{2k_a\beta\sN(t)}{\gamma}\norm{\overline{\ba} - \ha}^2
\end{align*}
Putting all of these together, $V(\bx, \ha, \hv, \overline{\ba}, \overline{\bv}, t)$ has a time derivative which satisfies
\begin{align*}
    \dot{V}(\bx, \ha, \hv, \overline{\ba}, \overline{\bv}, t)  &\leq -\rho(Q(\bx, t)) - \frac{\beta}{\gamma}\norm{\ha-\hv}^2 - \frac{k_v}{\gamma}\norm{\bar{\bv}-\hv}^2 - \frac{\beta\sN(t)}{\gamma}\norm{\ha-\hv}^2 \\
    &\phantom{=} + \frac{\beta\sN(t)}{\gamma}\left(k_a + 1\right)\left(\ha-\hv\right)^\T\left(\overline{\ba}-\ha\right) - \frac{k_v}{\gamma}\left(\ha-\hv\right)^\T\left(\overline{\bv}-\hv\right) - \frac{2k_a\beta\sN(t)}{\gamma}\norm{\overline{\ba} - \ha}^2\\
    &\leq -\rho(Q(\bx, t)) - \left(\frac{\beta}{\gamma} - \frac{k_v}{2\gamma}\right)\norm{\ha-\hv}^2 - \frac{k_v}{2\gamma}\norm{\bar{\bv}-\hv}^2\\
    &\phantom{=} - \frac{\beta\sN(t)}{\gamma}\left(1 - \frac{\epsilon(k_a+1)}{2}\right)\norm{\ha-\hv}^2 \\
    &\phantom{=} - \frac{\beta\sN(t)}{\gamma}\left(2k_a - \frac{k_a+1}{2\epsilon}\right)\norm{\overline{\ba} - \ha}^2
\end{align*}
Above, $\epsilon > 0$ is a small constant and the inequality follows by two applications of Young's inequality. Clearly, $V(\bx, \ha, \hv, \overline{\ba}, \overline{\bv}, t) \leq 0$ if $\beta - \frac{k_v}{2} \geq 0$, $1 - \frac{\epsilon(k_a+1)}{2} \geq 0$, and $2k_a - \frac{k_a+1}{2\epsilon} \geq 0$. The latter two inequalities lead to the requirement that $\frac{1}{4\epsilon - 1} \leq k_a \leq \frac{2}{\epsilon} - 1$. Taking the limit as $\epsilon \rightarrow 0$ shows that $k_a > 0$ can be arbitrary. The proof from here is identical to that of Proposition~\ref{prop:local_ho_sg}.
\end{proof}
\subsection{Proof of Proposition~\ref{prop:local_ho_sg_rie}}
\label{app2:prop:local_ho_sg_rie}
\begin{proof}
Consider the Lyapunov-like function
\begin{equation}
    V(\bx, \ha, \hv, t) = Q(\bx, t)  + \frac{1}{\gamma}\bregd{\ba}{\hv} + \frac{1}{2\gamma}\left(\hat{\ba} - \hat{\bv}\right)^\T\left(\hat{\ba} - \hat{\bv}\right).
    \label{eqn:ho_local_proof_v_rie}
\end{equation}
Equation (\ref{eqn:ho_local_proof_v_rie}) implies that, with $\mathcal{N}(t) = 1 + \mu N(t)$,
\begin{align}
    \dot{V}(\bx, \ha, \hv, t) &= \dot{Q}(\bx, \hat{\ba}, t) - \tilde{\ba}^\T\nabla_{\ha}\dot{Q}(\bx, \ha, t) - \frac{\beta}{\gamma} \Vert\hat{\ba} - \hat{\bv}\Vert^2 - \frac{\beta\mu}{\gamma} N(t) \Vert\hat{\ba} - \hat{\bv}\Vert^2 \nonumber\\
    &\phantom{=} + \left(\hat{\ba} - \hat{\bv}\right)^\T\left(\bI + \left[\nabla^2\psi(\hv)\right]^{-1}\right)\nabla_{\ha}\dot{Q}(\bx, \ha, t), \nonumber\\
    &\leq \dot{Q}(\bx, \ba, t) - \frac{\beta}{\gamma}\Vert\hat{\ba} - \hat{\bv}\Vert^2 - \frac{\beta\mu}{\gamma} N(t) \Vert\hat{\ba} - \hat{\bv}\Vert^2 \nonumber\\
    &\phantom{=} + \left(\hat{\ba} - \hat{\bv}\right)^\T\left(\bI + \left[\nabla^2\psi(\hv)\right]^{-1}\right)\nabla_{\ha}\dot{Q}(\bx, \ha, t),\nonumber\\
    &\leq -\rho(Q(\bx, t)) - \frac{\beta}{\gamma} \Vert\hat{\ba} - \hat{\bv}\Vert^2.
    \label{eqn:ho_local_proof_vdot_rie}
\end{align}
The first line to the second follows by convexity of $\dot{Q}(\bx, \ha, t)$ in its second argument, while the second line to the third follows by Assumption \ref{assmp:local_ho_cond_rie}. The remainder of the proof is identical to Proposition~\ref{prop:local_ho_sg}.
\end{proof}
\subsection{Proof of Proposition~\ref{prop:int_ho_rie}}
\label{app2:prop:int_ho_rie}
\begin{proof}
Consider the Lyapunov-like function
\begin{equation}
    V(\ha, \hv) = \frac{1}{\gamma}\bregd{\ba}{\hv} + \frac{1}{2\gamma}\left(\hat{\ba} - \hat{\bv}\right)^\T\left(\hat{\ba} - \hat{\bv}\right).
    \label{eqn:v_ho_int_rie}
\end{equation}
Equation (\ref{eqn:v_ho_int_rie}) implies that, with $\mathcal{N}(t) = 1 + \mu N(t)$,
\begin{align}
    \dot{V}(\bx, \ha, \hv, t) &= -\tilde{\ba}^\T\nabla_{\ha}R(\bx, \ha, t) - \frac{\beta}{\gamma}\Vert\hat{\ba} - \hat{\bv}\Vert^2 - \frac{\beta\mu}{\gamma}N(t)\Vert\hat{\ba} - \hat{\bv}\Vert^2 \nonumber\\
    &\phantom{=} + \left(\hat{\ba} - \hat{\bv}\right)^\T\left(\bI + \left[\nabla^2\psi(\hv)\right]^{-1}\right)\nabla_{\ha}R(\bx, \ha, t),\nonumber\\
    &\leq R(\bx, \ba, t) - R(\bx, \hat{\ba}, t) - \frac{\beta}{\gamma}\Vert\hat{\ba} - \hat{\bv}\Vert^2 - \frac{\beta\mu}{\gamma}N(t)\Vert\hat{\ba} - \hat{\bv}\Vert^2 \nonumber\\
    &\phantom{=} +  \left(\hat{\ba} - \hat{\bv}\right)^\T\left(\bI + \left[\nabla^2\psi(\hv)\right]^{-1}\right)\nabla_{\ha}R(\bx, \ha, t),\nonumber\\
    &\leq -kR(\bx, \hat{\ba}, t) - \frac{\beta}{\gamma}\Vert\hat{\ba} - \hat{\bv}\Vert^2.
    \label{eqn:vdot_ho_int_rie}
\end{align}
The first line to the second follows by convexity of $R(\bx, \ha, t)$ in its second argument, while the second to the third follows by Assumption \ref{assmp:int_ho_cond_rie}. The remainder of the proof is identical to Proposition~\ref{prop:int_ho}.
\end{proof}

\section{Further results on dynamics prediction for Hamiltonian systems}
\label{app2:ham}
We now provide some extensions to the results in Section~\ref{ssec:ham} by exploiting the structure of separable Hamiltonians. With a separable Hamiltonian, it is natural to estimate the kinetic and potential energies separately,
\begin{align*}
    T(\hp) &= \bY_p(\hp)\ha_p,\\
    U(\hq) &= \bY_q(\hq)\ha_q,
\end{align*}
where $\bY_p$ and $\bY_q$ are row vectors of basis functions for the kinetic and potential energies respectively. In this case, following the same derivation as in Section~\ref{ssec:dyn_predict}, the error dynamics become
\begin{align*}
    \dot{\tilde{\bp}} &= -\nabla_{\hq}\bY_q(\hq)\tilde{\ba}_{q} - k_p \tilde{\bp} - \left(\nabla_{\hq}U(\hq) - \nabla_{\bq}U(\bq)\right),\\
    \dot{\tilde{\bq}} &= \nabla_{\hp}\bY_p(\hp)\tilde{\ba}_{p} - k_q\tilde{\bq} + \left(\nabla_{\hp}T(\hp) - \nabla_{\bp}T(\bp)\right).
\end{align*}
Consider the adaptation laws
\begin{align*}
    \dot{\ha}_p &= -\gamma_p\left[\nabla^2\psi_{p}(\ha_p)\right]^{-1}\left(\nabla_{\hp}\bY_p(\hp)\right)^\T\tilde{\bq},\\
    \dot{\ha}_q &= \gamma_q\left[\nabla^2\psi_{q}(\ha_q)\right]^{-1}\left(\nabla_{\hq}\bY_q(\hq)\right)^\T\tilde{\bp},
\end{align*}
where $\psi_p(\cdot)$ and $\psi_q(\cdot)$ are strictly convex functions, and where $\gamma_p > 0$ and $\gamma_q > 0$ are positive learning rates. The Lyapunov-like function 
\begin{equation}
    \label{eqn:app_lyap_sep}
    V = \frac{1}{2}\tilde{\bp}^\T\tilde{\bp} + \frac{1}{2}\tilde{\bq}^\T\tilde{\bq} + \frac{1}{\gamma_p}\bregdg{{\psi_p}}{\ba_p}{\ha_p} + \frac{1}{\gamma_q}\bregdg{{\psi_q}}{\ba_q}{\ha_q}
\end{equation} 
shows that a sufficient condition for convergence $\tilde{\bp} \rightarrow 0$ and $\tilde{\bq}\rightarrow 0$ is for the Jacobian
\begin{equation*}
    \bJ = \begin{pmatrix} -k_p\bI & -\nabla_{\bq}^2U(\bq) \\ \nabla_{\bp}^2T(\bp) & -k_q \bI \end{pmatrix}
\end{equation*}
to be uniformly negative definite. A sufficient condition for uniform negative definiteness is given by (\ref{eqn:contr_sep}). 

While separable Hamiltonians encompass many physical systems, some, such as robotic systems, do not have this structure. A more general form encompassing robotic systems is
\begin{equation*}
    \HH(\bp, \bq) = T(\bp, \bq) + U(\bq).
\end{equation*}
Parameterizing these terms independently,
\begin{align*}
    T(\hp, \hq) &= \bY_p(\hp, \hq)\ha_p,\\
    U(\hq) &= \bY_q(\hq)\ha_q,
\end{align*}
the error dynamics becomes
\begin{align*}
    \dot{\tilde{\bp}} &= -\left(\nabla_{\hq}\bY_p(\hp, \hq)\right)\tilde{\ba}_{p} -\left(\nabla_{\hq}\bY_q(\hq)\right)\tilde{\ba}_{q} - k_p \tilde{\bp} - \left(\nabla_{\hq}U(\hq) - \nabla_{\bq}U(\bq)\right)- \left(\nabla_{\hq}T(\hp, \hq) - \nabla_{\bq}T(\bp, \bq)\right),\\
    \dot{\tilde{\bq}} &= \left(\nabla_{\hp}\bY_p(\hp, \hq)\right)\tilde{\ba}_{p} - k_q\tilde{\bq} + \left(\nabla_{\hp}T(\hp, \hq) - \nabla_{\bp}T(\bp, \bq)\right).
\end{align*}
Now consider the adaptation laws
\begin{align*}
    \dot{\ha}_p &= \gamma_p\left[\nabla^2\psi_{p}(\ha_p)\right]^{-1}\left(\left(\nabla_{\hq}\bY_p(\hp, \hq)\right)^\T\tilde{\bp} - \left(\nabla_{\hp}\bY_p(\hp, \hq)\right)^\T\tilde{\bq}\right),\\
    \dot{\ha}_q &= \gamma_q\left[\nabla^2\psi_{q}(\ha_q)\right]^{-1}\left(\nabla_{\hq}\bY_q(\hq)\right)^\T\tilde{\bp},
\end{align*}
again where $\psi_p(\cdot)$ and $\psi_q(\cdot)$ are strictly convex functions, and where $\gamma_p > 0$ and $\gamma_q > 0$ are positive learning rates. The Lyapunov-like function (\ref{eqn:app_lyap_sep}) shows that a sufficient condition for convergence is for the Jacobian matrix
\begin{equation*}
    \bJ = \begin{pmatrix} -k_p\bI - \nabla_{\bp}\nabla_{\bq}T(\bp, \bq) & -\nabla_{\bq}^2U(\bq) - \nabla_{\bq}^2 T(\bp, \bq)\\ \nabla_{\bp}^2T(\bp, \bq) & -k_q \bI + \nabla_{\bq}\nabla_{\bp}T(\bp, \bq) \end{pmatrix}
\end{equation*}
to be uniformly negative definite. Sufficient conditions for this are now given by
\begin{align*}
    k_p &> -\frac{1}{2}\lambda_{\min}\left(\nabla_{\hp}\nabla_{\hq}T(\hp, \hq) + \nabla_{\hq}\nabla_{\hp}T(\hp, \hq)\right),\\
    k_q &> \frac{1}{2}\lambda_{\max}\left(\nabla_{\hp}\nabla_{\hq}T(\hp, \hq) + \nabla_{\hq}\nabla_{\hp}T(\hp, \hq)\right),\\
    \lambda_p\lambda_q &> \frac{1}{4}\lambda_{\max}^2\left[\nabla^2_{\hp}T(\hp, \hq) - \nabla^2_{\hq}T(\hp, \hq) - \nabla^2_{\hq}U(\hq)\right],
\end{align*}
similar to the fully general case handled in Sec.~\ref{ssec:ham}. More general results can be obtained by using a non-Euclidean metric as a replacement for the momentum and position estimation error terms in (\ref{eqn:app_lyap_sep}).
\subsection{Implicit regularization for higher-order laws}
\label{app2:ho_imp_reg}
For simplicity, we only consider the linearly parameterized setting. The nonlinearly parameterized setting can be handled immediately.
\begin{prop}
\label{prop:implicit_reg_app}
Consider the velocity gradient algorithm with momentum
\begin{align*}
    \dot{\hv} &= -\left[\nabla^2\psi(\hv)\right]^{-1}\nabla_{\ha}\dot{Q}(\bx, \ha, t),\\
    \dot{\ha} &= \beta\sN(t)(\hv - \ha).
\end{align*}
Let $\psi:\mathbb{R}^p\rightarrow\mathbb{R}$ be a strongly convex function. Suppose that the unknown dynamics is linearly parameterized. Assume that $\ha(t) \rightarrow \ha_\infty \in \mathcal{A}$ where $\mathcal{A}$ is defined in \eqref{set:A}. Then $\ha_\infty = \arg\min_{\bthet \in \mathcal{A}}\bregd{\bthet}{\hv(0)}$.
In particular, if $\hv(0) = \arg\min_{\bthet\in\mathbb{R}^p}\psi(\bthet)$, then 
\begin{equation*}
    \ha_\infty = \arg\min_{\bthet \in \mathcal{A}} \psi(\bthet).
\end{equation*}
\end{prop}
\begin{proof}
First observe that if $\ha(t)\rightarrow\ha_{\infty}$, then $\hv(t) \rightarrow \ha_{\infty}$. To prove this fact, we will show that $\dot{\ha}(t) \rightarrow 0$. consider
\begin{equation*}
    \Vert\dot{\ha}\Vert_{\LL_2}^2 = \int_0^\infty \dot{\ha}(t)^\T \dot{\ha}(t) dt = \int_0^\infty \beta \sN(t)\Vert\ha(t) - \hv(t)\Vert^2dt < \beta \sup_{t}\sN(t) \Vert\ha - \hv\Vert_{\LL_2}^2.
\end{equation*}
Note that $\sup_t \sN(t) < \infty$ because $\bx(t)$ remains bounded and because $\bY(\bx, t)$ is locally bounded in $\bx$ uniformly in $t$. Further note that $\Vert\ha - \hv\Vert_{\LL_2} < \infty$ by Proposition~\ref{prop:local_ho_sg_rie}. Hence we conclude that $\Vert \dot{\ha}\Vert_{\LL_2} < \infty$. Moreover, $\Vert \dot{\ha} \Vert_{\LL_\infty} < \infty$ and $\Vert \ddot{\ha} \Vert_{\LL_\infty} < \infty$ by boundedness of $\hv$ and boundedness of $\bY(\bx(t), t)$. By Lemma~\ref{lem:barbalat}, then $\dot{\ha}(t) \rightarrow 0$. Because $\hv(t) = \ha(t) + \frac{\dot{\ha}(t)}{\beta \sN(t)}$, this implies that $\hv(t) \rightarrow \ha_{\infty}$.

Now let $\bthet$ be any constant vector of parameters. The Bregman divergence has time derivative
\begin{equation*}
    \frac{d}{dt}\bregd{\bthet}{\hv} = -\left(\frac{d}{dt}\nabla\psi(\hv)\right)^\T\left(\bthet - \hv\right).
\end{equation*}
Using that $\frac{d}{dt}\nabla\psi(\hv) = -\nabla_{\ha}\dot{Q}(\bx, \ha, t) = -\bY(\bx, t)^\T\frac{\partial Q}{\partial \bx}(\bx, t)$ and integrating both sides of the above shows
\begin{equation*}
    \bregd{\bthet}{\hv(0)} = \bregd{\bthet}{\ha_{\infty}} + \int_0^\infty \frac{\partial Q}{\partial \bx}(\bx(\tau), \tau)^\T\bY(\bx(\tau), \tau)\left(\hv(\tau) - \bthet\right)d\tau.
\end{equation*}
Taking $\bthet \in \mathcal{A}$, the integral becomes independent of $\bthet$. The proof from here is identical to the first-order case in Appendix~\ref{app2:prop:implicit_reg}.
\end{proof}

\subsection{Proof of Lemma~\ref{lemma:alpha_tyuk}}
\label{app2:alpha_tyuk}
\begin{proof}
Defining the vector $\hat{\bv}^t = \sum_{i=1}^m\hat{w}_i^t\bphi(\bx_i)$, (\ref{eqn:alphatron_weights}) implies the iteration on $\hat{\bv}$,
\begin{equation}
    \hat{\bv}^{t+1} = \hat{\bv}^t - \frac{\lambda}{m}\sum_{i=1}^m\left(\hat{f}(\hat{\bw}^t, \bx_i) - f(\bx_i)\right)\bphi(\bx_i).
    \label{eqn:alpha_v_iter}
\end{equation}
(\ref{eqn:alpha_v_iter}) shows that at time $t$,
\begin{equation}
    \hat{\bv}^t = -\frac{\lambda}{m}\sum_{i=1}^m\left(\sum_{j=1}^{t-1}\tilde{f}_i^j(\hat{\bw}^t, \bx_i)\right)\bphi(\bx_i),
    \label{eqn:alpha_v}
\end{equation}
where $\tilde{f}_i^j(\hat{\bw}^t, \bx_i)$ in (\ref{eqn:alpha_v}) is the function approximation error on the $i^\text{th}$ input example at iteration $j$, $\tilde{f}_i^j(\hat{\bw}^t, \bx_i) = \hat{f}(\hat{\bw}^j, \bx_i) - f(\bx_i)$.

Now, assuming that for the adaptive control problem $f(\bx, \ba, t) = u(\balf^\T(\bx, t)\ba)$, setting $\bP = \lambda\bI$, $\hat{\ba}(0) = \mathbf{0}$, and integrating both sides of (\ref{eqn:tyukin_alg}), we see that at time $t$,
\begin{equation}
    \hat{\ba}(t) = -\lambda\int_0^t\tilde{f}(\bx(t'), \ha(t'), \ba, t')\balf(\bx(t'), t')dt'.
    \label{eqn:tyukin_params}
\end{equation}
The current function approximation $\hat{f}(t)$ at time $t$ for the parameters in (\ref{eqn:tyukin_params}) can then be written
\begin{align}
    \hat{f}(t) &= u\left(\balf^\T(\bx, t)\hat{\ba}(t)\right) = u\left(\int_0^t-\lambda\tilde{f}(\bx(t'), \hat{\ba}(t'), \ba, t')\balf^\T(\bx(t), t)\balf(\bx(t'), t')dt'\right)\nonumber\\
    &= u\left(\int_0^t c(t')\mathcal{K}(t, t')dt'\right)
    \label{eqn:tyukin_func}
\end{align}
where we have defined $c(t') = -\lambda\tilde{f}(\bx(t'), \hat{\ba}(t'), \ba, t')$ and $\mathcal{K}(t, t') = \balf^\T(\bx(t), t)\balf(\bx(t'), t')$. Similarly, in the case of the Alphatron, the current approximation at iteration $t$ is given by
\begin{align}
    \hat{f}(\hat{\bw}^t, \bx) &= u\left(\left\langle \hat{\bv}^t, \bphi(\bx)\right\rangle_{\mathcal{H}}\right) = u\left(\sum_{i=1}^m\left(\sum_{j=1}^{t-1}-\frac{\lambda}{m}\tilde{f}_i^j(\hat{\bw}^t, \bx_i)\right)\left\langle\bphi(\bx), \bphi(\bx_i)\right\rangle\right)\nonumber\\
    &= u\left(\sum_{i=1}^m\hat{w}_i^t\mathcal{K}(\bx, \bx_i)\right),
    \label{eqn:alphatron_func}
\end{align}
where we have noted that with $\hat{w}^0_i = 0$ for all $i$, $\hat{w}^t_i = \sum_{j=1}^{t-1}-\frac{\lambda}{m}\tilde{f}_i^j(\hat{\bw}^t, \bx_i)$.
\end{proof}

\subsection{Proof of Proposition~\ref{prop:bgf_1}}
\label{app2:bgf_1}
\begin{proof}
Consider the Lyapunov-like function
\begin{equation*}
    V(\bx, \ha, t) = \frac{1}{2}s(\bx, \bx_d(t))^2 + \frac{1}{2}\tilde{\ba}^\T\bP^{-1}\tilde{\ba}.
\end{equation*}
$V(\bx, \ha, t)$ has time derivative
\begin{equation*}
    \dot{V}(\bx, \ha, t) = -\eta s(\bx, \bx_d(t))^2 - \frac{1}{2}\tilde{f}(\bx, \ha, \ba, t)^2 - \frac{\lambda}{2}\tilde{\ba}^\T\bP^{-1}\tilde{\ba} \leq 0.
\end{equation*}
Negative semi-definiteness of $\dot{V}(\bx, \ha, t)$ shows that $s(\bx(t), \bx_d(t))$ and $\hat{\ba}(t)$ remain bounded. Because $s(\bx(t), \bx_d(t))$ remains bounded, $\bx(t)$ remains bounded. Integrating $\dot{V}(\bx, \ha, t)$ with respect to time shows that $s(\bx(\cdot), \bx_d(\cdot)) \in \LL_2$ and $\tilde{f}(\bx(\cdot), \ha(\cdot), \ba, \cdot) \in \LL_2$. The proof is completed by application of Lemma~\ref{lem:conv} or Lemma~\ref{lem:barbalat}.
\end{proof}

\subsection{Proof of Proposition~\ref{prop:bgf_2}}
\label{app2:bgf_2}
\begin{proof}
Consider the Lyapunov-like function
\begin{equation*}
    V(\bx, \ha, \hv, t) = \frac{1}{2}s(\bx, \bx_d(t))^2 + \frac{1}{2}\tilde{\bv}^\T\bP(t)^{-1}\tilde{\bv} + \frac{1}{2}\left(\hat{\bv} - \hat{\ba}\right)^\T\bP(t)^{-1}\left(\hat{\bv} - \hat{\ba}\right).
\end{equation*}
$V(\bx, \ha, \hv, t)$ has time derivative
\begin{align*}
    \dot{V}(\bx, \ha, \hv, t) &= -\eta s(\bx, \bx_d(t))^2 + s(\bx, \bx_d(t)) \tilde{f}(\bx, \ha, \ba, t)\\
    &\phantom{=} - \left(\hat{\bv} - \hat{\ba} + \tilde{\ba}\right)^\T\left(s(\bx, \bx_d(t)) + \tilde{f}(\bx, \ha, \ba, t)\right)\bY(\bx, t)^\T + \frac{1}{2}\left(\tilde{\bv}^\T\bY(\bx, t)^\T\right)^2\\
    &\phantom{=} - \frac{\lambda(t)}{2}\tilde{\bv}^\T\bP(t)^{-1}\tilde{\bv} + \left(\hat{\bv} - \hat{\ba}\right)^\T\left(-\beta \sN(t) \left(\hat{\bv} - \hat{\ba}\right) - \left(s(\bx, \bx_d(t)) +\tilde{f}(\bx, \ha, \ba, t)\right)\bY(\bx, t)^\T\right)\\
    &\phantom{=} + \frac{1}{2}\left[\left(\hat{\bv} - \hat{\ba}\right)^\T\bY(\bx, t)^\T\right]^2 - \frac{\lambda(t)}{2}\left(\hat{\bv} - \hat{\ba}\right)^\T\bP(t)^{-1}\left(\hat{\bv} - \hat{\ba}\right)\\\
    &= -\eta s(\bx, \bx_d(t))^2 - \tilde{f}(\bx, \ha, \ba, t)^2 - 2\left(\hat{\bv} - \hat{\ba}\right)^\T\left(s(\bx, \bx_d(t)) + \tilde{f}(\bx, \ha, \ba, t)\right)\bY(\bx, t)^\T\\
    &\phantom{=} - \beta \sN(t) \Vert\hat{\bv} - \hat{\ba}\Vert^2 + \frac{1}{2}\left(\tilde{\bv}^\T\bY(\bx, t)^\T\right)^2 + \frac{1}{2}\left[\left(\hat{\bv} - \hat{\ba}\right)^\T\bY(\bx, t)^\T\right]^2\\
    &\phantom{=} -\frac{\lambda(t)}{2}\left(\tilde{\bv}^\T\bP(t)^{-1}\tilde{\bv} + \left(\hat{\bv} - \hat{\ba}\right)^\T\bP(t)^{-1}\left(\hat{\bv} - \hat{\ba}\right)\right).
\end{align*}
Now we use that $$\tilde{\bv}^\T\bY(\bx, t)^\T = \left(\hat{\bv} - \hat{\ba}\right)^\T\bY(\bx, t)^\T + \tilde{f}(\bx, \ha, \ba, t)$$ to say that $$\frac{1}{2}\left(\tilde{\bv}^\T\bY(\bx, t)^\T\right)^2 = \frac{1}{2}\left[\left(\hat{\bv} - \hat{\ba}\right)^\T\bY(\bx, t)^\T\right]^2 + \left(\hat{\bv} - \hat{\ba}\right)^\T\bY(\bx, t)^\T\tilde{f}(\bx, \ha, \ba, t) + \frac{1}{2}\tilde{f}(\bx, \ha, \ba, t)^2.$$ Hence,
\begin{align*}
    \dot{V}(\bx, \ha, \hv, t) &= -\eta s(\bx, \bx_d(t))^2 - \frac{1}{2}\tilde{f}(\bx, \ha, \ba, t)^2 - 2s(\bx, \bx_d(t))\left(\hat{\bv} - \hat{\ba}\right)^\T\bY(\bx, t)^\T\\
    &\phantom{=} - \tilde{f}(\bx, \ha, \ba, t)\left(\hat{\bv} - \hat{\ba}\right)^\T\bY(\bx, t)^\T - \beta \sN(t) \Vert\hat{\bv} - \hat{\ba}\Vert^2 + \left[\left(\hat{\bv} - \hat{\ba}\right)^\T\bY(\bx, t)^\T\right]^2\\
     &\phantom{=} -\frac{\lambda(t)}{2}\left(\tilde{\bv}^\T\bP(t)^{-1}\tilde{\bv} + \left(\hat{\bv} - \hat{\ba}\right)^\T\bP(t)^{-1}\left(\hat{\bv} - \hat{\ba}\right)\right)\\
    &= -\eta s(\bx, \bx_d(t))^2 - \frac{1}{2}\tilde{f}(\bx, \ha, \ba, t)^2 - 2s(\bx, \bx_d(t))\left(\hat{\bv} - \hat{\ba}\right)^\T\bY(\bx, t)^\T \\
    &\phantom{=} - \tilde{f}(\bx, \ha, \ba, t)\left(\hat{\bv} - \hat{\ba}\right)^\T\bY(\bx, t)^\T - \beta \Vert\hat{\bv} - \hat{\ba}\Vert^2 - \beta \mu\Vert\bY(\bx, t)\Vert^2 \Vert\hat{\bv} - \hat{\ba}\Vert^2 \\
    &\phantom{=} + \left[\left(\hat{\bv} - \hat{\ba}\right)^\T\bY(\bx, t)^\T\right]^2 - \frac{\lambda(t)}{2}\left(\tilde{\bv}^\T\bP(t)^{-1}\tilde{\bv} + \left(\hat{\bv} - \hat{\ba}\right)^\T\bP(t)^{-1}\left(\hat{\bv} - \hat{\ba}\right)\right)\\
    &\leq -\eta s(\bx, \bx_d(t))^2 - \frac{1}{2}\tilde{f}(\bx, \ha, \ba, t)^2 + 2|s(\bx, \bx_d(t))|\Vert\left(\hat{\bv} - \hat{\ba}\right)\Vert\Vert\bY(\bx, t)^\T\Vert \\
    &\phantom{=} + |\tilde{f}(\bx, \ha, \ba, t)|\Vert\left(\hat{\bv} - \hat{\ba}\right)\Vert\Vert\bY(\bx, t)^\T\Vert - \beta \Vert\hat{\bv} - \hat{\ba}\Vert^2 - \left(\beta \mu - 1\right)\Vert\bY(\bx, t)\Vert^2 \Vert\hat{\bv} - \hat{\ba}\Vert^2\\
    &\phantom{=} -\frac{\lambda(t)}{2}\left(\tilde{\bv}^\T\bP(t)^{-1}\tilde{\bv} + \left(\hat{\bv} - \hat{\ba}\right)^\T\bP(t)^{-1}\left(\hat{\bv} - \hat{\ba}\right)\right)\\
    &\leq -\eta \epsilon_1 s(\bx, \bx_d(t))^2 - \frac{\epsilon_2}{2}\tilde{f}(\bx, \ha, \ba, t)^2 - \beta\Vert\hat{\bv} - \hat{\ba}\Vert^2 \\
    &\phantom{=} - \left(\sqrt{(1-\epsilon_1)\eta}|s(\bx, \bx_d(t))| - \frac{1}{\sqrt{(1-\epsilon_1)\eta}}\Vert\hat{\bv}-\hat{\ba}\Vert\Vert\bY(\bx, t)\Vert\right)^2\\
    &\phantom{=} - \left(\sqrt{\frac{1-\epsilon_2}{2}}|\tilde{f}(\bx, \ha, \ba, t)| - \frac{1}{2}\sqrt{\frac{2}{1-\epsilon_2}}\Vert\hat{\bv} - \hat{\ba}\Vert\Vert\bY(\bx, t)^\T\Vert\right)^2\\
    &\phantom{=} - \frac{\lambda(t)}{2}\left(\tilde{\bv}^\T\bP(t)^{-1}\tilde{\bv} + \left(\hat{\bv} - \hat{\ba}\right)^\T\bP(t)^{-1}\left(\hat{\bv} - \hat{\ba}\right)\right)\\
    &\leq 0
\end{align*}
Above, $0 < \epsilon_1 < 1$ and $0 < \epsilon_2 < 1$ are both arbitrary and we have chosen $\mu = \frac{1}{\beta}\left(1 + \frac{1}{\eta(1-\epsilon_1)} + \frac{1}{2(1-\epsilon_2)}\right)$. Negative semi-definiteness of $\dot{V}(\bx, \ha, \hv, t)$ shows that $s(\bx(t), \bx_d(t))$, $\hat{\bv}(t)$, and $\hat{\ba}(t)$ remain bounded. Because $s(\bx(t), \bx_d(t))$ remains bounded, $\bx(t)$ remains bounded. Integrating $\dot{V}(\bx, \ha, \hv, t)$ with respect to time shows that $s(\bx(\cdot), \bx_d(\cdot)) \in \LL_2$ and $\tilde{f}(\bx(\cdot), \ha(\cdot), \ba, \cdot) \in \LL_2$. By local boundedness of $\tilde{f}(\bx(\cdot), \ha(\cdot), \ba, \cdot)$ in $\bx$ and $\hat{\ba}$ uniformly in $t$, $\tilde{f}(\bx(t), \ha(t), \ba, t)$ remains bounded and hence $\dot{s}(\bx(t), \bx_d(t))$ remains bounded. By Lemma~\ref{lem:barbalat}, $s(\bx(t), \bx_d(t)) \rightarrow 0$ and $\bx(t) \rightarrow \bx_d(t)$.
\end{proof}

\subsection{Proof of Proposition~\ref{prop:bgf_3}}
\label{app2:bgf_3}
\begin{proof}
Consider the Lyapunov-like function
\begin{equation}
    V(\ha, t) = \frac{1}{2}\tilde{\ba}^\T\bP(t)^{-1}\tilde{\ba}.
\end{equation}
$V(\ha, t)$ has time derivative
\begin{align*}
\dot{V}(\bx, \ha, t) &= -\tilde{f}(\bx, \ha, \ba, t)\balf(\bx, t)^\T\tilde{\ba} + \frac{1}{2}\left(\tilde{\ba}^\T\balf(\bx, t)\right)^2 - \frac{\lambda}{2}\tilde{\ba}^\T\bP(t)^{-1}\tilde{\ba},\\
& \leq -\frac{1}{D_1}\tilde{f}(\bx, \ha, \ba, t)^2 + \frac{1}{2D_2^2}\tilde{f}(\bx, \ha, \ba, t)^2 = -\left(\frac{1}{D_1} - \frac{1}{2D_2^2}\right)\tilde{f}(\bx, \ha, \ba, t)^2.
\end{align*}
For $D_1 < 2D_2^2$, $\dot{V}(\bx, \ha, t) \leq 0$. Integrating $\dot{V}(\bx, \ha, t)$ with respect to time then shows that $\tilde{f}(\bx(\cdot), \ha(\cdot), \ba(\cdot), \cdot) \in \LL_2$ over the maximal interval of existence of $\bx(t)$. Alternatively, using the same Lyapunov function,
\begin{align*}
    \dot{V}(\bx, \ha, t) &\leq -\left(\balf(\bx, t)^\T\tilde{\ba}\right)^2\left(D_2 - \frac{1}{2}\right).
\end{align*}
For $D_2 > \frac{1}{2}$, $\dot{V}(\bx, \ha, t) \leq 0$ and $\balf(\bx(\cdot), \cdot)^\T\tilde{\ba} \in \LL_2$ over the maximal interval of existence of $\bx(t)$. By Assumption \ref{assmp:tyukin}, this implies that $\tilde{f}(\bx(\cdot), \ha(\cdot), \ba, \cdot) \in \LL_2$ over this same interval. Hence, both approaches demonstrate that $\hat{\ba}$ remains bounded over the maximal interval of existence of $\bx(t)$, and that $\tilde{f}(\bx(\cdot), \ha(\cdot), \ba, \cdot) \in \LL_2$ over the same interval. Furthermore, these bounds are independent of the length of the interval, as they only depend on $V(\ha(0), 0)$. By Lemma \ref{lem:conv}, the proposition is proved.
\end{proof}

\subsection{Proof of Proposition~\ref{prop:bgf_4}}
\label{app2:bgf_4}
\begin{proof}
Consider the Lyapunov-like function
\begin{equation*}
    V(\ha, \hv, t) = \frac{1}{2}\left(\tilde{\bv}^\T\bP(t)^{-1}\tilde{\bv} + \left(\hat{\ba} - \hat{\bv}\right)^\T\bP(t)^{-1}\left(\hat{\ba} - \hat{\bv}\right)\right).
\end{equation*}
$V(\ha, \hv, t)$ has time derivative
\begin{align*}
    \dot{V}(\bx, \ha, \hv, t) &= -\tilde{\bv}^\T\balf(\bx, t)\tilde{f}(\bx, \ha, \ba, t) + \frac{1}{2}\tilde{\bv}^\T\left(-\lambda\bP(t)^{-1} + \balf(\bx, t)\balf(\bx, t)^\T\right)\tilde{\bv} \nonumber\\
    &\phantom{=} + \left(\hat{\ba} - \hat{\bv}\right)^\T\left(\beta \sN(t) \left(\hat{\bv} - \hat{\ba}\right) + \balf(\bx, t)\tilde{f}(\bx, \ha, \ba, t)\right) \\
    &\phantom{=} + \frac{1}{2}\left(\hat{\ba} - \hat{\bv}\right)^\T\left(-\lambda\bP(t)^{-1} + \balf(\bx, t)\balf(\bx, t)^\T\right)\left(\hat{\ba} - \hat{\bv}\right)\\
    &\leq -\left(\tilde{\ba}^\T\balf(\bx, t)\right)\tilde{f}(\bx, \ha, \ba, t) + \frac{1}{2}\left(\tilde{\bv}^\T\balf(\bx, t)\right)^2 + \frac{1}{2}\left(\left(\hat{\ba} - \hat{\bv}\right)^\T\balf(\bx, t)\right)^2 \\
    &\phantom{=} - \beta \sN(t) \Vert\hat{\ba} - \hat{\bv}\Vert^2 + 2 \tilde{f}(\bx, \ha, \ba, t)\left(\hat{\ba} - \hat{\bv}\right)^\T\balf(\bx, t)\\
    &\leq -D_2\left(\balf(\bx, t)^\T\tilde{\ba}\right)^2 + \frac{1}{2}\left(\tilde{\bv}^\T\balf(\bx, t)\right)^2 + \frac{1}{2}\left(\left(\hat{\ba} - \hat{\bv}\right)^\T\balf(\bx, t)\right)^2 \nonumber\\
    &\phantom{=} - \beta\sN(t) \Vert\hat{\ba} - \hat{\bv}\Vert^2 + 2|\tilde{f}(\bx, \ha, \ba, t)|\Vert\hat{\ba} - \hat{\bv}\Vert\Vert\balf(\bx, t)\Vert
\end{align*}
Observe that $$\frac{1}{2}\left(\tilde{\bv}^\T\balf(\bx, t)\right)^2 = \frac{1}{2}\left[\left(\hat{\bv} - \hat{\ba}\right)^\T\balf(\bx, t)\right]^2 + \left(\hat{\bv} - \hat{\ba}\right)^\T\balf(\bx, t)\left(\balf(\bx, t)^\T\tilde{\ba}\right) + \frac{1}{2}\left(\tilde{\ba}^\T\balf(\bx, t)\right)^2.$$ Then, rewrite
\begin{align*}
\dot{V}(\bx, \ha, \hv, t) &\leq -\left(D_2 - \frac{1}{2}\right)\left(\balf(\bx, t)^\T\tilde{\ba}\right)^2 + \left[\left(\hat{\ba} - \hat{\bv}\right)^\T\balf(\bx, t)\right]^2 \\
&\phantom{=} - \beta\sN(t)\Vert\hat{\ba} - \hat{\bv}\Vert^2 + \Vert\hat{\bv} - \hat{\ba}\Vert\Vert\balf(\bx, t)\Vert|\balf(\bx, t)^\T\tilde{\ba}|\left(2D_1 + 1\right)\\
&\leq -\left(D_2 - \frac{1}{2}\right)\left(\balf(\bx, t)^\T\tilde{\ba}\right)^2 - \beta\Vert\hat{\ba} - \hat{\bv}\Vert^2 - \left(\beta\mu-1\right)\Vert\hat{\ba} - \hat{\bv}\Vert^2\Vert\balf(\bx, t)\Vert^2 \\
&\phantom{=} + \Vert\hat{\bv} - \hat{\ba}\Vert\Vert\balf(\bx, t)\Vert|\balf(\bx, t)^\T\tilde{\ba}|\left(2D_1 + 1\right)\\
&\leq -\epsilon\left(D_2 - \frac{1}{2}\right)\left(\balf(\bx, t)^\T\tilde{\ba}\right)^2 - \beta\Vert\hat{\ba} - \hat{\bv}\Vert^2 \\
&\phantom{=} - \left(\sqrt{(1-\epsilon)\left(D_2 - \frac{1}{2}\right)}|\balf(\bx, t)^\T\tilde{\ba}| - \frac{2D_1+1}{2\sqrt{(1-\epsilon)\left(D_2 - \frac{1}{2}\right)}}\Vert\hat{\bv} - \hat{\ba}\Vert\Vert\balf(\bx, t)\Vert\right)^2.
\end{align*}
Above, $0 < \epsilon < 1$ is arbitrary and we have chosen $\mu = \frac{1}{\beta}\left(1 + \frac{\left(2D_1 + 1\right)^2}{(1-\epsilon)(4D_2 - 2)}\right)$. $\dot{V}(\bx, \ha, \hv, t)$ is clearly negative semi-definite for $D_2 < \frac{1}{2}$, which shows that $\hat{\bv}(t)$ and $\hat{\ba}(t)$ remain bounded over the maximal interval of existence of $\bx(t)$. Integrating $\dot{V}(\bx, \ha, \hv, t)$ with respect to time shows that $\balf(\bx(\cdot), \cdot)^\T\tilde{\ba}(\cdot) \in \LL_2$ over this interval, which implies that $\tilde{f}(\bx(\cdot), \ha(\cdot), \ba, \cdot) \in \LL_2$ over the same interval by Assumption \ref{assmp:tyukin}. Observe that the bounds are independent of the length of the interval, as they only depend on $V(\ha(0), \hv(0), 0)$. By Lemma \ref{lem:conv}, the proposition is proven.
\end{proof}

\end{document}